\documentclass[
a4paper, 
fontsize=12pt, 
twoside, 
openright, 
numbers=noenddot, 
]{scrreprt}

\usepackage[utf8]{inputenc}        
\usepackage[T1]{fontenc}             
\usepackage[english]{babel}          
\usepackage{lmodern}
\usepackage{csquotes} 
\usepackage[intlimits]{amsmath}      
\usepackage{amsfonts}                
\usepackage{amssymb}                 
\usepackage{amsthm}
\usepackage{array}                   
\usepackage{enumerate}               
\usepackage{latexsym}
\usepackage{graphicx}
\usepackage{multirow}
\usepackage{mathtools}
\usepackage[plainpages=false,pdfpagelabels,bookmarks=true]{hyperref}
\usepackage[a4paper,left=3cm,right=3cm,top=3cm,bottom=3.5cm,includeheadfoot, headsep=1cm, footskip=2cm,bindingoffset=0mm]{geometry}	
\usepackage{faktor}
\usepackage{paralist}
\usepackage{mathrsfs}
\usepackage{lastpage}
\usepackage{cleveref}
\usepackage{pdfpages}
\usepackage{blindtext}
\usepackage{tcolorbox}
\usepackage{enumitem}
\usepackage{tikz}
\usepackage{youngtab}
\usepackage{tkz-euclide}
\usepackage{tabularx}
\usepackage{longtable}
\usepackage{bbm}
\usepackage{MnSymbol}

\usetikzlibrary{cd, matrix,arrows,decorations.pathmorphing, babel, positioning, calc}

\setlist[enumerate, 1]{nosep, label=(\arabic*)}
\setlist[enumerate, 2]{nosep, label=(\roman*)}
\setlist[itemize, 1]{nosep, label=$\circ$} 
\setlist[itemize, 2]{nosep, label=-} 

\allowdisplaybreaks[3]				

\addtokomafont{disposition}{\rmfamily}

      \theoremstyle{plain}
      
      \newtheorem{theorem}{Theorem}
       \newtheorem{theoremanddefinition}[theorem]{Theorem and Definition}
      \newtheorem*{theorem*}{Theorem}
      \newtheorem{lemma}[theorem]{Lemma}
      \newtheorem*{lemma*}{Lemma}
      \newtheorem{cor}[theorem]{Corollary}
      \newtheorem*{cor*}{Corollary}
      \newtheorem{prop}[theorem]{Proposition}
      \newtheorem*{prop*}{Proposition}

      \theoremstyle{definition}
      
      \newtheorem{definition}[theorem]{Definition}
      \newtheorem*{definition*}{Definition}
      \newtheorem{notation}[theorem]{Notation}

 \newtheorem{exercise}{Exercise}    
\newtheorem*{exercise*}{Exercise}

      \theoremstyle{remark}
      
      \newtheorem{remark}[theorem]{Remark}
 \newtheorem{motivation}[theorem]{Motivation}
      \newtheorem*{remark*}{Remark}     
      \newtheorem{example}[theorem]{Example}
      \newtheorem*{example*}{Example}    
      \newtheorem*{question*}{Question}    
      \newtheorem*{fact*}{Fact}

\newcommand{\NN}{\mathbb{N}}
\newcommand{\N}{\mathbb{N}}

\newcommand{\RR}{\mathbb{R}}
\newcommand{\R}{\mathbb{R}}
\newcommand{\CC}{\mathbb{C}}
\newcommand{\C}{\mathbb{C}}

\newcommand{\cF}{\mathcal{F}}

\newcommand{\cA}{\mathcal{A}}
\newcommand{\cB}{\mathcal{B}}
\newcommand{\cC}{\mathcal{C}}
\newcommand{\cD}{\mathcal{D}}

\newcommand{\HH}{\mathcal{H}}

\newcommand{\cP}{\mathcal{P}}

\let\a\relax
\newcommand{\a}{\mathfrak{a}}

\newfont{\suet}{suet14}
\DeclareTextFontCommand{\scrpt}{\suet}
\renewcommand{\a}{\scrpt{a}}

\newcommand{\ee}{\varepsilon}

\newcommand{\GUE}{\hbox{\textsc{gue}}}

\newcommand{\ff}{\varphi}
\newcommand{\lb}{\langle}
\newcommand{\rb}{\rangle}

\newcommand{\la}{\langle}
\newcommand{\ra}{\rangle}

\newcommand{\kk}{\kappa}

\newcommand{\F}{\mathcal F}

\let\Re\relax
\let\Im\relax
\DeclareMathOperator{\Re}{Re}
\DeclareMathOperator{\Im}{Im}

\DeclareMathOperator{\supp}{supp}

\DeclareMathOperator{\tr}{tr}

\DeclareMathOperator{\id}{id}

\begin{document}


\includepdf{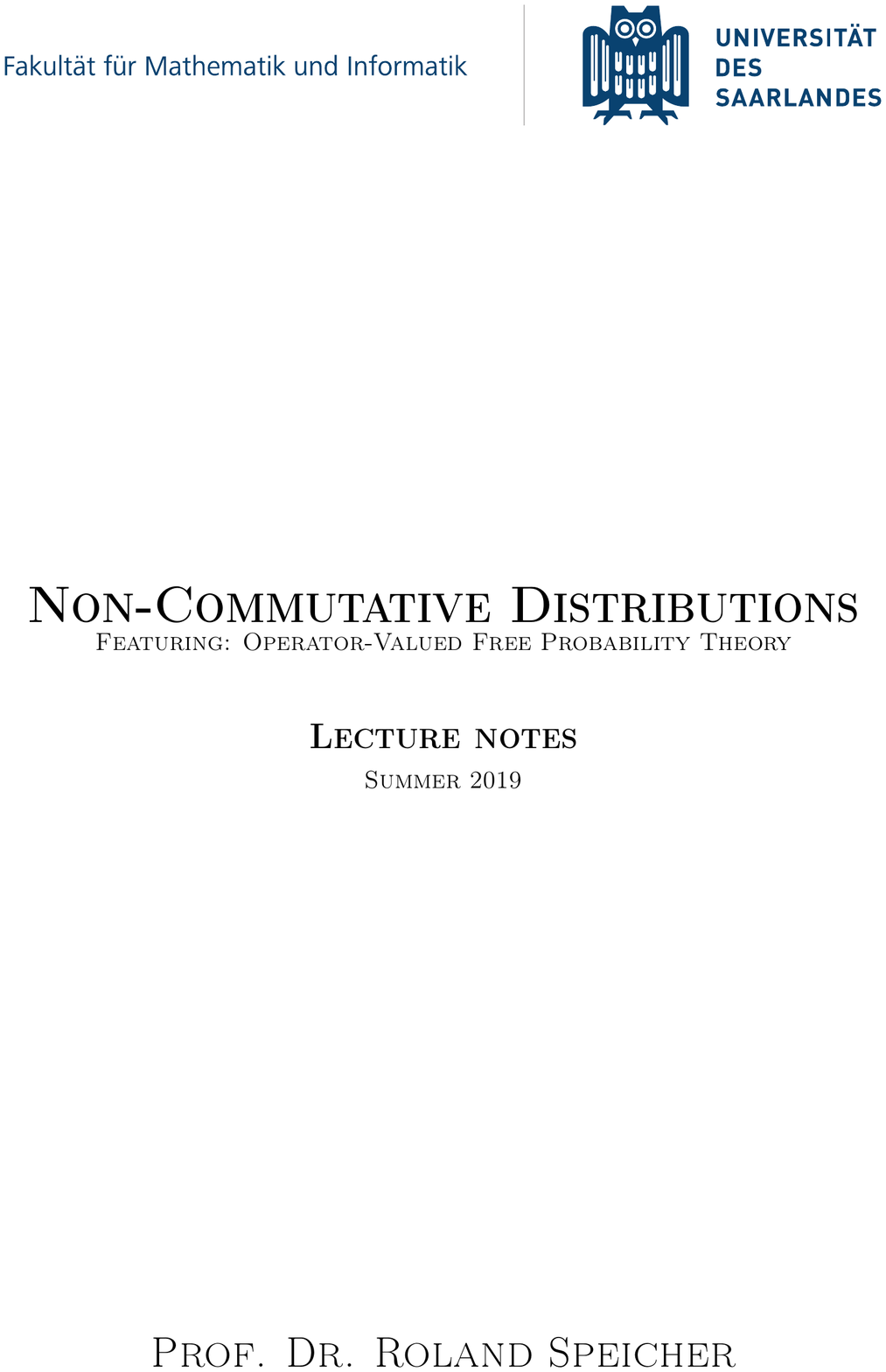}

This in an introduction to the theory of non-commutative distributions of non-commuting operators or random matrices. Starting from the basic problem to find a good approach to the meaning of “non-commutative distribution” we will, in particular, cover:  free analysis, which is a version of complex analysis for several non-commuting variables; the operator-valued version of free probability theory (combinatorial but also analytic aspects); the linearization trick to reduce non-linear scalar problems to linear operator-valued problems; the combination of operator-valued convolution and linearization to calculate the distribution of polynomials in free variables; the basic theory of non-commutative rational functions.
   
On one hand, this is a continuation of the \href{https://rolandspeicher.files.wordpress.com/2019/08/free-probability.pdf}{Free Probability Lecture Notes}. On the other hand, the theory of free probability is developed again, but in a more general, operator-valued context. So, in principle and with some additional efforts, it should be possible to read the present notes without having a prior knowledge on free probability. Big parts of the material do also not deal so much with free variables, but more general with analytic and algebraic aspects of maximal non-commuting variables. 

The material here was presented in the summer term 2019 at Saarland University in 20 lectures of 90 minutes each. The lectures were recorded and can be found online at \url{https://www.math.uni-sb.de/ag/speicher/web_video/index.html}.

Many of the presented results were actually achieved in recent years in the context of the ERC-grant ``\href{https://www.uni-saarland.de/lehrstuhl/speicher/research/erc-advanced-grant.html}{Non-Commutative Distributions in Free Probability}'' (2014-19).


\renewcommand{\contentsname}{Table of contents}
\tableofcontents

\numberwithin{theorem}{chapter}


\setcounter{chapter}{-1}
\chapter{Introduction}

We are interested in properties, preferably analytic, of distributions $\mu_{X_1,\dots,X_n}$ of 
\renewcommand{\labelitemi}{$\blacksquare$}
\begin{itemize}
\item[$\circ$]
operators $X_1,\dots,X_n$ on Hilbert spaces (typically from $C^*$-algebras or von Neumann algebras);
\item[$\circ$]
often those operators are ``limits in distribution'' of random matrix models;
\item[$\circ$]
typically our operators don't commute, which makes our distributions ``non-commutative''.
\end{itemize}

\section{Classical case}
Consider first the classical case of ``commutative'' distributions. Then random variables $X_1,\dots,X_n$ are measurable functions $X_i:\Omega\to\RR$ ($i=1,\dots,n$), where $(\Omega,\mathfrak{A},P)$ is a probability space, i.e., $P$ is a probability measure on the $\sigma$-algebra $\mathfrak{A}$ over $\Omega$, and the distribution $\mu_{X_1,\dots,X_n}$ is a probability measure on $\RR^n$, given as push-forward of $P$, i.e.,
$$\mu_{X_1,\dots,X_n}(B)=P(\{ \omega\in\Omega \mid (X_1(\omega),\dots,
X_n(\omega))\in B\}) \qquad\text{for any Borel set $B$ of $\RR^n$.}$$
There are various ways of describing or working with this object: $\mu_{X_1,\dots,X_n}$ is
\begin{enumerate}
\item[(i)]
a probability measure on $\RR^n$;
\item[(ii)]
a positive linear map, which allows to average over continuous functions of $X_1,\dots,X_n$:
\begin{align*}
E[f(X_1,\dots,X_n)]&=\int_{\RR^n}f(t_1,\dots,t_n)d\mu_{X_1,\dots,X_n}(t_1,\dots,t_n)\\
&=\int_\Omega f(X_1(\omega),\dots,X_n(\omega))dP(\omega),
\end{align*}
for continuous $f:\RR^n\to\CC$;
this is the same as (i) via the Riesz respresentation theorem;
\item[(iii)]
uniquely determined by its Fourier transform
$$\cF(t_1,\dots,t_n)=E[e^{-i(t_1X_1+\cdots +t_nX_n)}];$$
or other nice analytic functions on $\RR^n$ or $\CC^n$;
e.g., for $n=1$, one also has the Cauchy or Stieltjes transform $G(z)=E[(z-X)^{-1}]$;
\item[(iv)]
in many cases (e.g., compactly supported case) uniquely determined by its
moments
$$E[X_1^{k_1}\cdots X_n^{k_n}]\qquad\text{for all $k_1,\dots,k_n\in\NN_0$.}$$
\end{enumerate}

\section{Non-commutative case}
Consider now general (i.e., not necessarily commuting) $X_1,\dots,X_n\in\cA$ for a non-commutative probability space $(\cA,\ff)$, where $\cA$ is a unital algebra and $\ff:\cA\to\CC$ a unital linear functional (usually with some additional analytic structure). 
Can we give sense to $\mu_{X_1,\dots,X_n}$ in this setting?

The only item which makes directly sense is the combinatorial item (iv); and this will serve as our definition in the non-commutative case: \emph{$\mu_{X_1,\dots,X_n}$ is the collection of all moments}
$$\ff(X_{i(1)}\cdots X_{i(k)}) \qquad \text{for all $k\in\NN$; $1\leq i(1),\dots, i(k)\leq n$}$$
\emph{of our (non-commutative) random variables $X_1,\dots,X_n$.}

Our goal is an analytic understanding of this; i.e., to find non-commutative versions or replacements for items (i) - (iii). In particular, we would like to to
have notions for and results on
\begin{itemize}
\item
``smoothness'' or ``regularity'' of non-commutative distributions;
\item
absence of ``atoms'';
\item
existence of ``densities''.
\end{itemize}
We still don't know what a ``non-commutative probability measure'' is, but there has been quite some progress in recent years on dealing with this via versions of (ii) and (iii). In particular, we can say quite a bit about the distribution of
$f(X_1,\dots,X_n)$ for big classes of $(X_1,\dots,X_n)$ and big classes of $f$.
In particular, this gives results on the asymptotic eigenvalue distribution for polynomials in independent random matrices; or, equivalently, the distribution of polynomials in free variables. 

These results rely in particular on progress on
\begin{itemize}
\item
operator-valued versions of free probability theory of Voiculescu;
\item
free analysis (aka free non-commutative function theory);
\item
relating analytic questions about operators in von Neumann algebras with the theory  (of Cohn et al.) of non-commutative linear algebra and the free skew field (aka non-commutative rational functions).
\end{itemize}

All of this, and much more, will be covered in the coming chapters.

\chapter{Basic Definitions and Examples}

We start with the basic definitions and the most prominent example of 
a non-commutative distribution: namely free semicircular variables. They show up as the sum of creation and annihilation operators on
the full Fock space as well as the limit of our most beloved random
matrices, namely independent Gaussian random matrices.

\section{Non-commutative distributions and moments}

\begin{definition}\label{def:1-1}
\begin{enumerate}
\item
A \emph{non-commutative probability space} $(\cA,\ff)$ consists of 
\begin{itemize}
\item
a unital algebra $\cA$;
\item
a unital linear functional $\ff:\cA\to\CC$; unital means $\ff(1)=1$.
\end{itemize}
\item
A \emph{$C^*$-probability space} is a non-commutative probability space $(\cA,\ff)$, where
\begin{itemize}
\item
$\cA$ is a unital $C^*$-algebra;
\item
$\ff$ is a state, i.e., $\ff(A^*A)\geq 0$ for all $A\in\cA$.
\end{itemize}
\item
Elements $A_1,\dots,A_n\in (\cA,\ff)$ are called \emph{(non-commutative) random variables}.
\end{enumerate}

\end{definition}

\begin{remark}
By the GNS construction, a $C^*$-probability space can always be written as:
\begin{itemize}
\item
$\cA\subset B(\HH)$, for a Hilbert space $\HH$;
\item
$\ff(A)=\lb A\xi,\xi\rb$, for some unit vector $\xi\in\HH$.
\end{itemize}

\end{remark}

\begin{definition}\label{def:1-3}
Let $(\cA,\ff)$ be a non-commutative probability space and consider $A_1,\dots,A_n\in\cA$. The \emph{(non-commutative) distribution} $\mu_{A_1,\dots,A_n}$ of $A_1,\dots,A_n$ is given by the collection of all their joint moments:
$$\mu_{A_1,\dots,A_n}\quad\hat=\quad \{\ff(A_{i_1}\cdots A_{i_k})\mid k\in \NN;\ 
1\leq i_1,\dots,i_k\leq n\}.$$

\end{definition}

\section{The quest for an analytic understanding of non-commutative distributions}

\begin{remark}
We will usually work in a $C^*$-probability space and consider selfadjoint operators $X_1,\dots,X_n$. Our main goal is to get a better \emph{analytic} understanding of the distribution $\mu_{X_1,\dots,X_n}$. For $n=1$ or also for the multivariate classical case (i.e., general $n$, but the $X_i$ commute) there is a lot of (commutative) analysis available.
\end{remark}

\begin{example}

\begin{enumerate}
\item
$n=1$: Consider $X=X^*\in \cA$, where $(\cA,\ff)$ is a $C^*$-probability space. Then $\mu_X$ can be identified with a probability measure on $\RR$ (with compact support) via
$$\ff(X^k)=\int_\RR t^k d\mu_X(t)\qquad \text{for all $k\in \NN$}.$$
This follows by Weierstra\ss\ Approximation Theorem of continuous functions by polynomials on compact intervals and by Riesz Representation Theorem.
\item
The same applies to the general commutative situation. For a $C^*$-probability space $(\cA,\ff)$ and selfadjoint \emph{commuting} $X_1,\dots,X_n\in\cA$ the distribution $\mu_{X_1,\dots,X_n}$ can be identified with a compactly supported probability measure $\mu$ on $\RR^n$ via
$$\ff(X_{i_1}\cdots X_{i_k})=\int_{\RR^n} t_{i_1}\cdots t_{i_k} d\mu(t_1,\dots,t_n)\qquad\text{for all $n\in\NN$; $1\leq i_1,\dots,i_n\leq n$}.$$

\end{enumerate}
\end{example}

\begin{remark}
\begin{enumerate}
\item
Thus, in the the classical case, distributions ``are'' probability measures on $\RR^n$ and we can ask questions about their regularity:
\begin{itemize}
\item
do they have atoms;
\item
do they have a density (with respect to Lebesgue measure, or - equivalently, but maybe conceptually better - with respect to Gaussian measure);
\item 
what are the regularity properties of those densities?
\end{itemize}
\item
There are nice analytic functions which contain all the relevant information about classical distributions; in particular we have
\begin{enumerate}
\item
Fourier transform (aka characteristic function)
$$\cF(t_1,\dots,t_n)=E[e^{-i(t_1X_1+\cdots +t_nX_n)}];$$
\item
Cauchy transform (in the case $n=1$)
$$G(z)=\int \frac 1{z-t} d\mu(t) =\ff(\frac 1{z-X}),$$
which is defined and analytic on
$$\CC^+:=\{z\in\CC\mid \Im z>0\}.$$
\end{enumerate}
\item
If we have a $C^*$-probability space $(\cA,\ff)$, why are we not happy with $\ff$ restricted to the $C^*$-algebra generated by $X_1\dots,X_n$ as our analytic description? Actually, don't we say that 
$C^*(X_1,\dots,X_n)$ is like the continuous functions of $X_1,\dots, X_n$
and
$\text{vN}(X_1,\dots,X_n)$ is like the measurable functions of $X_1,\dots, X_n$; indeed 
\dots but in these phrases we cannot separate the functions from the operators.

What we really want is to compare random variables $X_1,\dots,X_n$ in $(\cA,\ff)$ with random variables $Y_1,\dots,Y_n$ in $(\cB,\psi)$, for two possibly different non-commutative probability spaces $(\cA,\ff)$ and $(\cB,\psi)$. We can only do this by comparing $\ff(f(X_1,\dots,X_n))$ with $\psi(f(Y_1,\dots,Y_n))$ for as big classes of $f$ as possible. Thus, $f$ must make sense as an abstract function which can be applied to tuples of non-commuting operators.

The same applies to the classical situation. Given classical probability spaces $(\Omega,P)$ and $(\tilde \Omega,\tilde P)$ and random variables
$X:\Omega\to\RR$ and $Y:\tilde \Omega\to\RR$, we are not comparing $P$ with $\tilde P$ or $X$ and $Y$ directly, but just their distribution, i.e.
$$\int_{\Omega}f(X(\omega))dP(\omega)\qquad\text{with}\qquad
\int_{\tilde \Omega} f(Y(\tilde \omega))d\tilde P(\tilde \omega)$$
for special classses of functions $f$; like: monomials, continuous, measurable.
\end{enumerate}
\end{remark}

\section{Examples of non-commutative distributions: full Fock space and random matrices}\label{section:1.3}

\begin{example}\label{ex:1.7}
For a Hilbert space $\HH$ we define the \emph{full Fock space} by
$$\cF(\HH):=\bigoplus_{k\ge0}\HH^{\otimes k}=\CC\cdot \Omega\oplus\HH \oplus
 \HH^{\otimes 2}\oplus\cdots
,$$
where $\Omega$ is a unit vector in $\HH^{\otimes 0}\simeq \CC$, called
\emph{vacuum}.

Elements in $\cF(\HH)$ are given by square summable linear combinations of
$f_1\otimes \cdots \otimes f_k$ ($k=0,1,\dots; f_1,\dots,f_k\in\HH$) with
inner product
$$\lb f_1\otimes\cdots\otimes f_k,g_1\otimes\cdots \otimes g_l\rb=
\delta_{kl} \lb f_1,g_1\rb \cdots \lb f_k,g_k\rb.$$
For $f\in\HH$, we define the \emph{(left) creation operator} $l(f)$, determined by
\begin{align*}
l(f)\Omega&= f\\
l(f) f_1\otimes\cdots\otimes f_k&=f\otimes f_1\otimes\cdots\otimes f_k.
\end{align*}
Its adjoint is the \emph{(left) annihilation operator} $l^*(f)$, given by
\begin{align*}
l^*(f)\Omega&=0\\
l^*(f) f_1\otimes\cdots \otimes f_k&=\lb f_1,f\rb f_2\otimes\cdots \otimes f_k.
\end{align*}
Let $\xi_1,\dots,\xi_n$ be an orthonormal system of vectors in $\HH$
(i.e., $\xi_i\perp \xi_j$ for $i\not=j$ and $\Vert \xi_i\Vert =1 $ for all $i$), then we consider the selfadjoint operators 
$$S_i:=l(\xi_i)+l^*(\xi_i)\qquad (i=1,\dots,n).$$
For $\ff$ we take
$$\ff(A):=\lb A\Omega,\Omega\rb\qquad \text{``vacuum expectation state''.}$$
We are interested in the non-commutative distribution $\mu_{S_1,\dots,S_n}$ of the operators $S_1,\dots,S_n$ in the $C^*$-probability space $(B(\cF(\HH)),\ff)$.
We have a quite good understanding of this, namely we know:
\begin{itemize}
\item
$S_1,\dots,S_n$ are free (in the sense of Voiculescu's free probability theory)
\item
and each $S_i$ has a \emph{semicircular distribution}
$$d\mu_{S_i}(t)=\frac 1{2\pi} \sqrt{4-t^2}dt \qquad \text{on $[-2,2]$},$$
\[ \begin{tikzpicture}	
	\tikzset{dot/.style={circle,fill=#1,inner sep=0,minimum size=4pt}}

	\node[dot=black] at (-2, 0)   (1) {};
	\node[dot=black] at (2, 0)   (2) {};
	\node[dot=black] at (0, 2)   (3) {};
		
	\draw (-2,-0.3) -- node {$-2$} (-2,-0.3);
	\draw (2,-0.3) -- node {$2$} (2,-0.3);
	\draw (0.3,2.3) -- node {$\frac{1}{\pi}$} (0.3,2.3);

	\draw [->] (-2.5,0) -- (2.5,0);	
	\draw [->] (0,-0.5) -- (0,2.5);	
	\draw (-2,0) -- (2,0) arc(0:180:2) --cycle;
	\end{tikzpicture}
	\]	
i.e., 
$$\ff(S_i^k)=\frac 1{2\pi} \int_{-2}^{+2} t^k \sqrt{4-t^2}dt
=\begin{cases}
0,& \text{$k$ odd}\\
\frac 1{k/2+1} \binom k{k/2}, & \text{$k$ even.}
\end{cases}$$
The non-zero moments are the \emph{Catalan numbers}. 
\end{itemize}
This $\mu_{S_1,\dots,S_n}$, the non-commutative distribution of free semicircular variables, is our benchmark; other distributions will be compared to this. In particular, the notion of a density (if there is any!) should be with regard to this.

\end{example}

\begin{example}\label{ex:1.8}
Many important distributions are given as limits of random matrices. Let $P(x_1,\dots,x_n)$ be a non-commutative selfadjoint polynomial in $n$ non-commuting variables. For example, for $n=2$,
$$P(x_1,x_2)=x_1^2+x_2^2\qquad\text{or}\qquad
P(x_1,x_2)=x_1^4+x_1x_2^2x_1+5.$$
We consider on the space of $n$-tuples $(X_1^{(N)},\dots,X_n^{(N)})$ of
selfadjoint $N\times N$ matrices the probability measure $\mu_N$ given by
$$d\mu_N(X_1^{(N)},\dots,X_n^{(N)})=c_N\cdot e^{-N^2 \tr[P(X_1^{(N)},\dots,X_n^{(N)})]}
d\lambda(X_1^{(N)})\dots d\lambda(X_n^{(N)}),$$
where $c_N$ is a normalization constant such that $\mu_N$ is a probability measure, $\tr$ denotes the normalized trace on matrices and 
$$d\lambda(X^{(N)})=\prod_{i=1}^N d(\Re x_{ii}) \prod_{1\leq i<j\leq N}
d(\Re x_{ij}) d(\Im x_{ij}) $$
is the Lebesgue measure on all entries of the selfadjoint matrix $X^{(N)}=(x_{ij})_{i,j=1}^N$ which are not constrained by the selfadjointness condition.
Then we consider on selfadjoint $N\times N$ matrices a state $\ff_N$ given by, for $k\in\NN$ and $1\leq i_1,\dots,i_k\leq n$, 
$$\ff_N(X_{i_1}^{(N)}\cdots X_{i_k}^{(N)}):=\int \tr[X_{i_1}^{(N)}\cdots X_{i_k}^{(N)}]d\mu_N(X_1^{(N)},\dots,X_n^{(N)})$$
and denote by $\mu_{X_1,\dots,X_n}$ the limit of this distribution, given
by the moments
$$\ff(X_{i_1}\cdots X_{i_k}):=\lim_{N\to\infty} \ff_N(X_{i_1}^{(N)}\cdots X_{i_k}^{(N)}),$$
provided these limits exist. The latter depends on $P$ and is, for $n\geq 2$, a big open question. Only some simple situations are will understood. E.g., for
$P(x_1,\dots,x_n)=x_1^2+\cdots +x_n^2$, corresponding to independent Gaussian random matrices (\GUE),
this limit exists and is, by results of Voiculescu, equal to the one from Example \ref{ex:1.7}, given by free semicirculars. 

To summarize, we are interested in the limit of multi-matrix models and want to understand whether such limits exist and, in particular, how to describe them.

\end{example}

The assignments address some more details about free semicirculars, in the context of the full Fock space (Exercise \ref{exercise:1}) and as the limit of random matrices (Exercise \ref{exercise:2}).

\section{Non-commutative polynomials and distributions}

\begin{definition}\label{def:1-9}
\begin{enumerate}
\item
We denote by $\CC\lb x_1,\dots,x_n\rb$ the \emph{polynomials in $n$ non-com-muting indeterminates} $x_1,\dots,x_n$; i.e., the unital algebra in $n$ algebraically free non-commuting generators $x_1,\dots,x_n$. Thus, a linear basis of $\CC\lb x_1,\dots,x_n\rb$ is given by 
all monomials  $x_{i_1}\cdots x_{i_k}$ ($k\in\NN_0$; $1\leq i_1,\dots,i_k\leq n$; $k=0$ corresponds to the constant polynomial $1$), and multiplication of two such monomials is done by justaposition. A general polynomial $p=p(x_1,\dots,x_n)\in \CC\lb x_1,\dots,x_n\rb$ is thus of the form
\begin{equation}\label{eq:1-gen-poly}
p(x_1,\dots,x_n)=\alpha_0+\sum_{k=1}^d \sum_{i_1,\dots,i_k=1}^n
\alpha_{i_1,\dots,i_k}x_{i_1}\cdots x_{i_k},
\end{equation}
for $d\in\NN_0$, $\alpha_0,\alpha_{i_1,\dots,i_k}\in\CC.$
We can make $\CC\lb x_1,\dots,x_n\rb$  to a $*$-algebra by declaring $x_i^*=x_i$ for all $i=1,\dots,n$.
\item
If $(\cA,\ff)$ is a $C^*$-probability space and $X_i=X_i^*\in\cA$ ($i=1,\dots,n$), then we have the evaluation map
\begin{align*}
\CC\lb x_1,\dots,x_n\rb&\to \cA\\
p(x_1,\dots,x_n)&\mapsto p(X_1,\dots,X_n),
\end{align*}
which is the $*$-homomorphism given by $1\mapsto 1$ and $x_i\mapsto X_i$ ($i=1,\dots,n$). More explicitly, for a non-commutative polynomial $p(x_1,\dots,x_n)$ of the form \eqref{eq:1-gen-poly} we have
\begin{equation}\label{eq:2-gen-poly-eval}
p(X_1,\dots,X_n)=\alpha_0+\sum_{k=1}^d \sum_{i_1,\dots,i_k=1}^n
\alpha_{i_1,\dots,i_k}X_{i_1}\cdots X_{i_k}.
\end{equation}
We denote by $\CC\lb X_1,\dots,X_n\rb\subset\cA$ the image of this map, i.e., the unital $*$-subalgebra of $\cA$, which is generated by $X_1,\dots,X_n$.
\item
We define now, more precisely as in Definition \ref{def:1-3}, the \emph{(non-commutative) distribution} $\mu_{X_1,\dots,X_n}$ as the linear functional
\begin{align*}
\mu_{X_1,\dots,X_n}:\CC\lb x_1,\dots,x_n\rb &\to\CC\\
p(x_1,\dots,x_n)&\mapsto \ff(p(X_1,\dots,X_n)).
\end{align*}
\end{enumerate}
\end{definition}

\begin{remark}
\begin{enumerate}
\item
With $\CC[x_1,\dots,x_n]$ we denote, as usual, the ring of polynomials in $n$ commuting variables.
\item
We might also need at some point the non-selfadjoint versions of Definition \ref{def:1-9}; i.e., if $(\cA,\ff)$ is just a non-commutative probability space the we do not put a $*$-structure on $\CC\lb x_1,\dots,x_n\rb$; or, if we deal with general, not necessarily selfadjoint, $A_1,\dots,A_n$ in a $C^*$-probability space, we have the $*$-polynomials in $n$ non-commuting non-selfadjoint indeterminates $z_1,\dots,z_n$,
$\CC\lb z_1,\dots,z_n,z_1^*,\dots,z_n^*\rb$.
\end{enumerate}
\end{remark}

\section{Generalizations of non-commutative distributions}

\begin{remark}
There appeared recently some generalizations of non-commutative distributions in the context of free probability, like:
\begin{enumerate}
\item[(i)]
Bi-distribution or pairs of faces (Voiculescu 2014 \cite{Voi-bifree}). There the random variables are divided into two classes, some random variables are declared as right variables, others as left variables.
\item[(ii)]
Trace polynomial distributions (Cebron 2013 \cite{Ceb}). There $\CC\lb x_1,\dots,x_n\rb$, the polynomials in $x_1,\dots,x_n$ with ``constant'' coefficients, is replaced by $\CC\{x_1,\dots,x_n\}$, the polynomials in $x_1,\dots,x_n$ with coefficients depending on ``(tracial) moments'' of $x_1,\dots,x_n$.
\item[(iii)]
Traffic distribution (Male 2011 \cite{Mal}). Moments can be identified with cyclic graphs
(for the case when $\ff$ is a trace); for example, 
$$\ff(T_1T_2T_3)=\frac 1N\sum_{i,j,k=1}^N t_{ij}^{(1)}t_{jk}^{(2)}t_{ki}^{(3)}$$
corresponds to
$$\includegraphics[scale=0.7]{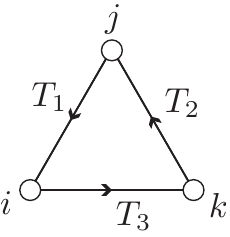}.$$
More general graphs, like
$$\includegraphics[scale=.8]{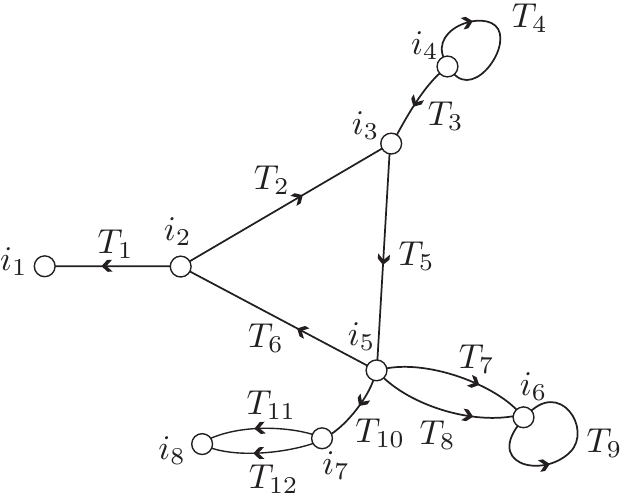}$$

correspond then to more general ``graph-moments''
\begin{equation*} 
\sum_{i_1,\dots,i_8=1}^N
t^{(1)}_{i_1i_2}t^{(2)}_{i_3i_2}t^{(3)}_{i_3i_4}t^{(4)}_{i_4i_4}
t^{(5)}_{i_5i_3}t^{(6)}_{i_2i_5}t^{(7)}_{i_6i_5}t^{(8)}_{i_6i_5}t^{(9)}_{i_6i_6}
t^{(10)}_{i_7i_5}t^{(11)}_{i_8i_7}t^{(12)}_{i_8i_7}
\end{equation*}
\end{enumerate}
For those generalizations, a general analytic theory is even more unclear than for the ordinary non-commutative distributions, and we will not address those generalizations in the following.
\end{remark}

\chapter{Operator-Valued Distributions and Operator-Valued Cauchy Transform}

Our main analytic object for dealing with non-commutative distributions will be a
version of the Cauchy transform. However, this can only be defined easily for one operator,
but in a more general, operator-valued setting. Since the information about the
non-commutative distribution of a non-commutative tuple can be rewritten in terms
of one operator-valued operator this opens the door to the analytic world of non-commutative
distributions.

\section{Going from several non-commuting operators to one operator-valued operator}

\begin{definition}
Let $(\cA,\ff)$ be a $C^*$-probability space and $X=X^*\in \cA$. The function
\begin{align}
G_X:\CC^+\to \CC^-;\quad
z\mapsto \ff(\frac 1{z-X})=\int_\RR \frac 1{z-t}d\mu_X(t)
\end{align}
is called \emph{Cauchy transform} of $X$ (or of $\mu_X)$.
\end{definition}

\begin{remark}
A Cauchy transform $G_X$ has the following properties.
\begin{itemize}
\item[(i)]
$G_X$ is analytic on $\CC^+$;
\item[(ii)]
$G_X$ has a power series expansion about $\infty$:
$$G_X(z)=\sum_{k=0}^\infty\frac{\ff(X^k)}{z^{k+1}}\qquad\text{for $\vert z\vert>\Vert X\Vert$;}$$
\item[(iii)]
we have
$$\lim_{\substack{z\in\CC^+\\ \vert z\vert\to\infty}} z G_X(z)=\ff(X^0)=1;$$
\item[(iv)]
$\mu_X$ can be recovered from $G_X$ by the \emph{Stieltjes inversion formula}
$$d\mu_X(t)=-\lim_{\ee\searrow 0}\frac 1\pi \Im G_X(t+i\ee)dt;$$
one should note that $t\mapsto -\Im G_X(t+i\ee)/\pi$ is, for each $\ee>0$, the density of a probability measure.
\end{itemize}
\end{remark}

\begin{motivation}\label{mot:2-3}
Let $(\cC,\ff)$ be a $C^*$-probability space and consider selfadjoint $X_1,\dots,$ $X_n\in\cC$. We would like to encode the information about $\mu_{X_1,\dots,X_n}$ in an analytic function, something like
\begin{equation}\label{eq:2-2}
\sum_{k=0}^\infty\sum_{i_1,\dots,i_k=1}^n z_{i_1}\cdots z_{i_k}\ff(X_{i_1}\cdots X_{i_k}).
\end{equation}
Since the $X_i$ do not commute in general, the variables $z_1,\dots,z_n$ should also not commute. Thus we need something like an analytic function in non-commuting variables. It is not clear how to give \eqref{eq:2-2} a good analytic meaning (in particular, if we want this in some non-commutative half-planes for $z_1,\dots,z_n$).

Instead, we will rewrite the above in terms of one variable $X$, but in an operator-valued setting. For this we put
$$M_n(\cC):=M_n(\CC)\otimes\cC=
\{(A_{ij})_{i,j=1}^n\mid A_{ij}\in\cC\}$$
and
$$\id\otimes\ff: M_n(\cC)\to M_n(\CC); \quad
(A_{ij})_{i,j=1}^n\mapsto (\ff(A_{ij}))_{i,j=1}^n. $$
Denoting
$$\cA:=M_n(\cC),\qquad \cB:=M_n(\CC),\qquad E:=\id\otimes\ff:\cA\to\cB,$$
we have now an operator-valued probability space, where, compared to Definition \ref{def:1-1}, $\CC$ is replaced by a (non-commutative) subalgebra $\cB$ of $\cA$ and $\ff$ is replaced by a conditional expectation $E$ onto $\cB$. In this setting we put
$$X:=\begin{pmatrix}
X_1&0&\hdots&0\\
0&X_2&\hdots&0\\
\vdots&\vdots&\ddots&\vdots\\
0&0&\hdots&X_n
\end{pmatrix}\in\cA.$$
All moments of $X_1,\dots,X_n$ with respect to $\ff$ can then be recovered from the $\cB$-valued moments $E[b_0 Xb_1X\cdots b_{k-1}Xb_k]$ ($b_0,\dots,b_k\in\cB$) of $X$. For example, $\ff(X_1X_2)$ can be recovered from
$$\begin{pmatrix}
1&0\\
0&0
\end{pmatrix}
\begin{pmatrix}
X_1&0\\
0&X_2
\end{pmatrix}
\begin{pmatrix}
0&1\\
0&0
\end{pmatrix}
\begin{pmatrix}
X_1&0\\
0&X_2
\end{pmatrix}
\begin{pmatrix}
0&0\\
1&0
\end{pmatrix}
=\begin{pmatrix}
X_1X_2&0\\0&0
\end{pmatrix}$$
as
$$\begin{pmatrix}
\ff(X_1X_2)&0\\0&0
\end{pmatrix}=E[b_0Xb_1Xb_2]$$
with
$$b_0=\begin{pmatrix}
1&0\\0&0
\end{pmatrix},\quad
b_1=\begin{pmatrix}
0&1\\0&0
\end{pmatrix},\quad
b_2=\begin{pmatrix}
0&0\\1&0
\end{pmatrix}.$$
\end{motivation}

\section{The operator-valued setting}

\begin{definition}
\begin{enumerate}
\item
An \emph{operator-valued (non-commutative) probability space} $(\cA,\cB,E)$ consists of
\begin{itemize}
\item
a unital algebra $\cA$;
\item
a unital subalgebra $1\in\cB\subset \cA$;
\item
a conditonal expectation $E:\cA\to\cB$, i.e.,
\begin{itemize}
\item
$E$ is linear;
\item
$E(1)=1$;
\item
$E$ has the bimodule property:
$$E[b_1Ab_2]=b_1 E[A]b_2\qquad\text{for all $b_1,b_2\in\cB$, $A\in\cA$;}$$
thus also in particular: $E[b]=b$ for all $b\in\cB$.
\end{itemize}
\end{itemize}
\item
If $\cA$ and $\cB$ are unital $C^*$-algebras and $E$ is positive (i.e., for all 
$A\in\cA$  there is $b\in\cB$ such that $E[A^*A]=b^*b$), then $(\cA,\cB,E)$ is an \emph{operator-valued $C^*$-probability space}.
\item
Elements $X\in\cA$ are called \emph{operator-valued} (or \emph{$\cB$-valued}) \emph{random variables}.
\item
The \emph{operator-valued moments} of $X$ are of the form
$$E[X],\quad E(Xb_1X],\quad E[Xb_1Xb_2X],\quad\dots\quad
E[Xb_1Xb_2\cdots Xb_{k-1}X],\quad\dots$$
\item
The collection of all operator-valued moments constitutes the \emph{operator-valued distribution} $\mu_X$ of $X$.
\end{enumerate}
\end{definition}

\begin{definition}
Let $(\cA,\cB,E)$ be an operator-valued $C^*$-probability space and $X=X^*\in\cA$. Then we define the \emph{operator-valued Cauchy transform} $G_X:
\cB\to\cB$ (actually not everywhere defined, nice domain will be specified later) by
$$G_X(b)=E[(b-X)^{-1}]\qquad \text{(if $b-X$ is invertible).}$$
\end{definition}

\begin{remark}
\begin{enumerate}
\item
$G_X$ is an analytic function between the Banach spaces $\cB\to\cB$ in Gateaux or Fr\`echet sense; more on this later.
\item
Formally, $G_X$ has a power series expansion: for $\Vert b^{-1}\Vert< 1/\Vert X\Vert$ we have
\begin{align*}
(b-X)^{-1}&=(b[1-b^{-1}X])^{-1}\\[.5em]
&=\sum_{k\geq 0} (b^{-1}X)^k b^{-1}\\[.5em]
&=b^{-1} + b^{-1}Xb^{-1}+b^{-1}Xb^{-1}Xb^{-1}X+\cdots,
\end{align*}
and thus
$$G_X(b)=\sum_{k\geq0} E[(b^{-1}X)^kb^{-1}].$$
\item
As we see from the power series expansion, $G_X$ does not contain information about all moments, but only about symmetric moments of the form
$E[XbXbXb\cdots bX]$. In order to get all moments we have to consider matricial extensions (amplifications) $G_X^{(m)}$ of $G_X$. For each $m\in\NN$, we amplify our setting to
$$(M_m(\cA),E\otimes\id, M_m(\cB))$$
and consider there the Cauchy transform of 
$$X\otimes 1=\begin{pmatrix}
X&0&\hdots&0\\
0&X&\hdots&0\\
\vdots&\vdots&\ddots&\vdots\\
0&0&\hdots&X
\end{pmatrix}\in M_m(\cA),$$
i.e., for $G_X^{(m)}:M_m(\cB)\to M_m(B)$ with
$b=(b_{ij})_{i,j=1}^m$ we have 
$$G_X^{(m)}(b)=
E\otimes\id [(b-X\otimes 1)^{-1}]
=E\otimes \id\left[ 
\begin{pmatrix}
b_{11}-X&b_{12}&\hdots&b_{1m}\\
b_{21}&b_{22}-X&\hdots&b_{2m}\\
\vdots&\vdots&\ddots&\vdots\\
b_{m1}&b_{m2}&\hdots&b_{mm}-X
\end{pmatrix}^{-1}\right].
$$
\item
Note that unsymmetric moments on base level $m=1$ can be recovered from symmetric moments on higher levels, similar as in \ref{mot:2-3}. For example, $E[Xb_1Xb_2X]$ can be recovered from a symmetric moment for $m=3$ as follows. For 
$$b=\begin{pmatrix}
0&b_1&0\\0&0&b_2\\0&0&0
\end{pmatrix},
\qquad
X\otimes 1=\begin{pmatrix}
X&0&0\\0&X&0\\0&0&X
\end{pmatrix}$$
we have
$$E\otimes \id[X\otimes 1\cdot b\cdot X\otimes 1\cdot b\cdot X\otimes 1]=
\begin{pmatrix}
0&0& E[Xb_1Xb_2X]\\0&0&0\\
0&0&0
\end{pmatrix}.$$
\item
Thus, in order to encode all operator-valued moments of $X$ in some analytic function, we do not just need $G_X=G_X^{(1)}$, but also all its matrix amplifications $G_X^{(m)}$. Those $G_X^{(m)}$ are related to each other for different $m$ as follows:
\begin{enumerate}
\item[(i)]
For invertible $b_1\in M_{m_1}(\cB)$ and $b_2\in M_{m_2}(\cB)$ we have
\begin{align*}
G_X^{m_1+m_2}
\begin{pmatrix}
b_1&0\\0&b_2
\end{pmatrix}&=
E\otimes\id \left[
\begin{pmatrix} 
(b_1-X\otimes 1)^{-1}&0\\
0& (b_2-X\otimes 1)^{-1}
\end{pmatrix}
\right]\\[.5em]
&=
\begin{pmatrix}
G_X^{m_1}(b_1)&0\\
0& G_X^{m_2}(b_2)
\end{pmatrix};
\end{align*}
\item[(ii)]
for invertible $S\in M_m(\CC)$ and $b\in M_m(\cB)$ we have (note that we have $S\cdot X\otimes 1\cdot S^{-1}=X\otimes 1$)
\begin{align*}
G_X^{(m)} (S b S^{-1})&
=E\otimes\id[(SbS^{-1} -X\otimes 1)^{-1}]\\
&=E\otimes \id[(SbS^{-1}-S \cdot X\otimes 1 \cdot S^{-1})^{-1}]\\
&=E\otimes\id [S (b-X\otimes 1)^{-1} S^{-1}]\\
&= S\cdot E\otimes\id [(b-X\otimes 1)^{-1}] \cdot S^{-1}\\
&=S\cdot G_X^{(m)}(b) \cdot S^{-1}.
\end{align*}
\end{enumerate}
\item
Collections of functions which satisfy (i) and (ii) are called fully matricial functions (by Voiculescu \cite{Voi-free-analysis}) or (free) non-commutative functions (by Vinnikov et al. \cite{KVV}). We will have to have a closer look on them in the next chapter.
\end{enumerate}
\end{remark}

\chapter{Non-Commutative Functions}

We will now formalize the algebraic properties of the $G_X^{(m)}$; but ignore first the question of domain. The main point will be to see that ``analyticity'' can be encoded in algebraic properties over matrices. Later, in Section \ref{section:3.2} we will also address the question of the domain. A good source for the material in this and the next chapter are the expository notes \href{http://davidjekel.com/wp-content/uploads/2019/07/OVNCP_06.pdf}{Operator-valued non-commutative probability} by David Jekel \cite{Jek}.

\section{How to encode analyticity in algebraic properties}

\begin{definition}
Let $\cB$ be a unital algebra. A collection $f=(f_m)_{m\in\NN}$ of functions
\begin{align*}
f_m:M_m(\cB)&\to M_m(\cB)\\
z&\mapsto f_m(z)
\end{align*}
is called a \emph{non-commutative function} (or \emph{fully matricial function}), if it satisfies the following two conditions.
\begin{itemize}
\item[(i)]
$f$ respects direct sums: 
\begin{equation}\label{eq:dir-sum}
f_{m_1+m_2}\left[
\begin{pmatrix}
z_1& 0\\
0& z_2
\end{pmatrix} \right]=
\begin{pmatrix}
f_{m_1}(z_1)& 0\\
0& f_{m_2}(z_2)
\end{pmatrix}
\end{equation}
for all $m_1,m_2\in\NN$, $z_1\in M_{m_1}(\cB)$, $z_2\in M_{m_2}(\cB)$.
\item[(ii)]
$f$ respects similarities:
\begin{equation}\label{eq:similarities}
f_m(SzS^{-1})=S f_m(z) S^{-1}
\end{equation}
for all $m\in\NN$, $z\in M_m(\cB)$, $S\in M_m(\CC)$ invertible.
\end{itemize}
\end{definition}

\begin{remark}
\begin{enumerate}
\item
It is fairly easy to see (and you are asked in Exercise \ref{exercise:4} to see this) that (i) and (ii) are equivalent to the fact that $f$ respects intertwinings: for all $n,m\in\NN$, $z_1\in M_n(\cB)$, $z_2\in M_m(\cB)$, and an $n\times m$ matrix $T\in M_{n,m}(\CC)$ we have
$$z_1T=Tz_2 \quad\implies\quad f_n(z_1)T=Tf_m(z_2).$$
\item
Usually, our interesting functions (like $G$) are not defined on all of $M_n(\cB)$, but only on subsets. The conditions above have then to be modified accordingly. We will ignore this for the moment, but come back to this issue later.
\item
We will often just write $f(z)$ instead of $f_m(z)$, when the $m$ is clear.
\item
We claim now that (i) and (ii) encode analyticity in an algebraic way. In particular, they should allow us to distinguish analytic functions like
$$f(z)=z\qquad \text{for all $z\in M_m(\cB)$}$$
from non-analytic ones, like
$$g(z)=z^*\qquad\text{for all $z\in M_m(\cB)$.}$$
Note that (i) does not see a difference here,
$$g
\begin{pmatrix}
z_1&0\\
0& z_2
\end{pmatrix}
=\begin{pmatrix}
z_1^*&0\\0&z_2^*
\end{pmatrix}=
\begin{pmatrix}
g(z_1)&0\\
0& g(z_2)
\end{pmatrix},$$
but (ii) does:
$$f(SzS^{-1})=SzS^{-1}=S f(z) S^{-1},$$
but
$$g(SzS^{-1})=(SzS^{-1})^*=S^{*-1} z^* S^*\not= S g(z)S^{-1}$$
in general for $S\in M_m(\CC)$ with $m\geq 2$.
Note that (ii) is for $m=1$ always trivially satisfied, since then $S\in\CC$.
\end{enumerate}
\end{remark}

\begin{example}
Consider the case $\cB=\CC$; i.e., let $f_1:\CC\to\CC$ be an analytic function.
Then one can extend this by holomorphic functional calculus to matrices via
\begin{align*}
f_m:M_m(\CC)&\to M_m(\CC)\\
z&\mapsto f_m(z):=\frac 1{2\pi i}\int_\Gamma \frac {f_1(\xi)}{\xi-z} d\xi,
\end{align*}
where we integrate around the eigenvalues of the matrix $z$. The collection
$f=(f_m)_{m\in\NN}$ satisfies then (i) and (ii):
\begin{align*}
f \begin{pmatrix} z_1 &0 \\ 0& z_2 \end{pmatrix}&=
\frac 1{2\pi i}\int_\Gamma f_1(\xi) \begin{pmatrix} \xi-z_1&0\\ 0& \xi-z_2
\end{pmatrix}^{-1} d\xi\\[.5em]
&=\frac 1{2\pi i}\int_\Gamma f_1(\xi) \begin{pmatrix} (\xi-z_1)^{-1}&0\\ 0& (\xi-z_2)^{-1}
\end{pmatrix} d\xi\\[.5em]
&=\begin{pmatrix} f(z_1) &0 \\ 0& f(z_2) \end{pmatrix}
\end{align*}
and
\begin{align*}
f(SzS^{-1})= \frac 1{2\pi i}\int_\Gamma f_1(\xi) (\xi-SzS^{-1})^{-1}d\xi
= \frac 1{2\pi i}\int_\Gamma f_1(\xi) S (\xi-z)^{-1} S^{-1}d\xi
=S f(z) S^{-1}.
\end{align*}
We can in this case also recover the derivative $f_1'$ from the action of the higher $f_m$, without taking limits. For this consider $z_1,z_2,w\in\CC$, then
\begin{align*}
f_2 \begin{pmatrix} z_1& w\\ 0&z_2 \end{pmatrix}&=
\frac 1{2\pi i} \int_\Gamma f(\xi) \begin{pmatrix}
\xi-z_1& -w\\
0& \xi-z_2
\end{pmatrix}^{-1} d\xi\\[.5em]
&=\frac 1{2\pi i} \int_\Gamma f(\xi)
\begin{pmatrix}
(\xi-z_1)^{-1}& (\xi-z_1)^{-1} w (\xi - z_2)^{-1}\\
0 & (\xi-z_2)^{-1}
\end{pmatrix}d\xi\\[.5em]
&=\begin{pmatrix}
f_1(z_1)& *\\
0& f_1(z_2)
\end{pmatrix}
\end{align*}
with
\begin{align*}
*&=\frac 1 {2\pi i} w \int_\Gamma f(\xi) (\xi-z_1)^{-1} (\xi-z_2)^{-1} d\xi\\[.5em]
&=\frac 1 {2\pi i} w\int_\Gamma f(\xi) \frac 1 {z_1-z_2}\left[ \frac 1{\xi-z_1} -
\frac 1{\xi-z_2}\right]d\xi\\[.5em]
&= w \frac {f(z_1)-f(z_2)}{z_1-z_2},
\end{align*}
and thus
$$f_2
\begin{pmatrix}
z& w\\
0 &z
\end{pmatrix}
=
\begin{pmatrix}
f_1(z)& f_1'(z) w\\
0& f_1(z)
\end{pmatrix}.
$$
\end{example}

\begin{remark}
In the same way, derivatives can be recovered for non-commutative functions, just relying on properties (i) and (ii) (and some continuity or boundedness condition). We address this in the following. For this one should note that upper triangular matrices are similar to diagonal matrices:
$$\begin{pmatrix}
z_1 & z_2-z_1\\
0& z_2
\end{pmatrix}=
\underbrace{
\begin{pmatrix}
1 &1\\ 0 & 1 \end{pmatrix} }_S
\begin{pmatrix}
z_1&0\\ 0& z_2 \end{pmatrix}
\underbrace{
\begin{pmatrix}
1 &-1\\ 0 & 1 \end{pmatrix} }_{S^{-1}}.
$$
\end{remark}

\begin{lemma}
Let $f$ be a non-commutative function. Then we have for $z_1\in M_n(\cB)$, 
$z_2\in M_m(\cB)$, $w\in M_{n,m}(\cB)$: 
$$f\begin{pmatrix}
z_1& w\\
0 &z_2
\end{pmatrix}
=
\begin{pmatrix}
f_n(z)& *\\
0& f_m(z)
\end{pmatrix}.
$$
\end{lemma}
We denote the entry in $*$ by
$\partial f(z_1,z_2)\sharp w$ or by  $\Delta f(z_1,z_2)[w]$. Note that this is an element in $M_{n,m}(\cB)$.

\begin{proof}
Write
$$f\begin{pmatrix}
z_1& w\\
0 &z_2
\end{pmatrix}
=
\begin{pmatrix}
a&b\\
c&d
\end{pmatrix}.
$$
Note that we have then by \eqref{eq:dir-sum}
$$
f
\begin{pmatrix}
z_1&w&0\\
0&z_2&0\\
0&0&z_1
\end{pmatrix}=
\begin{pmatrix}
f \begin{pmatrix}
z_1& w\\
0&z_2
\end{pmatrix}&
0\\0&f(z_1)
\end{pmatrix}=
\begin{pmatrix}
a&b&0\\c&d&0\\
0&0& f(z_1)
\end{pmatrix}.
$$
Furthermore, we have
$$
\begin{pmatrix}
z_1&w&0\\
0&z_2&0\\
0&0&z_1
\end{pmatrix}=
\underbrace{\begin{pmatrix}
1&0&-1\\
0&1&0\\
0&0&1
\end{pmatrix}}_S
\begin{pmatrix}
z_1 & w &0\\
0&z_2&0\\
0&0&z_1
\end{pmatrix}
\underbrace{
\begin{pmatrix}
1&0&1\\
0&1&0\\
0&0&1
\end{pmatrix}}_{S^{-1}}$$
and thus by \eqref{eq:similarities}
\begin{align*}
\begin{pmatrix}
a&b&0\\c&d&0\\
0&0& f(z_1)
\end{pmatrix}=
f
\begin{pmatrix}
z_1&w&0\\
0&z_2&0\\
0&0&z_1
\end{pmatrix}&=
\begin{pmatrix}
1&0&-1\\
0&1&0\\
0&0&1
\end{pmatrix}\cdot
f
\begin{pmatrix}
z_1 & w &0\\
0&z_2&0\\
0&0&z_1
\end{pmatrix}\cdot
\begin{pmatrix}
1&0&1\\
0&1&0\\
0&0&1
\end{pmatrix}\\[.5em]
&=
\begin{pmatrix}
1&0&-1\\
0&1&0\\
0&0&1
\end{pmatrix}\cdot
\begin{pmatrix}
a&b&0\\c&d&0\\
0&0& f(z_1)
\end{pmatrix}
\cdot
\begin{pmatrix}
1&0&1\\
0&1&0\\
0&0&1
\end{pmatrix}\\[.5em]
&=\begin{pmatrix}
a&b&a-f(z_1)\\
c&d& c\\
0&0&f(z_1)
\end{pmatrix}.
\end{align*}
This implies that $a=f(z_1)$ and $c=0$. Similarly, one gets that $d=f(z_1)$.
\end{proof}

\begin{lemma}\label{lem:partial-linear}
$\partial f(z_1,z_2)\sharp w$ is linear in $w$.
\end{lemma}

\begin{proof}
We have to show that
\begin{itemize}
\item[(i)] for all $\lambda\in\CC$ and $w\in M_{n,m}(\cB)$:
$$\partial f(z_1,z_2)\sharp (\lambda w)=\lambda \cdot 
\partial f(z_1,z_2)\sharp w;$$
\item[(ii)] for all $w_1,w_2\in M_{n,m}(\cB)$:
$$\partial f(z_1,z_2)\sharp (w_1+w_2)=
\partial f(z_1,z_2)\sharp w_1+\partial f(z_1,z_2)\sharp w_2.$$
\end{itemize}
We only show (i); the second part is similar, see Exercise \ref{exercise:5}.
\begin{itemize}
\item[(i)]
The case $\lambda=0$ is clear, since $\partial f(z_1,z_2)\sharp 0=0$. Thus assume that $\lambda\not=0$. We have
$$\begin{pmatrix}
z_1&\lambda w\\ 0 &z_2
\end{pmatrix}=
\begin{pmatrix}
\lambda&0\\ 0&1
\end{pmatrix}
\begin{pmatrix}
z_1& w\\
0&z_2
\end{pmatrix}
\begin{pmatrix}
 1/\lambda&0\\0&1
\end{pmatrix}$$
and thus
\begin{align*}
\begin{pmatrix}
f(z_1)&\partial f(z_1,z_2)\sharp(\lambda w)\\ 0 &f(z_2)
\end{pmatrix}&=
\begin{pmatrix}
\lambda&0\\ 0&1
\end{pmatrix}
\begin{pmatrix}
f(z_1)&  f(z_1,z_2)\sharp w\\
0&f(z_2)
\end{pmatrix}
\begin{pmatrix}
 1/\lambda&0\\0&1
\end{pmatrix}\\[.5em]
&=\begin{pmatrix}
f(z_1)& \lambda \cdot f(z_1,z_2)\sharp w\\
0&f(z_2)
\end{pmatrix}.
\end{align*}
\end{itemize}
\end{proof}

\begin{prop}\label{prop:3.7}
\begin{enumerate}
\item
$\partial f(z_1,z_2)$ is a difference operator, i.e., we have for all $m\in\NN$ and all $z_1,z_2\in M_m(\cB)$
$$f(z_1)-f(z_2)=\partial f(z_1,z_2)\sharp (z_2-z_1).$$
\item
If $f$ is continuous, then, for all $m\in\NN$ and all $z\in M_m(\cB)$, $\partial f(z,z)$ is a differential operator, i.e.,
$$\partial f(z,z)\sharp w =\lim_{\ee\searrow 0} \frac{f(z+\ee w)-f(z)}\ee.$$
\end{enumerate}
\end{prop}

\begin{proof}
\begin{enumerate}
\item
Put $S=\begin{pmatrix} 1& 1\\ 0&1 \end{pmatrix}$; then
$$\begin{pmatrix}
z_1& z_2-z_1\\0&z_2 \end{pmatrix}=
 S\begin{pmatrix}
z_1 &0\\ 0& z_2 \end{pmatrix} S^{-1},$$
and thus
\begin{align*}
\begin{pmatrix}
f(z_1)&\partial f(z_1,z_2)\sharp (z_2-z_1)\\
0&z_2 \end{pmatrix} &=
S \begin{pmatrix}
f(z_1)&0\\ 0& f(z_2)\end{pmatrix} S^{-1} \\[.5em]
&=
\begin{pmatrix}
f(z_1)& f(z_2)-f(z_1)\\
0& f(z_2)
\end{pmatrix}.
\end{align*}
\item
By Lemma \ref{lem:partial-linear} and by part (1), we have
\begin{align*}
\ee\cdot \partial f(z,z+\ee w)\sharp w=\partial f(z,z+\ee w)\sharp (\ee w)
=f(z+\ee w)-f(z).
\end{align*}
This yields
$$\partial f(z,z+\ee w)\sharp w=\frac 1\ee[f(z+\ee w)-f(z)],$$
and thus
$$f\begin{pmatrix}
z&w\\0& z+\ee w \end{pmatrix}=
\begin{pmatrix}
f(z) &\frac 1\ee [f(z+\ee w)-f(z)]\\
0 & f(z+\ee w) 
\end{pmatrix}.$$
As $f$ is assumed to be continuous, the left hand side of this converges for $\ee\searrow 0$ to
$$f\begin{pmatrix}
z&w\\0&z \end{pmatrix}=
\begin{pmatrix} f(z)& \partial f(z,z)\sharp w\\ 0& f(z) \end{pmatrix}$$
This implies then that also the right hand side of the above equation converges and we must have
$$\partial f(z,z)\sharp w=\lim_{\ee\searrow 0} \frac 1\ee [f(z+\ee w)-f(z)].$$
\end{enumerate}
\end{proof}

\begin{definition}\label{def:analytic}
Let $(E,\Vert\cdot\Vert_E)$ and $(F,\Vert\cdot\Vert_E)$ be complex Banach spaces and let $\emptyset\not=\Omega\subset E$ be open. A function $f:\Omega\to F$ is called
\begin{itemize}
\item[(i)] \emph{G\^{a}teau holomorphic} on $\Omega$, if
$$\lim_{\substack{ z\to 0\\ z\in\CC\backslash\{0\}}} \frac 1 z[f(x+zh)-f(x)]=:\delta f(x;h)$$
exists in $(F,\Vert\cdot\Vert_F)$ for all $x\in\Omega$ and all $h\in E$;
\item[(ii)]
\emph{analytic} on $\Omega$, if it is G\^{a}teau holomorphic and locally bounded, i.e., for all $x\in \Omega$ there exists $r=r(x)>0$ such that
$$\sup_{\substack{y\in\Omega\\ \Vert y-x\Vert_E<r}} \Vert f(y)\Vert_F <\infty.$$
\end{itemize}
\end{definition}

\begin{remark}
\begin{enumerate}
\item
By a theorem of Hille (1944) one knows that an analytic function is actually also
\emph{Fr\'echet holomorphic}, i.e., the ``total derivative'' $\delta f(x;\cdot):E\to F$ is a
bounded linear operator and
$$\lim_{\Vert h\Vert_E\to 0} \frac 1{\Vert h\Vert_E} \Vert f(x+h)-f(x) -\delta
f(x;h)\Vert_F=0.$$
Moreover, $f$ has locally a uniformly convergent ``Taylor series expansion''.
\item
In Proposition \ref{prop:3.7} we have seen how to get G\^{a}teau holomorphic from the algebraic conditions on our non-commutative functions, under the condition of continuity. According to Definition \ref{def:analytic} and the first part of this remark local boundedness is a more natural condition to ask for. It turns out that this is actually sufficient to ensure continuity (and thus analyticity) for our non-commutative functions.
\end{enumerate}
\end{remark}

\begin{prop}
Let $f=(f_m)_{m\in\NN}$, $f_m:M_m(\cB)\to M_m(\cB)$, be a non-commutative function. If $f$ is locally bounded (i.e., each $f_m$ is locally bounded), then $f$ is continuous (i.e., each $f_m$ is continuous). [To be precise: Boundedness and continuity is here with respect to the $C^*$-norm on each $M_m(\cB)$.]
\end{prop}

\begin{proof}
We know, by \ref{prop:3.7}, that for $z_1,z_2\in M_m(\cB)$
$$f\begin{pmatrix}
z_1&z_2-z_1\\ 0&z_2 \end{pmatrix}=
\begin{pmatrix}
f(z_1)& f(z_2)-f(z_1)\\ 0& f(z_2) \end{pmatrix};$$
and, since $\partial f(z_1,z_2)\sharp w$ is linear in $w$ (by Lemma \ref{lem:partial-linear}), also for $\lambda\in\CC$
\begin{equation}\label{eq:deriv-formula}
f\begin{pmatrix}
z_1&\lambda(z_2-z_1)\\ 0&z_2 \end{pmatrix}=
\begin{pmatrix}
f(z_1)&\lambda[ f(z_2)-f(z_1)]\\ 0& f(z_2) \end{pmatrix}.
\end{equation}
Now take $z\in M_m(\cB)$ and $\ee>0$; we want to find $\delta>0$ such that
$w\in M_m(\cB)$ and $\Vert w-z\Vert <\delta$ implies that
$\Vert f(w)-f(z)\Vert \leq\ee$.\\
For this we go to $M_{2m}(\cB)$ and consider there 
$$z\oplus z:=\begin{pmatrix} z&0\\ 0&z \end{pmatrix}.$$ 
Since $f_{2m}$ is locally bounded we find $r>0$ such that $\sup_{\dots} \Vert f(y)\Vert =:C<\infty$, where the
supremum is over $y\in M_{2m}(\cB)$ such that
$\Vert y-z\oplus z\Vert<r$.\\
Now we choose $\delta:=\min\{\frac r2, \ee\frac r{2C}\}$ and consider $w\in M_m(\cB)$ with $\Vert w-z\Vert< \delta$. Then we have
\begin{align*}
\left\Vert \begin{pmatrix}
w& \frac C\ee(w-z)\\ 0&z \end{pmatrix} -
\begin{pmatrix} z&0\\0&z \end{pmatrix}\right\Vert&=
\left\Vert\begin{pmatrix} w-z& \frac C\ee(w-z)\\0&0\end{pmatrix}\right\Vert
&\leq \underbrace{\Vert w-z\Vert}_{<\delta\leq\frac r2} +\frac C\ee\underbrace{\Vert w-z\Vert}_{<\delta\leq\frac \ee C \frac r2}<r
\end{align*}
and thus, by also using \eqref{eq:deriv-formula},
$$\left\Vert \begin{pmatrix}
f(w)& \frac C\ee[f(w)-f(z)]\\0& f(z)\end{pmatrix}\right\Vert=
\left\Vert f \begin{pmatrix} w&\frac C\ee (w-z)\\0&z\end{pmatrix}
\right\Vert\leq C.$$
This implies then
$\Vert \frac C\ee[f(w)-f(z)]\Vert\leq C$, thus
$\Vert f(w)-f(z)\Vert\leq\ee$.
\end{proof}

\begin{remark}
\begin{enumerate}
\item
This shows that for locally bounded non-commutative functions we get the derivative
$\delta f(z;w)=\partial f(z,z)\sharp w$ as part of the data of higher $f_m$:
$$f\begin{pmatrix} z&w\\0&z\end{pmatrix}=
\begin{pmatrix} f(z)&\delta f(z;w)\\ 0&f(z)\end{pmatrix}.$$
\item
In the same way we get also higher derivatives:
$$f \begin{pmatrix}
z_1 &w_1&0\\ 0&z_2&w_2\\ 0&0&z_3 \end{pmatrix}=
\begin{pmatrix}
f(z_1)& \partial f(z_1,z_2)\sharp w_1&*\\
0& f(z_2) & \partial f(z_1,z_3)\sharp w_2\\
0&0&f(z_3) \end{pmatrix},$$
where 
$$*=:\partial^2f(z_1,z_2,z_3)\sharp (w_1,w_2)$$
is a second-order difference quotient, which gives the second derivative \linebreak[4]
$\partial^2f(z,z,z)\sharp (w_1,w_2)$.
\item
One should also note that uniform local boundedness of $f$ allows us to control the size of the derivatives, so that one gets a convergent ``Taylor-Taylor expansion''
$$f(z+w)=\sum_{k=0}^\infty \partial^k(\underbrace{z,z,\dots,z)}_{k+1}\sharp (\underbrace{w,\dots,w}_k).$$
In Exercise \ref{exercise:7} you are asked to prove this expansion.\\
The Taylors here are two different people: Brook Taylor ($\sim$ 1715) from the usual Taylor series, and Joseph Taylor (1972), who started the theory of non-commutative functions in \cite{Tay}.
\\
For more on the Taylor-Taylor expansion (and also other aspects of non-commutative functions) one should consult the monograph
 \emph{Foundations of Free Non-Commutative Function Theory} (2014) by
D. Kaliuzhnyi-Verbovetskyi and V. Vinnikov \cite{KVV}.
\end{enumerate}
\end{remark}

\section{Rigorous definition of fully matricial functions, caring also about domain}\label{section:3.2}

\begin{definition}
\begin{enumerate}
\item
For a $C^*$-algebra $\cB$ we denote:
\begin{enumerate}
\item
$M_n(\cB)=M_n(\CC)\otimes\cB=:\cB^{(n)}$;
\item
for $z\in\cB^{(n)}$ we put
$$z^{(m)}=1_m\otimes z=\begin{pmatrix}
z&0&\hdots&0\\
0&z&\hdots&0\\
\vdots&\vdots&\ddots&\vdots\\
0&0&\hdots&z\end{pmatrix}
\in M_{nm}(\cB);$$
\item
for $z_1\in\cB^{(n)}$, $z_2\in \cB^{(m)}$ we put
$$z_1\oplus z_2=\begin{pmatrix}
z_1&0\\0&z_2
\end{pmatrix}\in \cB^{(n+m)};$$
\item
for $z\in\cB^{(n)}$ and $r>0$ we put
$$B^{(n)}(z,r):=\{w\in\cB^{(n)}\mid \Vert z-w\Vert<r\}\quad\text{and}
\quad
B(z,r):=\bigcup_{m\geq 1} B^{(nm)}(z^{(m)},r).$$
\end{enumerate}
\item
A \emph{fully matricial domain} $\Omega=(\Omega^{(n)})_{n\in\NN}$ over $\cB$ is a sequence of sets $\Omega^{(n)}\subset \cB^{(n)}$ satisfying the following conditions:
\begin{enumerate}
\item
$\Omega$ respects direct sums: $z_1\in \Omega^{(n)}$ and $z_2\in \Omega^{(m)}$ implies that $z_1\oplus z_2\in\Omega^{(n+m)}$;
\item
$\Omega$ is uniformly open; i.e., for each $z\in\Omega^{(n)}$ there exists 
$r>0$ such that $B(z,r)\subset\Omega$;
\item
$\Omega$ is non-empty; i.e., at least one $\Omega^{(n)}$ is non-empty.
\end{enumerate}
\item
Let $\Omega_1$ and $\Omega_2$ be fully matricial domains over $\cB_1$ and
$\cB_2$, respectively. A \emph{fully matricial function} $f=(f^{(n)})_{n\in\NN}:\Omega_1\to\Omega_2$ is a sequence of functions $f^{(n)}:\Omega_1^{(n)}\to\Omega_2^{(n)}$ satisfying the following conditions:
\begin{enumerate}
\item
$f$ respects intertwinings; i.e.,  for $z_1\in \Omega_1^{(n)}$, $z_2\in
\Omega_2^{(m)}$, $T\in M_{n\times m}(\CC)$ we have: $z_1T=Tz_2$ implies
that $f^{(n)}(z_1) T=T f^{(m)}(z_2)$.
\item
$f$ is uniformly locally bounded; i.e., for each $z\in \Omega_1^{(n)}$ there exit
$r>0$ and $M>0$ such that
$$B(z,r)\subset\Omega_1\qquad\text{and}\qquad f(B(z,r))\subset B(0,M).$$
\end{enumerate}
\end{enumerate}
\end{definition}

\begin{example}\label{ex:3.13}
\begin{enumerate}
\item
Non-commutative monomials and polynomials over $\cB$ are fully matricial with domains $\Omega_1^{(n)}=\Omega_2^{(n)}=M_n(\cB)$; see Exercise \ref{exercise:6}.
\item
Consider
$$\Omega^{(n)}:=\{z\in\cB^{(n)}\mid \text{$z$ is invertible}\}.$$
Then $\Omega=(\Omega^{(n)})_{n\in\NN}$ is a fully matricial domain and
$$f:\Omega\to\Omega;\quad z\mapsto f(z):=z^{-1}$$
is fully matricial.
\begin{proof}
It is clear that $\Omega$ respects direct sums and is non-empyt. To see that $\Omega$ is uniformly open, we claim that
$B(z,1/\Vert z^{-1}\Vert)\subset\Omega$. To check this,  note that for
$w\in B(z,1/\Vert z^{-1}\Vert)$ we have
\begin{align*}
w^{-1}=[z-(z-w)]^{-1}
=z^{-1}[1-(z-w)z^{-1}]^{-1}
=z^{-1} \sum_{k=0}^\infty [(z-w)z^{-1}]^k.
\end{align*}
Since $\Vert (z-w)z^{-1}\Vert<1$, the series converges in norm, and thus $w\in\Omega$.\\
From this calculation we also get that
$$\Vert w^{-1}\Vert\leq \frac {\Vert z^{-1}\Vert}{1-\Vert z^{-1}\Vert\cdot
\Vert z-w\Vert},$$
which shows that
$$f[B(z,1/(2\Vert z^{-1}\Vert))]\subset B(0,\Vert z^{-1}\Vert/2);$$
thus $f$ is uniformly locally bounded.\\
$f$ also respects intertwinings: suppose that $z_1T=Tz_2$; this implies that
$Tz_2^{-1}=z_1^{-1}T$, i.e., $Tf(z_2)=f(z_1)T$.
\end{proof}
\end{enumerate}
\end{example}

\begin{prop}
\begin{enumerate}
\item
Suppose that $f,g:\Omega_1\to\Omega_2$ are fully matricial. Then so are $f+g$ and $fg$.
\item
Suppose that $f:\Omega_1\to\Omega_2$ and $g:\Omega_2\to\Omega_3$ are fully matricial. Then so is the composition $g\circ f:\Omega_1\to\Omega_3$.
\end{enumerate}
\end{prop}

\begin{proof}
We only prove (1). Suppose that $z_1T=Tz_2$. Then we have
\begin{align*}
(f+g)(z_1)\cdot T=f(z_1)T+g(z_1)T=
Tf(z_2)+Tg(z_2)=
T\cdot (f+g)(z_2)
\end{align*}
and
$$(fg)(z_1)\cdot T=f(z_1)g(z_1)T=f(z_1)Tg(z_2)=Tf(z_2)g(z_2)=T\cdot (fg)(z_2).$$
Uniform local boundedness can be seen as follows: Consider $z\in\Omega^{(n)}$, then there are $r_1,M_1$ and $r_2,M_2$ such that
$$f(B(z,r_1))\subset B(0,M_1)\qquad\text{and}\qquad
g(B(z,r_2))\subset B(0,M_2).$$
Put $r:=\min(r_1,r_2)$; then we have for $w\in B(z,r)$
$$\Vert(f+g)(w)\Vert\leq \Vert f(w)\Vert+\Vert g(w)\Vert\leq M_1+M_2$$
and
$$\Vert (fg)(w)\Vert \leq \Vert f(w)\Vert\cdot \Vert g(w)\Vert\leq M_1\cdot M_2.$$
\end{proof}

\chapter{The Operator-Valued Cauchy Transform}

Now let's get serious about the operator-valued Cauchy transform as a fully
matricial function.

\section{The upper half plane as domain of the Cauchy transform}

\begin{definition}
Let $(\cA,\cB,E)$ be an operator-valued $C^*$-probability space and let $X=X^*\in\cA$. The \emph{Cauchy transform} $G_X=(G_X^{(n)})_{n\in\NN}$ of $X$ is defined by
\begin{align*}
G_X^{(n)}:H^+(M_n(\cB))\to H^-(M_n(\cB)),\qquad
z\mapsto \id\otimes E[\underbrace{(z-1\otimes X)^{-1}}_{\in M_n(\cA)}],
\end{align*}
where $H^+$ and $H^-$ denote the upper and lower, respectively, half-plane.
\end{definition}

\begin{notation}
Let $\cA$ be a unital $C^*$-algebra.
\begin{enumerate}
\item
For $A\in\cA$ we put
\begin{align*}
\Re(A)&:=\frac 12 (A+A^*) \qquad &\text{real part}\\[.5em]
\Im(A)&:=\frac 1{2i}(A-A^*)\qquad&\text{imaginary part}
\end{align*}
\item
We define the \emph{strict upper/lower half-plane} of $\cA$ by
\begin{align*}
H^+(\cA)&:=\{ A\in\cA\mid \exists\, \ee>0: \Im(A)\geq \ee\cdot 1\}\\
H^-(\cA)&:=\{ A\in\cA\mid \exists\, \ee>0: \Im(A)\leq -\ee\cdot 1\}.
\end{align*}
Instead of $\exists\, \ee>0: \Im(A)\geq \ee\cdot 1$ we will usually write
$\Im(A)>0$; and in the same spirit $\Im(A)<0$ for the second condition.
\end{enumerate}
\end{notation}

\begin{prop}\label{prop:4.3}
Let $A\in H^+(\cA)$. Then $A$ is invertible and $A^{-1}\in H^-(\cA)$.
\end{prop}

\begin{proof}
Put $X:=\Re(A)$ and $Y:=\Im(A)$; by assumption $Y$ is positive and invertible, and thus we can write
$$A=X+iY=Y^{1/2}[Y^{-1/2}XY^{-1/2}+i]Y^{1/2}.$$
Since $Y^{-1/2}XY^{-1/2}$ is selfadjoint, we have that $i$ is not in its spectrum and hence $A$ is invertible, with
$$A^{-1}=Y^{-1/2}[Y^{-1/2}XY^{-1/2}+i]^{-1}Y^{-1/2}.$$
Let us denote $Y^{-1/2}XY^{-1/2}$ by $\tilde X$, then we can calculate
$$[\tilde X+i]^{-1}=[(\tilde X-i)(\tilde X+i)]^{-1} (\tilde X-i)
=(\tilde X^2+1)^{-1} (\tilde X-i),$$
which gives finally
$$\Im(A^{-1})=Y^{-1/2}\cdot  \Im[\tilde X+i]^{-1}\cdot Y^{-1/2}
=- Y^{-1/2}\cdot \underbrace{(\tilde X^2+1)^{-1}}_{>0}\cdot Y^{-1/2}<0.$$
\end{proof}

\begin{prop}\label{prop:4.4}
$H^+(\cB_{nc}):=(H^+(M_n(\cB)))_{n\in\NN}$ is a fully matricial domain over 
$\cB$.
\end{prop}

\begin{proof}
\begin{itemize}
\item[(i)] 
$H^+(\cB_{nc})$ respects direct sums.
\\
Consider $z_1\in H^+(M_n(\cB))$ and $z_2\in H^+(M_m(\cB))$; then $\Im z_1\geq \ee_1 \cdot 1$ and $\Im z_2\geq
\ee_2 \cdot 1$ and thus
$$\Im \begin{pmatrix}
z_1&0\\ 0& z_2
\end{pmatrix}=
\begin{pmatrix}
\Im z_1&0\\ 0&\Im z_2
\end{pmatrix}
\geq
\begin{pmatrix}
\ee_1\cdot 1& 0\\ 0& \ee_2\cdot 1\end{pmatrix}
\geq
\min(\ee_1,\ee_2)\cdot 1;$$
hence $z_1\oplus z_2\in H^+(M_{n+m}(\cB))$.
\item[(ii)] $H^+(\cB_{nc})$ is uniformly open. \\
Consider $z\in H^+(M_n(\cB))$, i.e., $\Im z\geq \ee\cdot 1$; we claim that then
$B(z,\ee)\subset H^+(\cB_{nc})$. Namely, consider $w\in B(z,\ee)$, i.e., $w\in
M_{mn}(\cB)$ with $\Vert z^{(m)}-w\Vert <\ee$. Then we have
$$\Vert \Im z^{(m)}-\Im w \Vert \leq \Vert z^{(m)}-w\Vert<\ee,$$
and thus $\Im z^{(m)}-\Im w < \ee\cdot 1$, or 
$$\Im w> \underbrace{\Im z^{(m)}}_{=(\Im z)^{(m)}\geq\ee\cdot 1}-\ee \cdot 1\geq 0;$$
which shows that $w\in H^+(M_{mn}(\cB))$.
\item[(iii)]
$H^+(\cB_{nc})$ is clearly non-empty.
\end{itemize}
\end{proof}

\begin{theorem}\label{thm:4.5}
Let $(\cA,\cB,E)$ be an operator-valued $C^*$-probability space and $X=X^*\in\cA$. Then the Cauchy transform
\begin{align*}
G_X:H^+(\cB_{nc})\to H^-(\cB_{nc}),\qquad
z\mapsto \id\otimes E[(z-1\otimes X)^{-1}]
\end{align*}
is a fully matricial function.
\end{theorem}

\begin{proof}
\begin{itemize}
\item[(i)]
First we should check that $G_X^{(n)}$ sends $H^+(M_n(\cB))$ to $H^-(M_n(\cB))$. For this, consider $z\in H^+(M_n(\cB))$, i.e., $\Im z\geq \ee\cdot 1$. Then we have
$\Im(z-1\otimes X)=\Im z\geq \ee\cdot 1$, and thus
$z-1\otimes X\in H^+(M_n(\cA))$; then Proposition \ref{prop:4.3} tells us that
$(z-1\otimes X)^{-1}\in H^-(M_n(\cA))$.\\
Now we apply $\id\otimes E:M_n(\cA)\to M_n(\cB)$. By our assumption that $(\cA,\cB,E)$ is an operator-valued $C^*$-probability space, we have that $E:\cA\to\cB$ is positive. Since $E$ is a conditional expectations this implies that it is completely positive, i.e., all its amplifications $\id\otimes E$ are also positive. (Note that positivity of a linear map from $\cA$ to $\cB$ does in general not imply complete positivity; one needs some more structure, like conditional expectations.) So this implies then that
$$\id\otimes E:H^-(M_n(\cA))\to H^-(M_n(\cB))$$
and finally we have
$$G_X^{(n)}(z)=\id\otimes E[ \underbrace{(z-1\otimes X)^{-1}}_{\in H^-(M_n(\cA)}] \in H^-(M_n(\cB)).$$
\item[(ii)]
It is clear that $G_X$ respects intertwinings; compare Example \ref{ex:3.13} (2).
\item[(iii)]
It remains to see uniform local boundedness. Consider $z\in H^+(M_n(\cB))$, i.e.,
$\Im z\geq \ee\cdot 1$. As in the proof of Proposition \ref{prop:4.3}, we write 
\begin{align*}
&(z-1\otimes X)^{-1}\\&=
\left\{ \Im(z)^{1/2} \left[i\cdot 1+\Im(z)^{-1/2} \cdot(\Re(z)-1\otimes X)\cdot \Im(z)^{-1/2}\right]
\Im(z)^{1/2}\right\}^{-1}\\[.5em]
&=
\Im(z)^{-1/2}\underbrace{\bigl[ i\cdot 1+\underbrace{\Im(z)^{-1/2}\cdot (\Re(z)-1\otimes X)\cdot
\Im(z)^{-1/2}}_{\text{s.a. operator}}\bigr]^{-1}}_{\Vert\cdot\Vert\leq 1\quad\text{by functional calculus}} \Im(z)^{-1/2},
\end{align*}
which yields
$$\Vert (z-1\otimes X)^{-1}\Vert\leq \Vert \Im(z)^{-1/2}\Vert^2=
\Vert \Im(z)^{-1}\Vert.$$
Now note that $\id\otimes E$ has, as a normalized completely positive mapping, norm 1 and thus we have
\begin{align*}
\Vert G_X^{(n)}(z)\Vert&=
\Vert \id\otimes E[z-1\otimes X)^{-1}]\Vert\\
&
\leq\Vert (z-1\otimes X)^{-1}\Vert\\
&\leq \Vert \Im(z)^{-1}\Vert\\
&\leq \frac 1\ee
\end{align*}
since $\Im z\geq \ee\cdot 1$. Now we are ready to consider
$w\in B(z,\ee/2)$, say $w\in M_{mn}(\cB)$. According to the calculations in the proof of Proposition \ref{prop:4.4} we have 
$$\Im w\geq \Im z^{(m)} -\ee\cdot 1\geq \frac \ee 2\cdot 1,\quad\text{and thus}\quad
\Vert G_X(w)\Vert\leq \Vert \Im(w)^{-1}\Vert\leq\frac 2\ee;$$
hence we have a local uniform bound.
\end{itemize}
\end{proof}

\section{Positivity and boundedness properties of non-commutative distributions}

\begin{remark}
\begin{enumerate}
\item
In the scalar-valued case, i.e., $\cB=\CC$, all relevant information about distributions, i.e., probability measures, is encoded in the Cauchy transform; in particular we have
\begin{enumerate}
\item
weak convergence of probability measures corresponds to pointwise convergence of the Cauchy transforms;
\item
there are precise characterizations when an analytic function is a Cauchy transform.
\end{enumerate}
\item
There are kind of analogues of this in the operator-valued case. Of course, now
we are essentially encoding information about moments. Note that in the scalar-valued case moments describe probability measures uniquely if the latter are compactly supported, which corresponds to bounded operators. In the operator-valued case we restrict for now to bounded operators (in our $C^*$-probability spaces), thus to the non-commutative analogue of compactly supported measures. In the scalar case we can deal with any probability measure (via analytic tools, not via moments), in the operator-valued case the unbounded situation is quite unclear.
\item
Note that a compactly supported measure on the level of moments is characterized by
\begin{enumerate}
\item
positive definiteness of moments, in the sense that
$$\int p(t)\overline{p(t)} d\mu(t)\geq 0\qquad
\text{for any polynomial $\CC[t]$;}$$
\item
and exponential boundedness of moments: if $\supp \mu\subset [-M,M]$ then
$$\vert m_n\vert=\vert \int t^nd\mu(t)\vert
\leq \int \vert t\vert^n d\mu(t)
=\int_{-M}^M \vert t\vert^n d\mu(t)\leq M^n.$$
\end{enumerate}
We will now define non-commutative distributions abstractly via moments via such properties
\end{enumerate}
\end{remark}

\begin{definition}\label{def:4.7}
\begin{enumerate}
\item
Let $\cB$ be a unital algebra. We denote by $\cB\lb x\rb$ the polynomials in the formal variable $x$ with coefficients from $\cB$, i.e., the free product of $\CC\lb x\rb$ and $\cB$, with amalgamation  over $\CC\cdot 1$. Elements in $\cB\lb x\rb$ are thus linear combinations of monomials of the form
$$b_0xb_1x\cdots b_{k-1}xb_k\qquad\text{for $k\in\NN_0$, $b_0,\dots,b_k\in\cB$.}$$
The elements in $\cB$, corresponding to $k=0$, are the constant polynomials.
If $\cB$ is a $*$-algebra, then $\cB\lb x\rb$ becomes a $*$-algebra, too, by declaring $x^*=x$, i.e.,
$$(b_0xb_1x\cdots b_{k-1}xb_k)^*=b_k^*xb_{k-1}^*\cdots x b_1^*xb_0^*.$$
\item
If $\cB$ is a unital $C^*$-algebra, then a \emph{$\cB$-valued distribution} is a 
linear map $\mu:\cB\lb x\rb\to\cB$ sucht that:
\begin{enumerate}
\item
$\mu$ is unital, $\mu(1)=1$;
\item
$\mu$ is a $\cB$-$\cB$-bimodule map, i.e.,
$$\mu(bp(x)b')=b\mu(p(x))b'\qquad\text{for all $p(x)\in\cB\lb x\rb$, $b,b'\in\cB$;}$$
\item
$\mu$ is completely positive, i.e.,
$$\mu^{(n)}(p(x)^*p(x))\geq 0\qquad\text{for all $n\in\NN$ and all $p(x)\in M_n(\cB\lb x\rb)$,}$$
where $\mu^{(n)}$ is, as usual, the amplification $\id\otimes \mu$.
\end{enumerate}
We denote the set of all $\cB$-valued distributions by $\Sigma_\cB$.\\
$\mu\in\Sigma_\cB$ is \emph{exponentially bounded} if there exists $M>0$ such that we have for all $n\in \NN_0$ and all $b_1,\dots,b_n\in \cB$ that
$$\Vert\mu(xb_1xb_2\cdots xb_nx)\Vert\leq M^{n+1}\Vert b_1\Vert\cdots
\Vert b_n\Vert.$$
We write then $\mu\in \Sigma_\cB^0$.
\item
If $(\cA,\cB,E)$ is a $\cB$-valued $C^*$-probability space and $X=X^*\in\cA$, then the \emph{($\cB$-valued) distribution $\mu_X:\cB\lb x\rb\to \cB$ of $X$} is given by
$$\mu_X(p(x))=E[p(X)] \qquad
\text{for all $p(x)\in\cB\lb x\rb$.}$$
\end{enumerate}
\end{definition}

\begin{remark}
$\Sigma_\cB^0$ should consist of all possible $\cB$-valued distributions of selfadjoint random variables $X$ in $\cB$-valued $C^*$-probability spaces. That $\mu_X\in\Sigma_\cB^0$ for such $X$ is clear (see Exercise \ref{exercise:13}), that we also have the other direction is the main content of the following theorem of Popa and Vinnikov \cite{PoV}.
\end{remark}

\begin{theorem}[Popa,Vinnikov 2013]\label{thm:4.9}
For a unital $C^*$-algebra $\cB$ the following are equivalent for a linear
map $\mu:\cB\lb x\rb\to \cB$.
\begin{itemize}
\item[(i)]
$\mu\in\Sigma_\cB^0$.
\item[(ii)]
There exists a $\cB$-valued $C^*$-probability space $(\cA,\cB,E)$ and a
selfadjoint $X\in\cA$ such that $\mu_X=\mu$.
\end{itemize}
\end{theorem}

\begin{proof}[Rough sketch of the proof]
In the scalar-valued case we realize $X$ via left-multiplication by $x$ on $\CC\lb x\rb$ via GNS-like construction. Now we do an operator-valued version of this, i.e., we put on $\cB\lb x\rb$ a $\cB$-valued inner product by
$$\lb p(x),q(x)\rb_\mu:=\mu(p(x)^*q(x))\in\cB.$$
This gives on $\cB\lb x\rb$ a ($\CC$-valued) norm
$$\Vert p(x)\Vert_\mu:=\Vert \lb p(x),p(x)\rb_\mu\Vert_\cB^{1/2}.$$
Completing $\cB\lb x\rb$ with respect to this gives a Banach space 
$\overline{\cB\lb x\rb}^\mu$. $\cB\lb x\rb$ acts on this space via left multiplications. Checking a couple of technical details shows then that this action is bounded and adjointable and thus generates a $C^*$-algebra $\cA$. Let $X\in\cA$ be multiplication with $x$. We also have a conditional expectation
$E:\cA\to\cB$ given by $E[A]:=\lb 1,A1\rb_\mu$. With respect to this, $X$ has distribution $\mu$:
\begin{align*}
E[b_0Xb_1\cdots b_nXb_{n+1}]&=\lb 1,b_0Xb_1\cdots b_nXb_{n+1} 1\rb_\mu\\
&=\lb 1,b_0Xb_1\cdots b_nXb_{n+1}\rb_\mu\\
&=\mu(1b_0Xb_1\cdots b_nXb_{n+1})\\
&=\mu(b_0Xb_1\cdots b_nXb_{n+1}).
\end{align*}
\end{proof}

\section{Moments and Cauchy transform}

\begin{remark}
\begin{enumerate}
\item
Since $G_X$ depends only on the distribution $\mu_X$ of $X$, we can also write
$G_X=G_\mu$ for $\mu_X=\mu$.
\item
As in the classical case the moments of $X$ should be the coefficients in the power series expansion of $G_X$ in $z^{-1}$ about infinity. To formulate this nicely, it is better to go over to the function $H_X(z):=G_X(z^{-1})$.
\end{enumerate}
\end{remark}

\begin{prop}\label{prop:4.11}
Let $(\cA,\cB,E)$ be a $\cB$-valued $C^*$-probability space and $X=X^*\in\cA$. Then the function
\begin{align*}
H_X:H^-(\cB_{nc})\to H^-(\cB_{nc}),\qquad
z\mapsto H_X(z):=G_X(z^{-1})
\end{align*}
is a fully matricial function which has a fully matricial extension to a uniform
neighbourhood of $0$ and we have
$$E[b_0Xb_1\cdots b_{n-1}Xb_{n}]=\partial^{n+1}H_X(0,\dots,0)\sharp
(b_0,b_1,\dots,b_{n}).$$
\end{prop}

\begin{proof}
We have, uniformly in all $n$ (where we just write $X$ instead of $1\otimes X$ and $E$ instead of $\id\otimes E$):
\begin{align*}
H_X(z)=G_X(z^{-1})
=E[(z^{-1}-X)^{-1}]
=z\cdot E[(1-Xz)^{-1}]
=z\sum_{k\geq 0} E[(Xz)^k],
\end{align*}
where the sum converges uniformly for $\Vert z\Vert<1/\Vert X\Vert$. Thus $H_X$ has an extension to $B(0,1/\Vert X\Vert)$.\\
Note (see Exercise \ref{exercise:7}) that we have in general  that
$\partial^{n+1} H_X(0,\dots,0)\sharp (b_0,\dots,b_{n})$ is the upper right entry in $H_X(z)$, where
$$z:=\begin{pmatrix}
0&b_0&0&\dots&0&0\\
0&0&b_1&\hdots&0&0\\
\vdots&\vdots&\ddots& \ddots&0&0\\
0&0&0&\ddots&b_{n-1}&0\\
0&0&0&\hdots&0&b_{n}\\
0&0&0&\hdots&0&0
\end{pmatrix}.$$

Since $z$ is nilpotent we can use the expansion $H_X(z)=z\sum_{k\geq 0}
E[(Xz)^k]$ to calculate $H_X(z)$ in this case; the series will stop after the term $k=n$. Let us evaluate the cases $n=0$ and $n=1$.

For $n=0$ we have
$$H_X\begin{pmatrix}
0&b_0\\0&0\end{pmatrix}=
\begin{pmatrix}
0&b_0\\0&0\end{pmatrix},\quad
\text{and thus}\quad
\partial H_X(0,0)\sharp b_0=b_0=E[b_0].$$

For $n=1$ we have
\begin{align*}
H_X
\begin{pmatrix}
0&b_0&0\\
0&0&b_1\\
0&0&0
\end{pmatrix}
&=\begin{pmatrix}
0&b_0&0\\
0&0&b_1\\
0&0&0
\end{pmatrix}+
\begin{pmatrix}
0&b_0&0\\
0&0&b_1\\
0&0&0
\end{pmatrix}
\begin{pmatrix}
E[X]&0&0\\
0&E[X]&0\\
0&0&E[X]
\end{pmatrix} 
\begin{pmatrix}
0&b_0&0\\
0&0&b_1\\
0&0&0
\end{pmatrix}\\[.5em]
&=\begin{pmatrix}
0&b_0&b_0E[X]b_1\\
0&0&b_1\\
0&0&0
\end{pmatrix},
\end{align*}
and thus
$$\partial^2H_X(0,0,0)\sharp(b_0,b_1)=b_0E[X]b_1=E[b_0Xb_1].$$

The case of general $n$ works in the same way.
\end{proof}

\section{Analytic characterization of Cauchy transforms}

\begin{remark}\label{rem:4.12}
\begin{enumerate}
\item
In addition to the analyticity property from Proposition \ref{prop:4.11}, our Cauchy transforms $G_X$ have also a specific leading order for $z\to\infty$, namely
$$G_X(z)=z^{-1}+\cdots \qquad\text{or}\qquad
H_X(z)=z+\cdots$$
or more precisely:
$z_k G_X(z_k)\to 1$ in $M_n(\cB)$ for any sequence $(z_k)_k$ in $M_n(\cB)$ for which $\Vert z_k^{-1}\Vert \searrow 0$.
Those properties are sufficient to characterize Cauchy transforms $G_\mu$ for $\mu\in\Sigma_\cB^0$, as shown in the following theorem of John Williams \cite{Wil}.
\item
Recall first the classical scalar-valued version: Let $g:\CC^+\to\CC^-$ be an 
analytic function such
\begin{itemize}
\item[(i)] $iy g(iy)\to 1$ as $\RR\ni y\to\infty$
\item[(ii)]
and $h(z):=g(1/z)$ has an analytic continuation to a neighborhood of 0.
\end{itemize}
Then there exists a (uniquely determined) compactly supported Borel probability measure $\mu$ on $\RR$ such that $g=G_\mu$, i.e.,
$$g(z)=\int \frac 1{z-t}d\mu(t).$$
Note that without (ii) this gives a characterization of $G_\mu$ for arbitrary 
probability measures on $\RR$.
\end{enumerate}
\end{remark}

\begin{theorem}[Williams 2017]\label{thm:Williams}
Let $\cB$ be a unital $C^*$-algebra and $g=(g^{(n)})_{n\in\NN}$ be a fully matrical function $g:H^+(\cB_{nc})\to H^-(\cB_{nc})$ such that 
\begin{itemize}
\item[(i)]
for any $n\in\NN$ and for any sequence $(z_k)_{k\in\NN}$ with $z_k\in M_n(\cB)$ and which satisfies $\lim_{k\to\infty}\Vert z_k^{-1}\Vert= 0$ we have
$$\lim_{k\to\infty} z_kg^{(n)}(z_k)=1\qquad\text{in $M_n(\cB)$;}$$
\item[(ii)]
the fully matricial function $h=(h^{(n)})_{n\in\NN}$, with $h^{(n)}(z):=
g^{(n)}(z^{-1})$, has a fully matricial extension to a uniform neighborhood of 0.
\end{itemize}
Then $g=G_\mu$ for some $\mu\in \Sigma_\cB^0$.
\end{theorem}

\begin{proof}[Sketch of proof]
According to Proposition \ref{prop:4.11} we define the distribution by
$$\mu(b_0xb_1\cdots b_nxb_{n}):=\partial^{n+1} h(0,\dots,0)\sharp
(b_0,b_1,\dots,b_{n}).$$
One has to check that this has all the properties required in Definition \ref{def:4.7} for $\Sigma_\cB^0$.

Exponential boundenness comes from uniform boundedness of $h$; furthermore we have
$$\mu(b)=\partial h(0,0)\sharp b=\frac d{dt} h(0+tb)\vert_{t=0}=
\lim_{t\to 0} \frac {h(tb)}t=
\lim_{t\to 0} b\cdot \underbrace{(tb)^{-1} g((tb)^{-1})}_{\to 1}=b,$$
and thus: $\mu\vert_\cB=\id$. From this and complete positivity follows by general arguments the bimodule property.

The main problem is to show the positivity property. We reduce the problem to the scalar-valued version by applying states. For this note that 
$$\text{$b\in\cB$ positive} \quad\Longleftrightarrow \quad\text{$\phi(b)\geq 0$ for all states $\phi:\cB\to\CC$.}$$
Hence we consider, for a state $\phi$, 
$$\phi(g(\xi\cdot 1)):\CC\to\CC$$
as a function in $\xi\in\CC$; it satisfies the classical characterizing properties of a Cauchy transform, hence
$$\phi(g(\xi\cdot 1))=\int \frac 1{\xi-t}d\mu_\phi(t)$$
for some probability measure $\mu_\phi$. But the coefficients in the expansion about $\infty$ for this are $\phi(E[X^k])$, hence the $E[X^k]$ are under all $\phi$ a positive definite sequence in $\CC$, and thus the $E[X^k]$ themselves are positive definite in $\cB$. In order to get this also for general moments in $\cB\lb x\rb$ one has to consider matrix versions of this and apply states $\phi$ to the $(1,1)$-entry of matrices in $M_n(\cB)$.

\end{proof}

\chapter{Operator-Valued Freeness}

In order to be able to say something more on operator-valued distributions we need more structure in the distribution. The most prominent case is given by variables which are free. It is crucial that we have an operator-valued version of free probability theory, which behaves nicely with respect to matrix amplifications. This operator-valued freeness
will be presented here and will play a main role in most of the coming chapters.
Operator-valued free probability theory is, as its scalar-valued version, due to Voiculescu \cite{Voi-Ast}. Our presentation of operator-valued freeness is mainly based on \cite{Sp,MSp}.

\section{Definition and basic properties of operator-valued freeness}

\begin{definition}
\begin{enumerate}
\item
Let $(\cA,\cB,E)$ be an operator-valued probability space. Subalgebras $\cB\subset \cA_i\subset \cA$, $i\in I$, are called \emph{free} if $E[a_1\cdots a_k]=0$, whenever we have:
\begin{itemize}
\item
$k\in\NN$;
\item
$a_j\in \cA_{i_j}$, with $i_j\in I$, for all $j=1,\dots,k$;
\item
$E[a_j]=0$ for all $j=1,\dots,k$;
\item
$i_1\not= i_2\not= i_3\not=\dots \not= i_{k-1}\not= i_k$
(neighboring elements are from different subalgebras).
\end{itemize}
Instead of \emph{free} we will also say \emph{freely independent}, or more precisely \emph{free with respect to $E$} or \emph{free (with amalgamation) over $\cB$} or similar phrases.
\item
Random variables $X_i\in\cA$, $i\in I$, are \emph{free} if the corresponding
subalgebras 
$$\cB\lb X_i\rb:=\text{algebra generated by $X$ and $\cB$}
=\{p(X_i)\mid p(x)\in\cB\lb x\rb\}$$
are free
\end{enumerate}
\end{definition}

\begin{prop}\label{prop:5.2}
If $\cA_i$, $i\in I$, are free then $E$ is on the algebra generated by all $\cA_i$ determined by the restrictions $E\vert_{\cA_i}$ for all $i\in I$ and by the freeness condition.
\end{prop}

\begin{proof}
The algebra generated by all $\cA_i$ consists of elements which are linear combinations of $a_1\cdots a_k$ where $k\in\NN_0$, $a_j\in\cA_{i_j}$ with
$i_j\in I$; we an also assume that $i_1\not=i_2\not=\dots\not=i_k$. Consider such $a_1\cdots a_k$. We have to show that $E[a_1\cdots a_k]$ is determined by $E\vert_{\cA_i}$ ($i\in I$). We do this by induction. The case $k=0$ ($E[1]=1$) and $k=1$ are clear (as $a_1\in\cA_{i_1}$).

Consider now general $k$. We put
$$a_j^o:=a_j-\underbrace{E[a_j]}_{\in\cB\subset\cA_{i_j}} \in \cA_{i_j},\quad\text{then}\quad
E[a_j^o]=0.$$ 
We get then
$$E[a_1\cdots a_k]=E\bigl[(a_1^o+E[a_1])\cdots (a_k^o+E[a_k])\bigr]=
E[a_1^o\cdots a_k^o]+rest.$$
The first term vanishes by the definition of freeness and the $rest$-term is a sum of terms of smaller length, which are already determined by the induction hypothesis.
\end{proof}

\begin{example}\label{ex:5.3}
\begin{enumerate}
\item
Consider $a_1\in\cA_1$ and $a_2\in\cA_2$. Then we have
\begin{align*}
0&=E\bigl[(a_1-E[a_1])(a_2-E[a_2])\bigr]\\[.5em]
&=E[a_1a_2]-E\bigl[a_1\cdot E[a_2]\bigr]-E\bigl[E[a_1]\cdot a_2\bigr]+E\bigl[E[a_1]\cdot E[a_2]\bigr]
\end{align*}
The three last terms are actually all equal to $E[a_1]\cdot E[a_2]$, which leads to
$E[a_1a_2]=E[a_1]\cdot E[a_2]$.
\item
Consider $a_1,\tilde a_2\in\cA_1$ and $a_2\in\cA_2$. Then we have
\begin{align*}
0&=E\bigl[(a_1-E[a_1])(a_2-E[a_2])(\tilde a_1-E[\tilde a_1])\bigr]\\[.5em]
&=E[a_1a_2\tilde a_1]-E\bigl[a_1\cdot E[a_2]\cdot \tilde a_1\bigr] +\text{six other terms which cancel}.
\end{align*}
Thus we obtain $E[a_1a_2\tilde a_1]=E[a_1E[a_2] \tilde a_1]$. This cannot be factorized further, as $E[a_2]\in\cB$ does in general not commute with $a_1$ or $\tilde a_1$. However, this is okay, as $E[a_2]\in\cB$ and hence $a_1E[a_2]\tilde a_1\in \cA_1$, so $E[a_1E[a_2] \tilde a_1]$ is a moment which is determined by $E[a_2]$ and by $E\vert_{\cA_1}$.
\item
For $a_1,\tilde a_1\in\cA_1$ and $a_2,\tilde a_2\in\cA_2$ one calculates in the same way
$$E[a_1a_2\tilde a_1\tilde a_2]=E[a_1E[a_2]\tilde a_1]\cdot E[\tilde a_2]
+E[a_1]\cdot E[a_2 E[\tilde a_1]\tilde a_2]
-E[a_1] E[a_2] E[\tilde a_1] E[\tilde a_2].$$
\end{enumerate}
\end{example}

\begin{remark}
\begin{enumerate}
\item
If $\cB=\CC$ and $E=\ff$, then $\ff(a)$ commutes with everything and we can factorize the final results, like
$$\ff(a_1a_2\tilde a_1)=\ff(a_1\ff(a_2)\tilde a_1)=\ff(a_1\tilde a_1)\ff(a_2),$$
and we get the formulas from usual (scalar-valued) free probability.
\item
Note: on the level of moments, operator-valued freeness works like scalar-valued freeness, but one has to keep the original order of the elements.
\item
Note also that with respect to $E:\cA\to\cB$ the ``non-commutative scalars'' $\cB$ are free from any subalgebra.
\item
For a random variable $X\in\cA$, the restriction of $E$ to $\cB\lb X\rb$ is exactly the information about the moments of $X$. Hence Proposition \ref{prop:5.2} says in this case that the joint moments of free variables $X_i$ ($i\in I$) are determined by the moments of the individual variables.

For example, for $X$ and $Y$ free we have
$$E[XbY]=E[X]\cdot b\cdot E[Y]=E[Xb]\cdot E[Y]=E[X] \cdot E[bY]$$
and
$$E[Xb_1Yb_2X]=\underbrace{E\bigl[ Xb_1 \cdot \underbrace{E[Y]}_{\substack{\text{moment}\\ \text{of $Y$}}} \cdot b_2X\bigr]}_{\text{moment of $X$}}.$$
\item
Note that Proposition \ref{prop:5.2} gives us essentially a free product construction on an algebraic level. Since we want to do our constructions on an analytic $C^*$-probability level, we should extend our abstract notion of $\cB$-valued distributions from Definition \ref{def:4.7} from the case of one variable to the multivariate case.
\end{enumerate}
\end{remark}

\section{$\cB$-valued joint distributions}

\begin{definition}\label{def:5.5}
\begin{enumerate}
\item
Let $\cB$ be a unital $C^*$-algebra. We denote by $\cB\lb x_i;i\in I\rb$ the non-commutative polynomials in the formal variables $x_i$ ($i\in I$) with coefficients from $\cB$; they are linearly spanned by monomials of the form 
$$b_0x_{i_1}b_1 x_{i_2}\cdots b_{k-1} x_{i_k} b_k\qquad
\text{with $k\in\NN_0$; $b_0,\dots b_k\in \cB$; $i_1,\dots,i_k\in I$.}$$
This becomes a $*$-algebra by declaring $x_i^*=x_i$ for all $i\in I$. 
\item
A \emph{$\cB$-valued (joint) distribution} is a linear map $\mu:\cB\lb x_i; i\in I\rb\to \cB$ such that
\begin{itemize}
\item[(i)]
$\mu(1)=1$;
\item[(ii)]
$\mu$ is a $\cB$-$\cB$-bimodule map;
\item[(iii)]
$\mu$ is completely positive, i.e., 
$$\id\otimes\mu(p^*p)\geq 0 \qquad
\text{for all $n\in\NN$ and $p=p(x_i;i\in I)\in M_n(\cB\lb x_i; i\in I\rb)$;}$$
\item[(iv)]
$\mu$ is exponentially bounded, i.e., there exists $M>0$ such that for all
$k\in\NN_0$, $b_1,\dots, b_{k-1}\in\cB$, $i_1,\dots,i_k\in I$ the following holds:
$$\Vert \mu(x_{i_1}b_1x_{i_2}\cdots b_{k-1} x_{i_k}\Vert\leq
M^k \Vert b_1\Vert \cdots \Vert b_k\Vert.$$
\end{itemize}
We denote
\begin{align*}
\Sigma_\cB^I:=\{\text{$\mu$ satisfying (i), (ii), (iii)}\},\qquad
\Sigma_\cB^{I,0}:=\{ \mu\in \Sigma_\cB^I, \text{satisfying also (iv)}\}.
\end{align*}
\item
If $(\cA,\cB,E)$ is a $\cB$-valued $C^*$-probability space and $X_i=X_i^*\in\cA$ for all $i\in I$, then the $\cB$-valued joint distribution
$\mu_{(X_i;i\in I)}\in \Sigma_\cB^{I,0}$ is given by
$$\mu_{(X_i;i\in I)}(p(x_i;i\in I)):=E[p(X_i;i\in I)]\qquad
\text{for all $p(x_i;i\in I)\in \cB\lb x_i;i\in I\rb$.}$$
\end{enumerate}
\end{definition}

\begin{theorem}
For a unital $C^*$-algebra $\cB$ and for a linear map $\mu:\cB\lb x_i;i\in I\rb\to \cB$ the following are equivalent.
\begin{itemize}
\item[(i)]
$\mu\in\Sigma_\cB^{I,0}$.
\item[(ii)]
There exist a $\cB$-valued $C^*$-probability space $(\cA,\cB,E)$ and 
$X_i=X_i^*\in\cA$ for each $i\in I$ sucht that
$\mu=\mu_{(X_i;i\in I)}$.
\end{itemize}
\end{theorem}

\begin{proof}[``Proof'']
This can be done as in the proof of Theorem \ref{thm:4.9}, or it can also be reduced (at least for $\vert I\vert< \infty$) directly to Theorem \ref{thm:4.9} with the usual matrix trick by taking  a diagonal matrix $X$, where the $X_i$ are sitting on the diagonal.
\end{proof}

\section{Compatibility of operator-valued freeness with matrix amplifications}

\begin{prop}\label{prop:5.7}
Let $(\cA,\cB,E)$ be an operator-valued probability space and let
$\cB\subset\cA_i\subset\cA$, $i\in I$, be free over $\cB$. Then, for any $n\in\NN$, in the operator-valued probability space $(M_n(\cA),M_n(\cB),\id\otimes E)$, the subalgebras $M_n(\cB)\subset M_n(\cA_i)\subset M_n(\cA)$, $i\in I$, are free over $M_n(\cB)$.
\end{prop}

\begin{proof}
Consider $A_j\in M_n(\cA_{i_j})$ such that 
$i_1\not=i_2\not=\dots\not= i_r$ and $\id\otimes E[A_j]=0$ for all $j=1,\dots,r$. We have to show that
$\id\otimes E[A_1\cdots A_r]=0$. Write
$A_j=(a_{kl}^{(j)})_{k,l=1}^n$ with $a_{kl}^{(j)}\in\cA_{i_j}$. Then
$$\id\otimes E[A_j]=(E[a_{kl}^{(j)}])_{k,l}=0$$
means that all $E[a_{kl}^{(j)}]=0$. For $A:=A_1\cdots A_r=(a_{kl})_{k,l=1}^n$ we have
$$a_{kl}=\sum_{\substack{r_1,\dots,r_{k-1}=1}}^n \underbrace{a_{kr_1}^{(1)}}_{\in\cA_{i_1}} \underbrace{a_{r_1r_2}^{(2)}}_{\in\cA_{i_2}}\cdots \underbrace{a_{r_{k-1}l}^{(r)}}_{\in \cA_{i_r}}.$$
For each fixed choice of $r_1,\dots,r_{k-1}$, the factors in the product are coming alternatingly from different subalgebras and each is centred under $E$.
Hence, by the freeness of the $\cA_i$, we get
$$E[a_{kl}]=\sum \underbrace{E[a_{kr_1}^{(1)} a_{r_1r_2}^{(2)}\cdots a_{r_{k-1}l}^{(r)}]}_{=0,\text{ for all $r_1,\dots,r_{k-1}$}}=0;
$$
but this means that
$$\id\otimes A=(E[a_{kl}])_{k,l=1}^n=0.$$
\end{proof}

\begin{remark}
\begin{enumerate}
\item
Note that $M_n(\cA)$ is also a $\cB$-valued probability space with respect to $\tr\otimes E$, where $\tr$ denotes the normalized trace on $M_n(\CC)$. We are not claiming freeness in this space -- this is actually not true in general.

For example, consider a scalar-valued probability space $(\cA,\ff)$. Then $M_2(\cA)$ is both a scalar-valued probability space (with respect to $\tr\otimes \ff$) and an operator-valued probability space (with respect to
$\id\otimes \ff$). Freeness with respect to $\ff$ goes only over to freeness
with respect to $\id\otimes\ff$, but not with respect to $\tr\otimes\ff$. For example, if $a_1,\tilde a_1\in\cA_1$ and $a_2,\tilde a_2\in\cA_2$ are free in $\cA$, then for
$$A_1=\begin{pmatrix}
a_1&0\\ 0&\tilde a_1
\end{pmatrix}\in M_2(\cA_1)
\qquad\text{and}\qquad
A_2=\begin{pmatrix}
a_2&0\\ 0&\tilde a_2
\end{pmatrix}\in M_2(\cA_2)$$
we have
$$A_1A_2=\begin{pmatrix}
a_1a_2&0\\ 0& \tilde a_1\tilde a_2
\end{pmatrix}$$
and thus on the operator-valued level:
\begin{align*}
\id\otimes\ff[A_1A_2]&=
\begin{pmatrix}
\ff(a_1a_2)&0\\
0& \ff(\tilde a_1\tilde a_2)
\end{pmatrix}\\[.5em]
&=\begin{pmatrix}
\ff(a_1)\ff(a_2)&0\\
0& \ff(\tilde a_1)\ff(\tilde a_2)
\end{pmatrix}\\[.5em]
&=\begin{pmatrix}
\ff(a_1)&0\\
0& \ff(\tilde a_1)
\end{pmatrix}
\begin{pmatrix}
\ff(a_2)&0\\
0& \ff(\tilde a_2)
\end{pmatrix}\\[.5em]
&=\id\otimes\ff [A_1]\cdot \id\otimes\ff [A_2];
\end{align*}
on the scalar-valued level, on the other side, we have in general:
\begin{align*}
\tr\otimes\ff(A_1A_2)&=\frac 12[\ff(a_1)\ff(a_2)+\ff(\tilde a_1)\ff(\tilde a_2)]\\[.5em]
&\not= \frac 12[\ff(a_1)+\ff(\tilde a_1)]\cdot \frac 12 [\ff(a_2+\ff(\tilde a_2)]\\[.5em]
&=\tr\otimes\ff(A_1)\cdot \tr\otimes \ff(A_2).
\end{align*}
\item
Note however that, even if in the end we are only interested in moments with respect to 
$\tr\otimes E$, it is good to know something about the moments with respect to $\id\otimes E$, since those are related by $\tr\otimes E=\tr[\id\otimes E]$; i.e., instead of going directly down to $\cB$,
$$M_n(\cA)\overset{\tr\otimes E}{\longrightarrow} \cB$$
we can also decompose this into two steps:
$$M_n(\cA)\overset{\id\otimes E}{\longrightarrow} M_n(\cB)\overset{\tr}{\longrightarrow}\cB.$$
This simple observation will be crucial for our latter investigations!
\end{enumerate}
\end{remark}

\section{Structure of formulas for mixed moments in free variables}

\begin{remark}
\begin{enumerate}
\item
We have to understand better the structure of the formulas for mixed moments in free variables. This is analogous to the scalar-valued case, in particular non-crossing partitions will feature prominently. For the relevant definitions and notations in relation with partitions and kernels of multi-indices we refer to Chapter 2 of the \href{https://rolandspeicher.files.wordpress.com/2019/08/free-probability.pdf}{Free Probability Lecture Notes}.
\item
As in the scalar-valued case, we get for ``non-crossing moments'' a kind of factorizing into the moments of the individual subalgebras; however, we have now to respect the nestings of the blocks. This is just an iteration of the
``factorization'' from Example \ref{ex:5.3},
\begin{equation}\label{eq:factorization}
E[a_1a_2\tilde a_1]=E\bigl[ a_1\cdot E[a_2]\cdot \tilde a_1\bigr]\qquad
\text{for $\{a_1,\tilde a_1\}$ free from $a_2$.}
\end{equation}
For example, consider $\{a_1,a_2,a_3\}$, $\{e_1,e_2\}$, $c$, $d$ which are free
with respect to $E$. Then we can iterate the factorization \eqref{eq:factorization} as follows:
\begin{align*}
\begin{tikzpicture}[baseline=-2pt]
    \begin{scope}[xshift=-2.em, xscale=0.8]
      \draw (0,0em) -- ++ (0,-2em) -- ++(4em,0) -- ++(0,2em);
      \draw (4em,-2em) -| ++(2em,2em);
      \draw (5em,0) -- ++ (0,-1em);
      \draw (1em,0) -- ++ (0,-1.5em) -| (3em,0);
      \draw (2em,0) -- ++ (0,-1em);      
    \end{scope}
     \node[fill=white, inner sep=1pt] at (0,0) {$E[a_1e_1ce_2a_2da_3]$};  
   \end{tikzpicture}
&=E\bigl[\bigl(a_1 E[e_1 c e_2] a_2 \bigr) d \bigl(a_3\bigr)\bigr]\\
&=E\bigl[ a_1 \underbrace{E[e_1 c e_2]}_{E[e_1E(c)e_2]} a_2 E[d] a_3\bigr]\\
&=E\bigl\{ a_1 E\bigl[ e_1 E[c] e_2\bigr] a_2 E[d] a_3\bigr\}.
\end{align*}
We will denote this ``factorization'' by
$$E_\pi[a_1,e_1,c,e_2,a_2,d,a_3]\qquad\text{for $\pi=$  }
\begin{tikzpicture}[baseline=-15pt]
    \begin{scope}[xshift=2.2em, scale=0.8]
      \draw (0,0em) -- ++ (0,-2em) -- ++(4em,0) -- ++(0,2em);
      \draw (4em,-2em) -| ++(2em,2em);
      \draw (5em,0) -- ++ (0,-1em);
      \draw (1em,0) -- ++ (0,-1.5em) -| (3em,0);
      \draw (2em,0) -- ++ (0,-1em);      
    \end{scope};  
   \end{tikzpicture}
$$
\item
Note that also for ``crossing moments'' only non-crossing factorizations show up in the formula expressing it in the individual moments, like in part (3) of Example \ref{ex:5.3}, for $\{a_1,\tilde a_1\}$ free from $\{a_2,\tilde a_2\}$:
 \begin{align*}
       \begin{tikzpicture}[baseline=0pt]
    \begin{scope}[xshift=-1.25em, xscale=0.95]
      \draw (0,-1em) --  (0,-2em) -- (2em,-2em) -- (2em,-1em);
      \draw (1em,-1em) --  (1em,-2.5em) -- (3em,-2.5em) -- (3em,-1em);
    \end{scope}
    \node[fill=white, inner sep=1pt] at (0,0) {$E[a_1a_2\tilde a_1 \tilde a_2)$};
    \useasboundingbox (-2.625em,-1.625em) rectangle (2.625em, 0.75em);
  \end{tikzpicture}  
&=\begin{tikzpicture}[baseline=0pt]
    \begin{scope}[xshift=-2.3em, xscale=1.6]
      \draw (0,-.5em) -- ++ (0,-2em) -- ++(2em,0) -- ++(0,2em);
 \draw (1em,-.5em) -- ++ (0,-1.5em);
 \draw (3.5em,-0.5em) -- ++ (0,-2em);
    \end{scope}
    \node[fill=white, inner sep=1pt] at (0,0) {$E\bigl[a_1E[a_2]\tilde a_1\bigr] E[ \tilde a_2]$};
    \useasboundingbox (-2.625em,-1.625em) rectangle (2.625em, 0.75em);
  \end{tikzpicture}  
                  +           
\begin{tikzpicture}[baseline=0pt]
    \begin{scope}[xshift=-1.5em, xscale=1.6]
 \draw (1em,-1em) --  (1em,-2.5em) -- (3em,-2.5em) -- (3em,-1em);
 \draw (-.5em,-.5em) -- ++ (0,-2em);
 \draw (2em,-0.5em) -- ++ (0,-1.5em);
    \end{scope}
    \node[fill=white, inner sep=1pt] at (0,0) {$E[a_1] E\bigl[a_2 E[\tilde a_1] \tilde a_2\bigr]$};
    \useasboundingbox (-2.625em,-1.625em) rectangle (2.625em, 0.75em);
  \end{tikzpicture}  
 -          \begin{tikzpicture}[baseline=0pt]                        \begin{scope}[xshift=0.375em,xscale=1]
       \draw (1em,-.5) --  (1em,-2.5em);
    \end{scope}
    \node[fill=white, inner sep=1pt] at (1em,0) {$E[a_1]$};
    \useasboundingbox (-0.05em,-1.125em) rectangle (2.075em, 0.75em); 
  \end{tikzpicture}
                           \begin{tikzpicture}[baseline=0pt]                        \begin{scope}[xshift=0.375em,xscale=1]
       \draw (1em,-.5) --  (1em,-2.5em);
    \end{scope}
    \node[fill=white, inner sep=1pt] at (1em,0) {$E[a_2]$};
    \useasboundingbox (-0.05em,-1.125em) rectangle (2.075em, 0.75em); 
  \end{tikzpicture}
                           \begin{tikzpicture}[baseline=0pt]                        \begin{scope}[xshift=0.375em,xscale=1]
    \draw (1em,-.5) --  (1em,-2.5em);
    \end{scope}
    \node[fill=white, inner sep=1pt] at (1em,0) {$E[\tilde a_1]$};
  \end{tikzpicture}
                           \begin{tikzpicture}[baseline=0pt]                        \begin{scope}[xshift=0.375em,xscale=1]
    \draw (1em,-.5) --  (1em,-2.5em);
    \end{scope}
    \node[fill=white, inner sep=1pt] at (1em,0) {$E[\tilde a_2]$};
    \useasboundingbox (-0.05em,-1.125em) rectangle (2.075em, 0.75em); 
  \end{tikzpicture}.
 \end{align*}
This is quite relevant in the operator-valued case; whereas in the scalar-valued situation the meaning of a crossing term like
$$\ff_{
  \begin{tikzpicture}[baseline=0pt]
    \begin{scope}[xshift=-1.25em, scale=0.5, yshift=1.5em]
      \draw (0,-1em) --  (0,-2em) -- (2em,-2em) -- (2em,-1em);
      \draw (1em,-1em) --  (1em,-2.5em) -- (3em,-2.5em) -- (3em,-1em);
    \end{scope}
  \end{tikzpicture}  }
(a_1,a_2,\tilde a_1,\tilde a_2)=\ff(a_1\tilde a_1)\cdot \ff(a_2\tilde a_2)$$
is clear, there is no canonical definition for
$$E_{
  \begin{tikzpicture}[baseline=0pt]
    \begin{scope}[xshift=-1.25em, scale=0.5, yshift=1.5em]
      \draw (0,-1em) --  (0,-2em) -- (2em,-2em) -- (2em,-1em);
      \draw (1em,-1em) --  (1em,-2.5em) -- (3em,-2.5em) -- (3em,-1em);
    \end{scope}
  \end{tikzpicture}  }
[a_1,a_2,\tilde a_1,\tilde a_2]$$
in the operator-valued case:
$$E[a_1\tilde a_2]\cdot E[a_2\tilde a_2]\not= E[a_2\tilde a_2]\cdot E[a_1 \tilde a_1]$$
in general, and there is no nested version which respects the order of the variables.
\end{enumerate}
\end{remark}

\begin{definition}\label{def:5.10}
\begin{enumerate}
\item
Let $\cB\subset \cA$ be an inclusion of unital subalgebras. A \emph{$\cB$-balanced map} $T:\cA^n\to\cB$ is a $\CC$-multilinear map, which satisfies also the following conditions for all $a_1,\dots,a_n
\in\cA$, $b,b'\in\cB$, $k=1,\dots,n-1$:
\begin{align*}
T(ba_1,a_2,\dots, a_nb')&=b T(a_1,a_2,\dots,a_n) b'\\
T(a_1,\dots,a_k b,a_{k+1},\dots,a_n)&=T(a_1,\dots,a_k, ba_{k+1},\dots,a_n).
\end{align*}
\item
For a given sequence $T_n:\cA^n\to\cB$ ($n\in\NN$) of $\cB$-balanced maps, we define the corresponding multiplicative maps $T_\pi$ ($n\in\NN$, $\pi\in NC(n)$) recursively on the number of blocks by:
$T_{1_n}:=T_n$ for all $n\in \NN$; and, for
$$\pi=\sigma\cup \underbrace{(p+1,p+2,\dots,p+q)}_{\text{interval block}} \in NC(n)$$
we set
$$T_\pi(a_1,\dots,a_n):=T_\sigma(a_1,\dots,a_p\cdot T_q(a_{p+1},\dots,a_{p+q}),
a_{p+q+1},\dots,a_n).$$
Note that $T_\pi$ is also $\cB$-balanced.
\end{enumerate}
\end{definition}

\begin{example}
For 
$\pi=
\begin{tikzpicture}[baseline=0pt]
    \begin{scope}[yshift=.7em, scale=0.4]
 \draw (0em,-0em) --  (0em,-3em) -- (9em,-3em) -- (9em,-0em);
 \draw (1em,-0em) --  (1em,-2.em) -- (4em,-2.em) -- (4em,-0em);
 \draw (4em,-2.em) --  (8em,-2.em) -- (8em,0);
 \draw (2em,-0em) --  (2em,-1.em) -- (3em,-1.em) -- (3em,-0em);
\draw (6em,-0em) --  (6em,-1.em) -- (7em,-1.em) -- (7em,-0em);
\draw (5em,-0em) --  (5em,-1.em);
    \end{scope}
  \end{tikzpicture} $
we have
\begin{align*}
&T_\pi\underset
{
\begin{tikzpicture}[baseline=0pt]
    \begin{scope}[yshift=0em, scale=1.4]
 \draw (0em,-0em) --  (0em,-3em) -- (9em,-3em) -- (9em,-0em);
 \draw (1em,-0em) --  (1em,-2.em) -- (4em,-2.em) -- (4em,-0em);
 \draw (4em,-2.em) --  (8em,-2.em) -- (8em,0);
 \draw (2em,-0em) --  (2em,-1.em) -- (3em,-1.em) -- (3em,-0em);
\draw (6em,-0em) --  (6em,-1.em) -- (7em,-1.em) -- (7em,-0em);
\draw (5em,-0em) --  (5em,-1.em);
    \end{scope}
  \end{tikzpicture}  
} 
{(a_1,a_2,a_3,a_4,a_5,a_6,a_7,a_8,a_9,a_{10})}\\\quad \\&=
T_2\Bigl(a_1\cdot T_3\bigl(a_2\cdot T_2(a_3,a_4),a_5\cdot T_1(a_6)\cdot T_2(a_7,a_8),
a_9\bigr),a_{10}\Bigr).
\end{align*}
\end{example}

\begin{prop}
Let $(\cA,\cB,E)$ be a $\cB$-valued probability space and let $\cB\subset
\cA_i\subset\cA$, $i\in I$, be free with respect to $E$. We denote, for $n\in\NN$, by $E_n:\cA^n\to \cB$ the $\cB$-balanced map given by
$E_n(a_1,a_2,\dots,a_n):=E[a_1a_2\cdots a_n]$
and by $E_\pi$, for all $n\in\NN$, $\pi\in NC(n)$, the corresponding 
multiplicative map. 
Consider now $a_j\in\cA_{i_j}$ for $j=1,\dots,k$. If $\ker i\in\cP(k)$ is
non-crossing, then 
$$E[a_1 a_2\cdots a_k]=E_{\ker i}(a_1,a_2,\dots,a_k).$$
\end{prop}

\begin{proof}
By iteration of \eqref{eq:factorization}:
\begin{equation*}
E[a_1a_2\tilde a_1]=E\bigl[ a_1\cdot E[a_2]\cdot \tilde a_1\bigr]
=E_{
\begin{tikzpicture}[baseline=0pt]
    \begin{scope}[yshift=0.4em, scale=.4]
      \draw (0,0em) --  (0,-2em) -- (2em,-2em) -- (2em,0em);
 \draw (1em,0em) --  (1em,-1.em);
    \end{scope}
  \end{tikzpicture}  }(a_1,a_2,a_3)
\qquad
\text{for $\{a_1,\tilde a_1\}$ free from $a_2$.}
\end{equation*}
\end{proof}

\section{Positivity of free product contructions}

\begin{prop}\label{prop:5.13}
Let $(\cA,\cB,E)$ be a $\cB$-valued probability space. Assume that $\cB$ is a unital $C^*$-algebra and $\cA$ a $*$-algebra. Let $*$-subalgebras $\cB\subset
\cA_i\subset\cA$, $i\in I$, be free with respect to $E$ and assume that $\cA$ is generated by all $\cA_i$, $i\in I$, as an algebra. If $E$ is positive restricted to each $\cA_i$ then it is also positive on $\cA$.\newline
(Recall that ``positive'' on a $*$-algebra $\cA$ means that
$E[a a^*]\geq 0$ for all $a\in\cA$.)
\end{prop}
 
\begin{proof}
\begin{itemize}
\item[(i)]
As in the proof of Proposition \ref{prop:5.2} one can see by recursion that each element in $\cA$ can be written as a linear combination of elements of the form $a_1\cdots a_n$ with
\begin{itemize}
\item 
$n\in\NN_0$ ($n=0$ corresponds to elements from $\cB$);
\item
$a_k\in\cA_{i_k}$
\item
$i_1\not= i_2\not= \dots \not= i_n$
\item
$E[a_k]=0$ for all $k=1,\dots,n$.
\end{itemize}
If $a$ is a sum of such elements and we want to argue that $E[aa^*]\geq 0$, then we have to understand $E$ appplied to a product of two such elements.
\item[(ii)]
So let us consider two such elements $a_1\cdots a_n$ and $\tilde a_1\cdots
\tilde a_m$ as above, with 
$a_k\in\cA_{i_k}$ and $\tilde a_l\in \cA_{j_l}$. Then we have
$$E[a_1\cdots a_n \tilde a_m^*\cdots \tilde a_1^*]=
\delta_{nm} E\Bigl[a_1 E\bigl[ a_2\cdots E[a_n\tilde a_n^*]\cdots
\tilde a_2^*\bigr]\tilde a_1^*\Bigr],
$$
which is only different from 0 if $i_k=j_k$ for all $k=1,\dots,n$.
\newline
As an example for the derivation of the above formula consider $n=m=3$:
\begin{align*}
E[a_1a_2a_3\tilde a_3^*\tilde a_2^*\tilde a_1^*]
&=E[a_1a_2\cdot ( (a_3\tilde a_3^*)^o +E[a_3\tilde a_3^*])\cdot \tilde a_2^*\tilde 
a_1^*]\\
&=
E[a_1\underbrace{a_2 E[a_3\tilde a_3^*]\tilde a_2^*}_{(\dots)^o+E[\dots]}\tilde 
a_1^*]\\
&=E[a_1 E[a_2 E[a_3\tilde a_3^*]\tilde a_2^*]\tilde a_1^*].
\end{align*}
Hence, for the calculation of $E[aa^*]$, if suffices to consider $a$ which are sums of products of the same lenght and the same $i$-pattern.
\item[(iii)]
Consider
$$a=\sum_{k=1}^r a_1^{(k)}\cdots a_n^{(k)},\quad
\text{where $n\in\NN_0$, $r\in\NN$, $a_j^{(k)}\in A_{i_j}$ for all $k=1,\dots,r$}$$
with $i_1\not= i_2\not=\cdots \not= i_n$ and $E[a_j^{(k)}]=0$ for all
$j=1,\dots,n$ and $k=1,\dots,r$. Then we have
$$E[aa^*]=\sum_{k,l=1}^r E\Bigl[ a_1^{(k)}\cdots E\bigl[a_{n-1}^{(k)}
E[a_n^{(k)}a_n^{(l)*}]a_{n-1}^{(l)*}\bigr]\cdots a_1^{(l)*}\Bigr].
$$
Now note that $(E[a_n^{(k)}a_n^{(l)*}])_{k,l}$ is a positive matrix in $M_r(\cB)$
since $E$ is completely positive (see Exercise \ref{exercise:9}). But since $\cB$ and thus also $M_r(\cB)$ is a $C^*$-algebra, this means that we can write this positive matrix as $BB^*$ for some $B=(b_{r_n}^{(k)})_{k,r_n=1}^r
\in M_r(\cB)$, which yields then concretely that
$$E[a_n^{(k)}a_n^{(l)*}]=\sum_{r_n=1}^r b_{r_n}^{(k)} b_{r_n}^{(l)*}\qquad
\text{for all $k,l=1,\dots,r$.}$$
Thus we can continue our above calculation as follows
$$E[aa^*]=\sum_{k,l=1}^r
E\bigl[ a_1^{(k)}\cdots \underbrace{E[a_{n-1}^{(k)}\cdot \sum_{r_n=1}^r b_{r_n}^{(k)} b_{r_n}^{(l)*}\cdot a_{n-1}^{(l)*}]}_{
\sum_{r_n=1}^r {E[(a_{n-1}^{(k)} b_{r_n}^{(k)})(a_{n-1}^{(l)}b_{r_n}^{(l)})^*]}}
\cdots a_1^{(l)*}\bigr].
$$
Again, $(E[(a_{n-1}^{(k)} b_{r_n}^{(k)})(a_{n-1}^{(l)}b_{r_n}^{(l)})^*])_{k,l}$ is a positive matrix in $M_r(\cB)$ and its entries can thus be written in the form
$$ E[(a_{n-1}^{(k)} b_{r_n}^{(k)})(a_{n-1}^{(l)}b_{r_n}^{(l)})^*]=\sum_{r_{n-1}=1}^r
b_{r_{n-1},r_n}^{(k)} b_{r_{n-1},r_n}^{(l)*}
$$ 
for some $b_{r_{n-1},r_n}^{(k)}\in\cB$. Iterating this leads finally to
\begin{align*}
E[aa^*]&=\sum_{k,l=1}^r \sum_{r_1=1}^r\dots \sum_{r_n=1}^r
b_{r_1,\dots,r_n}^{(k)} b_{r_1,\dots,r_n}^{(l)*}\\
&=\sum_{r_1=1}^r\dots \sum_{r_n=1}^r
\underbrace{\Bigl( \sum_{k=1}^r b_{r_1,\dots,r_n}^{(k)}\Bigr) \Bigl(\sum_{l=1}^r b_{r_1,\dots,r_n}^{(l)}\Bigr)^*}_{\geq 0}\\
&\geq 0.
\end{align*}
\end{itemize}
\end{proof}

\begin{theorem}\label{thm:5.14}
Let $\cB$ be a unital $C^*$-algebra. Let $\mu_I\in \Sigma_\cB^{I,0}$ be a joint distribution on $\cB\lb x_i; i\in I\rb$ and $\mu_J\in \Sigma_\cB^{J,0}$ be a joint distribution on $\cB\lb y_j; j\in J\rb$, with $I\cap J=\emptyset$. Then there exists a uniquely determined $\mu\in \Sigma_\cB^{I\cup J,0}$ on
$\cB\lb x_i, y_j;i\in I, j\in J\rb$ such that:
\begin{itemize}
\item
$\mu$ restricted to $\cB\lb x_i; i\in I\rb$ is $\mu_I$ and 
$\mu$ restricted to $\cB\lb y_j; j\in J\rb$ is $\mu_J$;
\item
$\cB\lb x_i; i\in I\rb$ and $\cB\lb y_j; j\in J\rb$ are free with respect to $\mu$.
\end{itemize}
We write then $\mu=\mu_I * \mu_J$.
\end{theorem}

\begin{proof}
As a linear map we can define $\mu$ (uniquely!) by the knowledge of $\mu_I$ and $\mu_J$ and the freeness condition, by writing each element in
$\cB\lb x_i, y_j; i\in I, j\in J\rb$ as a linear combination of alternating products of centred elements from $\cB\lb x_i;i\in I\rb$ and from $\cB\lb y_j;j\in J\rb$.
On all such products $\mu$ is set to 0, only on constant terms $b\in\cB$ it is
$\mu(b)=b$.

One has then to check the properties (i)-(iv) from Definition \ref{def:5.5} in order to see that $\mu\in\Sigma_\cB^{I\cup J,0}$.
(i) and (ii) are clear.
(iii) on the base level ist just Proposition \ref{prop:5.13}; that it is also true for the matrix amplifications follows from the same proposition, if we take also
into account that freeness between $\cB\lb x_i; i\in I\rb$ and $\cB\lb y_j; j\in J\rb$ goes also over to matrices, by Proposition \ref{prop:5.7}. For (iv) we have to see that we get also exponential bounds for mixed moments in $x_i$ and $y_j$, if they are free, and if we have such bounds for the $x_i$, $i\in I$, and for the $y_j$, $j\in J$, separateley. We will see this later, when we have developed more theory for the structure of such mixed moments; see Example \ref{ex:9.6}.
\end{proof}

\begin{cor}\label{cor:5.15}
Let $\cB$ be a unital $C^*$-algebra. For each $p=p(x_i;i\in I)\in
\cB\lb x_i; i\in I\rb$ with $p=p^*$ we have a corresponding operation
$p^\square$ on $\Sigma_\cB^0$ given by
\begin{align*}
p^\square:\underbrace{\Sigma_\cB^0\times \cdots\times \Sigma_\cB^0}_{\text{$\vert I\vert$-times}}\to \Sigma_\cB^0,\qquad
(\mu_i)_{i\in I}\mapsto p^\square (\mu_i;i\in I),
\end{align*}
where $p^\square (\mu_i;i\in I)$ is the distribution of $p(x_i; i\in I)$ with respect to $\underset{i\in I}* \mu_i$.
\end{cor}

\begin{remark}
\begin{enumerate}
\item
Note that via matrix amplifications we can also do the same for all selfadjoint
$p\in M_n(\cB\lb x_i;i\in I\rb)$.
\item
$\square$ is the generic symbol for an operation with free variables, to be used with care and imagination; for example, we have the \emph{free convolution}
$\mu_1\boxplus \mu_2$ for $p(x_1,x_2)=x_1+x_2$ and the \emph{free commutator}
$[\mu_1\square \mu_2]$  for $p(x_1,x_2)=x_1x_2+x_2x_1$ or the \emph{free anti-commutator}
$\{\mu_1\square\mu_2\}$ for $p(x_1,x_2)=\frac 1i (x_1x_2-x_2x_1)$.
\item
In the scalar-valued case, $\cB=\CC$, all those operations $p^\square$ are on the level of compactly supported probability measures. In the \href{https://rolandspeicher.files.wordpress.com/2019/08/free-probability.pdf}{Free Probability Lecture Notes} we saw how to deal with $\mu_1\boxplus \mu_2$, but we could not
address general $p^\square$. We will see later that in our operator-valued context we have tools for dealing with such general $p^\square$.
\end{enumerate}
\end{remark}

\chapter{Operator-Valued Free Central Limit Theorem and Operator-Valued Semicircular
Elements}\label{chapter:6}

Our benchmark distribution of free semicircular variables from Section \ref{section:1.3} corresponds on the operator-valued
level to an operator-valued semicircular element. This arises also abstractly in the operator-valued
theory canonically as the limit distribution in a free central limit theorem; furthermore, this operator-valued
distribution has
a very concrete and nice description both on a combinatorial level (via moments) as well as on an
analytic level (via an explicit equation for its operator-valued Cauchy transform).

\section{Operator-valued free central limit theorem}

\begin{remark}
\begin{enumerate}
\item
A central limit theorem asks about the limit distribution of 
$$D_{1/\sqrt N} (\mu\boxplus \cdots \boxplus\mu) \overset {N\to\infty}\longrightarrow\quad ?$$
where $D_{1/\sqrt N}$ denotes dilation by a factor $1/\sqrt N$. In terms of random variables the question can be stated as
$$\frac{X_1+\cdots +X_N}{\sqrt N} \overset {N\to\infty}\longrightarrow\quad ?$$
if $X_i$ are free and identically distributed (f.i.d.).

The relevant convergence is ``in distribution'', which means that moments converge. Since moments are elements in $\cB$, we also have to specify the type
of convergence there -- we will usually take convergence in norm in $\cB$.
\item
The relevant information about the input distribution is the second moment (first moments are assumed to be zero); in the operator-valued case the second
moment is given by a mapping $\eta:\cB\to\cB$ with $\eta(b):=E[XbX]$.

In a $C^*$-setting $E$, and thus also $\eta$, must be completely positive: for 
$(b_{ij})_{i,j=1}^n\in M_n(\cB)$ we have
$$\id\otimes E[1\otimes X\cdot (b_{ij})_{i,j=1}^n\cdot 1\otimes X]=
(\underbrace{E[Xb_{ij}X]}_{\eta(b_{ij})})_{i,j=1}^n=\id\otimes \eta((b_{ij})_{i,j=1}^n),$$
and thus: $\id\otimes\eta(bb^*)=\id\otimes E[(1\otimes X\cdot b)(1\otimes X\cdot b)^*]\geq 0$.

We also have that every completely positive $\eta$ can show up as second moment of a $\mu\in\Sigma_\cB^0$, see Exercise \ref{exercise:16}.
\item
Much of the calculations for the central limit theorem and description of the limit
are similar to the scalar-valued situation (see Chapter 2 of the \href{https://rolandspeicher.files.wordpress.com/2019/08/free-probability.pdf}{Free Probability Lecture Notes}). Let us first check the calculation of the moments in the limit.

We consider $(X_i)_{i\in\NN}$ which are f.i.d.~with respect to $E$. We also assume that
\begin{itemize}
\item
$X_i$ centred: $E[X_i]=0$ for all $i\in\NN$;
\item
second moments are given by $\eta:\cB\to\cB$: $E[X_ibX_i]=\eta(b)$ for
all $i\in \NN$ and all $b\in\cB$.
\end{itemize}
Then we put $S_N:=(X_1+\cdots+ X_N)/\sqrt N$ and calculate its moments.
\begin{align*}
&E[S_Nb_1S_Nb_2\cdots S_N b_{k-1}S_N]
=\frac 1{N^{k/2}} \sum_{i:[k]\to [N]} E[X_{i(1)}b_1X_{i(2)}\cdots X_{i(k-1)}b_{k-1}
X_{i(k)}]\\
&=\frac 1{N^{k/2}} \sum_{\pi\in\cP(k)}\sum_{\substack{i:[k]\to [N]\\
\ker i=\pi}} \underbrace{E[X_{i(1)}b_1X_{i(2)}\cdots X_{i(k-1)}b_{k-1}
X_{i(k)}]}_{\substack{=:g(\pi)\\ \text{depends only on $\ker i$ by Prop. \ref{prop:5.2}}}}\\
&=\frac 1{N^{k/2}} \sum_{\pi\in\cP(k)} g(\pi)\cdot \underbrace{\#\{i:{k}\to [N]\mid \ker i=\pi\}}_{\sim N^{\# \pi}}.
\end{align*}
Now observe that if $\pi$ has a singleton, then $g(\pi)=0$; because we have
$E[X_i]=0$ and by the factorization \eqref{eq:factorization}. This implies then that only $\pi\in\cP(k)$ without singleton contribute; for those we have necessarily
$\# \pi\leq k/2$. Now we have enough information to go to the limit $N\to\infty$. There only $\pi$ with $\#\pi=k/2$ survive; but those have to be
pairings $\pi\in\cP_2(k)$.

If $\pi$ is crossing, then the definition of freeness (and interval stripping) gives
$g(\pi)=0$; here is an example which illustrates this:
\begin{align*}
 g(\,\begin{tikzpicture}[baseline=0pt]
    \begin{scope}[yshift=.8em, scale=.5]
 \draw (0em,-0em) --  (0em,-1em) -- (2em,-1em) -- (2em,-0em);
\draw (1em,-0em) --  (1em,-2em) -- (5em,-2em) -- (5em,-0em);
\draw (3em,-0em) --  (3em,-1em) -- (4em,-1em) -- (4em,-0em);
    \end{scope}
  \end{tikzpicture}  \, )&=
E[\underset{
\begin{tikzpicture}[baseline=0pt]
    \begin{scope}[yshift=0em, scale=2.1]
 \draw (0em,-0em) --  (0em,-.5em) -- (2em,-.5em) -- (2em,-0em);
\draw (1em,-0em) --  (1em,-1em) -- (5em,-1em) -- (5em,-0em);
\draw (3em,-0em) --  (3em,-.5em) -- (4em,-.5em) -- (4em,-0em);
    \end{scope}
  \end{tikzpicture}  
}{X_1b_1X_2b_2X_1b_3X_3b_4X_3b_5X_2]}\\\quad\\&=
E[\underbrace{X_1b_1}\underbrace{X_2b_2}\underbrace{X_1b_3E[X_3b_4X_3]}\underbrace{b_5X_2}]\qquad\text{(alternating and centred)}
\\
&=0
\end{align*}
So we get for our moment in the limit:
$$\lim_{N\to\infty} E[S_Nb_1S_Nb_2\cdots S_N b_{k-1}S_N]=\sum_{\pi\in
NC_2(k)} g(\pi).$$
Up to this point we just repeated the arguments for the scalar-valued case. But now there will be a difference, namely $g(\pi)$ is not the same for all
$\pi\in NC_2(k)$. We have
$$g(
\begin{tikzpicture}[baseline=0pt]
    \begin{scope}[yshift=.5em, scale=.8]
 \draw (0em,-0em) --  (0em,-.5em) -- (1em,-.5em) -- (1em,-0em);
    \end{scope}
  \end{tikzpicture} 
)=E[
\underset{
\begin{tikzpicture}[baseline=0pt]
    \begin{scope}[yshift=0em, scale=2.1]
 \draw (0em,-0em) --  (0em,-.5em) -- (1em,-.5em) -- (1em,-0em);
    \end{scope}
  \end{tikzpicture}  }{
X_1b_1X_1}]=\eta(b_1)$$
\begin{align*}
g(
\begin{tikzpicture}[baseline=0pt]
    \begin{scope}[yshift=0.6em, scale=.6]
 \draw (0em,-0em) --  (0em,-1em) -- (1em,-1em) -- (1em,-0em);
 \draw (2em,-0em) --  (2em,-1em) -- (3em,-1em) -- (3em,-0em);
    \end{scope}
  \end{tikzpicture}
)=E[\underset{
\begin{tikzpicture}[baseline=0pt]
    \begin{scope}[yshift=0em, scale=2.1]
 \draw (0em,-0em) --  (0em,-.5em) -- (1em,-.5em) -- (1em,-0em);
 \draw (2em,-0em) --  (2em,-.5em) -- (3em,-.5em) -- (3em,-0em);
    \end{scope}
  \end{tikzpicture}  }{
X_1b_1X_1b_2X_2b_3X_2}]
=E[\underbrace{E[X_1b_1X_1]}_{\eta(b_1)}b_2 \underbrace{E[X_2b_3X_2]}_{\eta(b_3)}]
=\eta(b_1)b_2\eta(b_3)
\end{align*}

\begin{align*}
g(
\begin{tikzpicture}[baseline=0pt]
    \begin{scope}[yshift=0.6em, scale=.6]
 \draw (1em,-0em) --  (1em,-.5em) -- (2em,-.5em) -- (2em,-0em);
 \draw (0em,-0em) --  (0em,-1em) -- (3em,-1em) -- (3em,-0em);
    \end{scope}
  \end{tikzpicture}
)=E[\underset{
\begin{tikzpicture}[baseline=0pt]
    \begin{scope}[yshift=0em, scale=2.1]
 \draw (1em,-0em) --  (1em,-.5em) -- (2em,-.5em) -- (2em,-0em);
 \draw (0em,-0em) --  (0em,-1em) -- (3em,-1em) -- (3em,-0em);
    \end{scope}
  \end{tikzpicture}  }{
X_1b_1X_2b_2X_2b_3X_1}]
=E[X_1b_1\underbrace{E[X_2b_2X_2]}_{\eta(b_2)}b_3X_1]
=\eta(b_1\eta(b_2)b_3)
\end{align*}
Thus the limit variable $S$ has moments
$$
E[Sb_1S]=\eta(b_1),\qquad
E[Sb_1Sb_2Sb_3S]=\eta(b_1)b_2\eta(b_3)+\eta(b_1\eta(b_2)b_3)$$
and in general
$$E[Sb_1S\cdots Sb_{k-1}S]=\sum_{\pi\in NC_2(k)} \eta_\pi(b_1,\dots,b_{k-1}),$$
where $\eta_\pi:\cB^{k-1}\to\cB $ is the $\CC$-multilinear map given by
$$\eta_\pi(b_1,\dots,b_{k-1})=E_\pi[X_ib_1,X_ib_2,\dots,X_ib_{k-1},X_i].$$
\item
Note that even in the case where all $b_1,\dots,b_{k-1}$ are equal to 1, the 
contributions of the $\eta_\pi(1,1,\dots,1)$ are different in general.
\begin{align*}
\eta_{\,\begin{tikzpicture}[baseline=0pt]
    \begin{scope}[yshift=0.4em, scale=.6]
 \draw (0em,-0em) --  (0em,-1em) -- (1em,-1em) -- (1em,-0em);
 \draw (2em,-0em) --  (2em,-1em) -- (3em,-1em) -- (3em,-0em);
    \end{scope}
  \end{tikzpicture}\,}(1,1,1)=\eta(1)\cdot 1\cdot \eta(1)=\eta(1)^2,
\quad
\eta_{\,
\begin{tikzpicture}[baseline=0pt]
    \begin{scope}[yshift=0.4em, scale=.6]
 \draw (1em,-0em) --  (1em,-.5em) -- (2em,-.5em) -- (2em,-0em);
 \draw (0em,-0em) --  (0em,-1em) -- (3em,-1em) -- (3em,-0em);
    \end{scope}
  \end{tikzpicture}\/}(1,1,1)=\eta(1\cdot\eta(1)\cdot 1)=\eta(\eta(1)).
\end{align*}
Note that $\eta$ does not need to be unital: $\eta(1)\not=1$ in general.
\end{enumerate}
\end{remark}

Let us collect our observations in the following theorem.

\begin{theoremanddefinition}\label{thm:6.2}
Let $(\cA,\cB,E)$ be a $\cB$-valued $C^*$-probability space. Consider
selfadjoint $X_i\in\cA$, $i\in \NN$, which are f.~i.~d. (free and identically
distributed) with
\begin{itemize}
\item
$E[X_i]=0$ for all $i\in\NN$;
\item $E[X_ib X_i]=\eta(b)$ for all $i\in\NN$ and $b\in\cB$, for a completely
positive $\eta:\cB\to\cB$.
\end{itemize}
Put $S_N:=(X_1+\cdots +X_N)/\sqrt N$.
Then $\mu_{S_N}$ converges in distribution for $N\to\infty$ to $\nu_\eta\in
\Sigma_\cB^0$, which is given by 
\begin{equation}\label{eq:nueta}
\nu_\eta(b_0xb_1\cdots b_{k-1}xb_k)=\sum_{\pi\in NC_2(k)}
b_0\eta_\pi(b_1,\dots,b_{k-1})b_k
\end{equation}
for $k\in\NN$ and $b_0,\dots,b_k\in \cB$. In particular, this says that all odd moments are zero. 

Such a distribution $\nu_\eta\in \Sigma_\cB^0$, given by \eqref{eq:nueta}, is called \emph{$\cB$-valued semicircular distribution}, with covariance $\eta$.
A selfadjoint element $S$ with $\mu_S=\nu_\eta$ is called \emph{($\cB$-valued) semicircular element}.
\end{theoremanddefinition}

\section{Some basic properties of operator-valued semicircular elements}

\begin{remark}\label{rem:6.3}
\begin{enumerate}
\item
Note that this definition is compatible with amplifications: if $S$ is a semicircular element in $(\cA,\cB,E)$ with covariance $\eta:\cB\to\cB$, then
$1\otimes S$ is a semicircular element in $(M_n(\cA),M_n(\cB),\id\otimes E)$ with covariance $\id\otimes\eta:M_n(\cB)\to M_n(\cB)$. For a more general version of this see also Exercise \ref{exercise:17}.
\item
Let us check that indeed $\nu_\eta\in\Sigma_\cB^0$, i.e., that we have positivity and exponential boundedness. On can do this by contructing bounded operators on the full Fock space which have $\nu_\eta$ as distribution (for this see Exercise \ref{exercise:19}). We do it here more abstractly.
\begin{enumerate}
\item
positivity

Since positivity is preserved in a central limit, we only need a distribution $\mu\in\Sigma_\cB^0$ which has first moment zero and second moment given by $\eta$. In Exercise \ref{exercise:16} we construct such a distribution, an operator-valued Bernoulli element.
\item
exponential boundedness

We have to estimate the norm of 
$$\nu_\eta(xb_1\cdots b_{k-1}x)=\sum_{\pi\in NC_2(k)}
\eta_\pi(b_1,\dots,b_{k-1})$$
for $k=2m$ even. Note first that $\eta$ as a positive map is bounded, i.e.,
$$\Vert \eta(b)\Vert\leq \Vert\eta\Vert\cdot\Vert b\Vert\qquad
\text{for all $b\in\cB$,}\qquad\text{where $\Vert\eta\Vert<\infty$.}$$
This implies that we have for each $\pi\in NC_2(2m)$
$$\Vert \eta_\pi(b_1,\dots,b_{k-1})\Vert \leq \Vert\eta\Vert^m\cdot\Vert b_1\Vert\cdots \Vert b_{2m-1}\Vert.$$
As an illustration for this let us have a look on the estimates for the two contributions of order 4:
\begin{align*}
\eta_{\,\begin{tikzpicture}[baseline=0pt]
    \begin{scope}[yshift=0.4em, scale=.6]
 \draw (0em,-0em) --  (0em,-1em) -- (1em,-1em) -- (1em,-0em);
 \draw (2em,-0em) --  (2em,-1em) -- (3em,-1em) -- (3em,-0em);
    \end{scope}
  \end{tikzpicture}\,}(b_1,b_2,b_3)&=\Vert \eta(b_1)\cdot b_2\cdot \eta(b_3)\Vert\\
&\leq
\Vert\eta(b_1)\Vert\cdot\Vert b_2\Vert\cdot \Vert \eta(b_3)\Vert\\
&\leq \Vert\eta\Vert^2\cdot\Vert b_1\Vert\cdot \Vert b_2\Vert\cdot\Vert b_3\Vert
\end{align*}
and
\begin{align*}
\eta_{\,
\begin{tikzpicture}[baseline=0pt]
    \begin{scope}[yshift=0.4em, scale=.6]
 \draw (1em,-0em) --  (1em,-.5em) -- (2em,-.5em) -- (2em,-0em);
 \draw (0em,-0em) --  (0em,-1em) -- (3em,-1em) -- (3em,-0em);
    \end{scope}
  \end{tikzpicture}\/}(b_1,b_2,b_3)&=\Vert\eta(b_1\cdot\eta(b_2)\cdot b_3)\Vert\\&\leq
\Vert\eta\Vert\cdot \Vert b_1\cdot\eta(b_2)\cdot b_3\Vert\\
&\leq \Vert\eta\Vert \cdot\Vert b_1\Vert\cdot \Vert\eta(b_2)\Vert\cdot \Vert b_3\Vert\\
&\leq \Vert\eta\Vert^2\cdot\Vert b_1\Vert\cdot \Vert b_2\Vert\cdot\Vert b_3\Vert
\end{align*}
Thus -- by also using the fact that the number of elements of $NC_2(2m)$ is given by the $m$-th Catalan number, which is smaller than $4^m$ -- we can now get our exponential bound:
\begin{align*}
\Vert\nu_\eta(xb_1\cdots b_{k-1}x)\Vert &\leq \# NC_2(2m)\cdot\Vert \eta\Vert^m\cdot\Vert b_1\Vert\cdots \Vert b_{2m-1}\Vert\\[0.5em]
&\leq \underbrace{2^{2m}\Vert\eta\Vert^m}_{(2\Vert\eta\Vert)^{2m}}\cdot\Vert b_1\Vert\cdots \Vert b_{2m-1}\Vert.
\end{align*}
\end{enumerate}
\item\label{rem:6.3.3}
In (1) we said that if $S$ is $\cB$-valued semicircular, then
$$1\otimes S=\begin{pmatrix}
S&0&\dots&0\\
0&S&\hdots&0\\
\vdots&\vdots&\ddots&\vdots\\
0&0&\hdots&S\end{pmatrix}$$
is also semicircular, over $M_m(\cB)$. This is true more general; if we have free semicircular elements over $\cB$ and put linear combinations of them as entries into an selfadjoint $m\times m$-matrix, then this is an $M_m(\cB)$-valued semicircular element. 
The proof can be done by using our free central limit theorem. Let us elaborate on this via the example
$$S=\begin{pmatrix}
0&S_1\\
S_1& S_2
\end{pmatrix},$$
where $S_1$ and $S_2$ are free and semicircular over $\cB$, with covariances $\eta_1$ and $\eta_2$, respectively. Then we can realize $S_1$ and $S_2$ as
$$S_1=\lim_{N\to\infty}\frac {X_1+\cdots +X_N}{\sqrt N},\qquad
S_2=\lim_{N\to\infty}\frac {Y_1+\cdots +Y_N}{\sqrt N},$$
where all $X_i,Y_j$ are free and $E[X_i]=0=E[Y_j]$, $E[X_ibX_i]=\eta_1(b)$,
$E[Y_jbY_j]=\eta_2(b)$. This gives us for $S$ the realization
\begin{align*}
S&=\lim_{N\to\infty}\begin{pmatrix}
0&{(X_1+\cdots +X_N)}/{\sqrt N}\\
(X_1+\cdots +X_N)/{\sqrt N}& (Y_1+\cdots +Y_N)/{\sqrt N}\end{pmatrix}\\[0.5em]
&=\lim_{N\to\infty} \frac 1{\sqrt N}\left[\begin{pmatrix}
0& X_1\\
X_1& Y_1\end{pmatrix}
+\cdots+
\begin{pmatrix}
0& X_N\\
X_N& Y_N\end{pmatrix}\right]
\end{align*}
The summands in the last sum are f.~i.~d. with respect to $\id\otimes E$ with vanishing first moment, and thus, by our central limit theorem, $S$ is an $M_2$-valued semicircular element. Its variance $\eta$ is given by the second moment
\begin{align*}
\eta\begin{pmatrix}
b_{11}&b_{12}\\
b_{21}& b_{22}\end{pmatrix}&=
\id\otimes E \left[
\begin{pmatrix}
0&S_1\\ S_1& S_2 \end{pmatrix}
\begin{pmatrix}
b_{11}&b_{12}\\
b_{21}& b_{22}\end{pmatrix}
\begin{pmatrix}
0& S_1\\ S_1& S_2 \end{pmatrix}\right]\\[0.5em]
&=\id\otimes E\begin{pmatrix}
S_1b_{22}S_1& S_1b_{21}S_1+S_1b_{22}S_2\\
S_1b_{12}S_1+S_2b_{22}S_1& S_1b_{11}S_1+S_2b_{21}S_1+S_1b_{12}S_2+S_2b_{22}S_2
\end{pmatrix}\\[0.5em]
&=\begin{pmatrix}
\eta_1(b_{22})& \eta_1(b_{21})\\
\eta_1(b_{12})& \eta_1(b_{11})+ \eta_2(b_{22})
\end{pmatrix}
\end{align*}

\end{enumerate}
\end{remark}

\section{Equation for the Cauchy transform of the semicircle}

\begin{remark}
In order to derive an equation for the Cauchy transform of $\nu_\eta$ we are
looking for recursions among the moments. Consider, with $\mu_S=\nu_\eta$,
$$E[Sb_1Sb_2\cdots b_{2m-1}S]=\sum_{\pi\in NC_2(2m)} \eta_\pi(b_1,\dots,b_{2m-1}).$$
We write $\pi\in NC_2(2m)$ in the form
$\pi=(1,l)\cup \pi_1\cup \pi_2$, where necessarily $l=2k$ even.
 $$ \begin{tikzpicture}
    \node[circle, fill=white, draw=black, scale=0.4, label=above:{$1$}] (n1) at (0,0) {};
    \node[circle, fill=white, draw=black, scale=0.4] (n2) at (2,0) {};
    \node[circle, fill=white, draw=black, scale=0.4, label=above:{$l$}] (n3) at (2.5,0) {};
    \node[circle, fill=white, draw=black, scale=0.4] (n4) at (3,0) {};
    \node[circle, fill=white, draw=black, scale=0.4, label=above:{$2m$}] (n5) at (5,0) {};
    \path (n1) -- node [pos=0.5] {$\ldots$} (n2);
    \path (n4) -- node [pos=0.5] {$\ldots$} (n5);
    \draw[gray, dashed] ($(n1)+(0.2,-0.15)$)  rectangle ($(n2)+(0.2,0.15)$);
    \draw[gray, dashed] ($(n4)+(-0.2,-0.15)$)  rectangle ($(n5)+(0.2,0.15)$);
    \draw (n1) -- ++ (0,-1) -| (n3);
    \node[align=left] at (1.5,-.5) {$\pi_1$};
 \node[align=left] at (4,-.5) {$\pi_2$};
    \node at ($(n1)+(-1.5em,0em)$) {$\pi$};
  \end{tikzpicture}
$$

In this parametrization we can express $\eta_\pi$ as
\begin{align*}
\eta_\pi(b_1,\dots,b_{2m-1})&=
E_\pi[\underset{
\begin{tikzpicture}
\begin{scope}[xshift=-12.5em, scale=2.9]
    \draw (0,0) -- (0,-.7em) -- (2em,-.7em) -- (2em,0);
  \node at (1em,-.3em) {$\pi_1$};
\node at (3em,-.3em) {$\pi_2$};
\end{scope}
  \end{tikzpicture}}
{Sb_1,\dots Sb_{l-1}, S b_l,S\dots b_{2m-1},S}]\\[0.5em]
&=\eta\bigl(E_{\pi_1}[b_1S,\dots, Sb_{l-1}]\bigr)\cdot 
E_{\pi_2}[b_lS,\dots ,b_{2m-1}S]
\end{align*}
Thus
\begin{align*}
&E[Sb_1\cdots b_{2m-1}S]=\sum_{k=1}^m\sum_{\substack{\pi_1\in
NC_2(2(k-1))\\ \pi_2\in NC_2(2(m-k))}}
\eta\bigl(E_{\pi_1}[b_1S,\dots ,Sb_{2k-1}]\bigr)\cdot 
E_{\pi_2}[b_{2k}S,\dots ,b_{2m-1}S]\\
&=
\sum_{k=1}^m \eta\bigl(\sum_{\pi_1\in NC_2(2(k-2))}
E_{\pi_1}[b_1S,\dots, Sb_{2k-1}]\bigr)
\cdot 
\sum_{\pi_2\in NC_2(2(m-k))}
E_{\pi_2}[b_{2k}S,\dots, b_{2m-1}S]\\
&=\sum_{k=1}^m \eta\bigl( E[b_1S\cdots Sb_{2k-1}]\bigr)
\cdot E[b_{2k}S\cdots S b_{2m-1} S].
\end{align*}
Consider now the operator-valued Cauchy transform (on the base level)
\begin{align*}
G:=G_S:H^+(\cB)\to H^-(\cB);\qquad
z\mapsto G(z)=E[(z-S)^{-1}].
\end{align*}
For large $\Vert z\Vert$ we have
\begin{align*}
G(z)&= z^{-1}\sum_{m\geq 0} E[(Sz^{-1})^{2m}]\\
&=z^{-1}+z^{-1}\sum_{m\geq 1} {E[(Sz^{-1})^{2m}]}\\
&=z^{-1}+z^{-1}\sum_{m\geq 1}
{\sum_{k=1}^m
\eta\bigl(z^{-1} E[(Sz^{-1})^{2(k-1)}]\bigr)\cdot z^{-1}  E[(Sz^{-1})^{2(m-k)}]}
\\
&=z^{-1} +z^{-1}\cdot\eta\bigl(\sum_{k=1}^\infty 
z^{-1} E[(Sz^{-1})^{2(k-1)}]\bigr)\cdot\bigl(\sum_{m-k=0}^\infty z^{-1} \cdot E[(Sz^{-1})^{2(m-k)}]\bigr)\\
&=z^{-1}+z^{-1}\cdot \eta(G(z))\cdot G(z),
\end{align*}
or equivalently
\begin{equation}\label{eq:semicirc}
zG(z)=1+\eta(G(z))\cdot G(z).
\end{equation}
So we conclude: $G(z)$ satisfies for large $\Vert z\Vert$ the Equation \eqref{eq:semicirc}; by analytic extension it must then satisfy \eqref{eq:semicirc} also for all $H^+(\cB)$.

The same calculation and arguments work also for all matricial amplifications of $G$.
\end{remark}

\section{Solution of the equation for the semicircle}

\begin{remark}
\begin{enumerate}
\item
In the case $\cB=\CC$ and the normalization $\eta(z)=z$ ($z\in\CC$) -- 
corresponding to $\ff(S^2)=1$ -- we get the quadratic equation for the 
Cauchy transform $G_S:H^+(\CC)\to H^-(\CC)$ of  a scalar-valued
semicircle:
\begin{equation}\label{eq:semicirc-scalar}
zG(z)=1+G(z)^2.
\end{equation}
This can be solved explicitly as
$$G(z)=\frac {z\pm \sqrt{z^2-4}}2,$$
where we have to choose the ``$-$'' sign, since we have for Cauchy
transforms $\lim_{y\to\infty} iy G(iy)=1$; see Remark \ref{rem:4.12}.

From this explicit form for the Cauchy transform one can derive then via the Stieltjes inversion formula the semicircle density.

Note that of the two solutions of \eqref{eq:semicirc-scalar} only one, namely $G(z)$, lies in the right space $H^-(\CC)$, the other solution is in $H^+(\CC)$.
\item
How can we deal with \eqref{eq:semicirc} for general $\cB$ and $\eta$. Note first that \eqref{eq:semicirc} is, in the case $\cB=M_n(\CC)$, actually a system of quadratic equations for the entries of the $n\times n$-matrix $G(z)$. There are no explicit solutions nor a general theory for such systems.
\item
Usually there can be many solutions of such equations; we are, however, interested in a solution which lies in $H^-(\cB)$. To get an idea, consider the very
simple example: $\cB=M_2(\CC)$, $\eta=\id$, 
$z=
z_1\oplus z_2$
with $z_1,z_2\in\CC$ and we are just looking for solutions of the form
$G(z)=w=
w_1\oplus w_2$
with $w_1,w_2\in\CC$. Then \eqref{eq:semicirc} decouples into
$$z_1w_1=1+w_1^2,\qquad z_2w_2=1+w_2^2.$$
Hence we have two solutions for $w_1$ and two solutions for $w_2$:
$$w_1^\pm=\frac {z_1\pm \sqrt{z_1^2-4}}2,\qquad
w_2^\pm=\frac {z_2\pm \sqrt{z_2^2-4}}2.$$
This yields four possible solutions for $w$, of which only $w_1^-\oplus w_2^-$ is in $H^-(M_2(\CC))$. 

We want to show that this is true in general: of the many possible solutions there is exactly one in $H^-(\cB)$.
\item
The idea to see this is to rewrite the Equation \eqref{eq:semicirc} as a fixed point equation:
\begin{align*}
zG(z)=1+\eta(G(z))\cdot G(z) &\quad\Leftrightarrow\quad
z=G(z)^{-1}+\eta(G(z))\\[.4em]
&\quad\Leftrightarrow\quad
G(z)=[z-\eta(G(z))]^{-1},
\end{align*}
i.e., with 
$F_z: w\mapsto [z-\eta(w)]^{-1}$
we have that $G(z)$ is a fixed point of $F_z$.

To see the existence and uniqueness of the fixed point, $F_z$ should be a 
contraction. For large $z$ (i.e., small $\Vert z^{-1}\Vert$) this is true in
operator norm. For general $z\in H^+(\cB)$ the operator norm does not
work any more, but one gets a contraction in an ``analytic'' metric. The following is a kind of general version of the Schwarz Lemma or Denjoy-Wolff Theorem (for the lattter, see 5.6 and Assignment 9 of \href{https://rolandspeicher.files.wordpress.com/2019/08/free-probability.pdf}{Free Probability Lecture Notes}). See also \cite{Har,Kra} for nice expositions around the Earle--Hamilton Theorem.
\end{enumerate}
\end{remark}

\begin{theorem}[Earle, Hamilton 1968]\label{thm:Earle-Hamilton}
Let $D$ be a non-empty domain in a complex Banach space $X$ and let
$h:D\to D$ be a bounded holomorphic function. If $h(D)$ lies strictly inside $D$ -- i.e., there is some $\ee>0$ such that $B_\ee(h(x))\subset D$ whenever $x
\in D$ -- then $h$ is a strict contraction in some (namely, Carath\'eodory-Riffen-Finsler) metric $\rho$, and thus has a unique fixed point in $D$. Furthermore, there is a constant $m>0$ sucht that one has for all $x,y\in D$ that $\rho(x,y)\geq m \Vert x-y\Vert$, and thus $(h^n(x_0))_{n\in\NN}$ converges also in norm, for any $x_0\in D$, to this fixed point.
\end{theorem}

We want to apply this to our fixed point equation for the semicircular distribution. In the next proposition we check that the assumptions of the Earle--Hamilton Theorem are satisfied in this case. We follow here the original work of Helton, Rashidi Far, Speicher \cite{HRS}.

\begin{prop}
Let $\cB$ be a unital $C^*$-algebra and $\eta:\cB\to\cB$ a positive linear map.
For fixed $z\in H^+(\cB)$ we define the map
$$F_z: w\mapsto F_z(w):= [z-\eta(w)]^{-1}.$$
Then we have
\begin{itemize}
\item[(i)]
$F_z:H^-(\cB)\to H^-(\cB)$.
\item[(ii)]
$F_z$ is bounded, with
$$\Vert F_z(w)\Vert \leq \Vert (\Im z)^{-1}\Vert\qquad\text{for all $w\in H^-(\cB)$.}$$
\item[(iii)]
For $R>0$ we put
$$H^-_R(\cB):=\{w\in H^-(\cB)\mid \Vert w\Vert<R\}.$$
Then, for $R> \Vert( \Im z)^{-1}\Vert$, we have that $F_z(H_R^-(\cB))$ lies
strictly inside $H_R^-(\cB)$.
\end{itemize}
\end{prop}

\begin{proof}

\begin{itemize}
\item[(i)]
For $w\in H^-(\cB)$ we have, by the positivity of $\eta$, that $\eta(w)\in H^-(\cB)$, and thus $-\eta(w)\in H^+(\cB)$. But then we have, for $z\in H^+(\cB)$, that also $z-\eta(w)\in H^+(\cB)$. Taking the inverse moves us then into
$H^-(\cB)$.
\item[(ii)]
In the proof of Theorem \ref{thm:4.5} we have seen (put $X=0$ there) that
$\Vert z^{-1}\Vert\leq \Vert (\Im z)^{-1}\Vert$ for $z\in H^+(\cB)$, and thus also
$$\Vert [z-\eta(w)]^{-1}\Vert\leq \Vert [\Im (z-\eta(w))]^{-1}\Vert.$$
In order to estimate this further, note
$$\Im (z-\eta(w))=\Im z - \underbrace{\Im \eta(w)}_{\leq 0}\geq \Im z>0,$$
which implies
$$0<[\Im (z-\eta(w))]^{-1}\leq (\Im z)^{-1},$$
and thus finally
$$\Vert [\Im (z-\eta(w))]^{-1}\Vert\leq \Vert(\Im z)^{-1} \Vert.$$
\item[(iii)]
Note that (ii) shows that for $R>\Vert (\Im z)^{-1}\Vert$ we have
$F_z:H_R^-(\cB)\to H_R^-(\cB)$. We have to see that $F_z(w)$ stays away from the boundary of $H_R^-(\cB)$. For the part $\Vert w\Vert =R$ this is clear, there it stays away at least by an amount $R-\Vert (\Im z)^{-1}\Vert$. In order to see that it also stays away from the ``real axis'' we need an estimate for $\Im F_z(w)$, uniform in $w\in H_R^-(\cB)$. We have
\begin{align*}
\Im F_z(w)&=\frac 1{2i}[F_z(w)-F_z(w)^*]\\[0.5em]
&=F_z(w)^*\underbrace{\left[\frac{F_z(w)^{*-1}-F_z(w)^{-1}}{2i}\right]}_{=\Im(z^*-\eta(w)^*)\leq \Im z^*} F_z(w)\\[0.5em]
&\leq F_z(w)^*\cdot \Im z^* \cdot F_z(w)\\[0.5em]
&= -F_z(w)^*\cdot \Im z\cdot F_z(w).
\end{align*}
Let us write the  last term, without the minus-sign, in the form
\begin{align*}
F_z(w)^*\cdot \Im z\cdot F_z(w)=
\left[ F_z(w)^{-1}\cdot (\Im z)^{-1}\cdot F_z(w)^{*-1}\right]^{-1}.
\end{align*}
We estimate now
\begin{align*}
F_z(w)^{-1}\cdot (\Im z)^{-1}\cdot F_z(w)^{*-1}&\leq \Vert F_z(w)^{-1}\Vert^2\cdot \Vert (\Im z)^{-1}\Vert\cdot 1\\[0.5em]
&=\Vert z-\eta(w)\Vert^2\cdot \Vert (\Im z)^{-1}\Vert\cdot 1\\[0.5em]
&\leq (\Vert z\Vert +\Vert\eta\Vert\cdot\Vert w\Vert)^2 \cdot \Vert (\Im z)^{-1}\Vert\cdot 1\\[0.5em]
&\leq (\Vert z\Vert +\Vert\eta\Vert\cdot R)^2 \cdot \Vert (\Im z)^{-1}\Vert\cdot 1.
\end{align*}
and thus, by taking the inverse and by noting that 
$F_z(w)^{-1}\cdot (\Im z)^{-1}\cdot F_z(w)^{*-1}$ is positive:
$$F_z(w)^*\cdot \Im z\cdot F_z(w)\geq \frac 1{(\Vert z\Vert +\Vert\eta\Vert\cdot R)^2 \cdot \Vert (\Im z)^{-1}\Vert}\cdot 1.$$
Putting everything together gives then the wanted estimate
$$\Im F_z(w)\leq - \frac 1{(\Vert z\Vert +\Vert\eta\Vert\cdot R)^2 \cdot \Vert (\Im z)^{-1}\Vert}\cdot 1,$$
which is independent of $w\in H^-(\cB)$.
\end{itemize}
\end{proof}

\begin{theorem}[Helton, Rashidi Far, Speicher 2007]\label{thm:HRS}
Let $\cB$ be a unital $C^*$-algebra and $\eta:\cB\to\cB$ a positive linear map. For fixed $z\in H^-(\cB)$ there exists exactly one solution $w\in H^-(\cB)$ to
\begin{equation}\label{eq:HRS-fixed-point}
zw=1+\eta(w)\cdot w.
\end{equation}
This $w$ is the limit of iterates $w_n=F_z^n(w_0)$ for any $w_0\in H^-(\cB)$.
Furthermore, we have that
$$\Vert w\Vert\leq \Vert (\Im z)^{-1}\Vert\qquad\text{and}\qquad
\Im w\leq 
- \frac 1{\bigl(\Vert z\Vert +\Vert\eta\Vert\cdot \Vert (\Im z)^{-1}\Vert\bigr)^2 \cdot \Vert (\Im z)^{-1}\Vert}\cdot 1.$$
\end{theorem}

\begin{proof}
By the Earle--Hamilton Theorem \ref{thm:Earle-Hamilton}, each $H_R^-(\cB)$ contains, for $R> \Vert (\Im z)^{-1}\Vert$, exactly one fixed point of $F_z$, i.e., a solution to \eqref{eq:HRS-fixed-point}. (Note that our map $F_z$ is holomorphic.) For any $w_0\in H^-(\cB)$ we choose $R$ such that $w_0\in H_R^-(\cB)$ (i.e., $R> \Vert w_0\Vert$), then Earle--Hamilton guarantees that
$F_z^n(w_0)$ converges in $H_R^-(\cB)$ to $w$.
\end{proof}

\begin{remark}
\begin{enumerate}
\item
Clearly, this solution $w$ from Theorem \ref{thm:HRS} must be the value $G(z)$ of the Cauchy transform of our operator-valued semicircular element $S$ with covariance $\eta$.
\item
The linearity of $\eta$ is not essential for the arguments; one can generalize Theorem \ref{thm:HRS} in the same way to the case where $\eta:H^+(\cB)\to
H^+(\cB)$ is an analytic and bounded map.
\item
The theorem does not give estimates for the speed of convergence. In particular, for small $\Im z$, the convergence can be very slow. One can usually improve 
this by taking averages of the iterates. For example, replace $w\mapsto F_z(w)$ by
$w\mapsto G_z(w):=\frac 12 w+\frac 12 F_z(w)$. $G_z$ has the same fixed point
as $F_z$ and maps $H_R^-(\cB)$ strictly into its interior. Thus, by Earle--Hamilton, sequences $(G_z^n(w_0))_{n\in\NN}$ converge also (and usually faster) to the wanted fixed point of $F_z$.
\end{enumerate}
\end{remark}

\chapter{Matrices of Semicirculars and Matrix-Valued Semicirculars (and Block Random Matrices)}

Here we want to be a bit more concrete about the relation between matrices of free semicirculars and matrix-valued semicircular elements. 
We will here also encounter the idea that we can consider our matrices both as scalar-valued and as operator-valued elements. Understanding the relation between these two points of view will be crucial for applications of operator-valued free probability to random matrix models with some more structure, like block matrices.

\section{Matrix-valued semicirculars}

\begin{remark}
In Remark \ref{rem:6.3} we have seen that matrices of free semicirculars are
matrix-valued semicirculars. We restrict here to the special case where $\cB=\CC$, i.e., the entries of our matrices are scalar-valued free semicirculars. Let us first give the precise statement for this.
\end{remark}

\begin{prop}\label{prop:7.2}
Let $(\cA,\ff)$ be a $C^*$-probability space and $S_1,\dots,S_d$ be free standard semicirculars (i.e., $\ff(S_i^2)=1$). For $n\geq 1$ and selfadjoint
$b_1,\dots,b_d\in M_n(\CC)$ we consider
$$S:=b_1\otimes S_1+\cdots +b_d\otimes S_d\in M_n(\CC)\otimes\cA\,\hat=\, M_n(\cA).$$
Then $S$ is in the matrix-valued $C^*$-probability space $(M_n(\cA),M_n(\CC),
\id\otimes\ff)$ a matrix-valued semicircular element with covariance
$$\eta:M_n(\CC)\to M_n(\CC)\qquad \text{given by}\qquad
\eta(b)=\sum_{j=1}^d b_jbb_j.$$
\end{prop}

The proof of this is an assignment, Exercise \ref{exercise:21}.

\section{Treating matrix-valued semicirculars as scalar-valued variables}

\begin{remark}
\begin{enumerate}
\item
We are now, however, interested in $S$ as a scalar-valued random variable in the $C^*$-probability space $(M_n(\cA),\tr\otimes \ff)$, i.e., instead of the 
operator-valued Cauchy transform
$$G_S:H^+(M_n(\CC))\to H^-(M_n(\CC)),\qquad b\mapsto G_S(b)=\id\otimes\ff [(b-S)^{-1}]$$
we need the scalar-valued Cauchy transform 
$$g_S:H^+(\CC)\to H^-(\CC),\qquad z\mapsto g_s(z)=\tr\otimes\ff [(b-S)^{-1}].$$
Note that for $z\in\CC$ we clearly have
$$g_S(z)=\tr[G_S(z\cdot 1)].$$
So if we can calculate $G_S$, we can from this also get $g_S$.
\item
Note that being semicircular on an operator-valued level does in general not
imply to be semicircular on a scalar-valued level. Let us check this in the next example.
\end{enumerate}
\end{remark}

\begin{example}\label{ex:7.4}
Consider, for $\alpha,\beta\in\RR$,
$$S=\begin{pmatrix}
\alpha S_1&0\\ 0& \beta S_2
\end{pmatrix}
=\begin{pmatrix}
\alpha&0\\ 0&0
\end{pmatrix}\otimes S_1+
\begin{pmatrix}
0&0\\ 0&\beta
\end{pmatrix}\otimes S_2.$$
Then $S$ is for all $\alpha,\beta$ an $M_2(\CC)$-valued semicircular element.
However, on the scalar level we have the second moment
$$\tr\otimes\ff[S^2]=\frac 12\bigl(\alpha^2\ff(S_1^2)+\beta\ff(S_2^2)\bigr)=
\frac 12 (\alpha^2+\beta^2);$$
and if $S$ is semicircular, then its fourth moment must be given by twice the square of this, i.e., by 
$2(\tr\otimes\ff[S^2])^2=(\alpha^2+\beta^2)^2/2$. On the other hand we can calculate the fourth moment directly as 
$$\tr\otimes\ff[S^4]=\frac 12\bigl(\alpha^4\ff(S_1^4)+\beta^4\ff(S_2^4)\bigr)=\alpha^4+\beta^4.$$
But $\alpha^4+\beta^4=(\alpha^2+\beta^2)^2/2$ if and only if $\vert\alpha\vert=\vert \beta\vert$. Thus in general, semicircularity is not preserved; but there are special cases where it is.
\end{example}

\begin{theorem}\label{thm:7.5}
Consider unital $C^*$-algebras $\cD\subset\cB\subset\cA$ with conditional
expectations $E_\cB:\cA\to\cB$ and $E_\cD:\cA\to\cD$ which are compatible in the sense that $E_\cD\circ E_\cB=E_\cD$. Consider a $\cB$-valued semicircular element $S\in\cA$, with covariance $\eta:\cB\to\cB$ with $\eta(b)=E_\cB[SbS]$.
If $\eta(\cD)\subset \cD$, then $S$ is also a $\cD$-valued semicircular element, with covariance given by the restriction of $\eta$ to $\cD$.
\end{theorem}

\begin{example}
Before we prove this let us reconsider Example \ref{ex:7.4}; there $\cD=\CC$,
$\cB=M_2(\CC)$, $E_\cD=\ff$, $E_\cB=\id\otimes \ff$, and $\eta:\cB\to\cB$
is given by
$$\eta\begin{pmatrix}
b_{11}&b_{12}\\ b_{21}& b_{22}
\end{pmatrix}=
\id\otimes\ff
\begin{pmatrix}
b_{11} \alpha^2S_1^2& b_{12}\alpha\beta S_1S_2\\
b_{21} \beta\alpha S_2S_1& b_{22}\beta^2 S_2^2
\end{pmatrix}=
\begin{pmatrix}
\alpha^2 b_{11}& 0\\ 0& \beta^2 b_{22}
\end{pmatrix}.$$
To check that $\eta(\cD)\subset\cD$ we just have to see that $\eta(1)\in\CC$;
but
$$\eta\begin{pmatrix} 1&0\\ 0&1 \end{pmatrix}
=\begin{pmatrix}
\alpha^2&0\\ 0& \beta^2
\end{pmatrix}\in \CC\cdot 1 \qquad\text{if and only if}\qquad \alpha^2=\beta^2.$$
One might note that $\eta$ maps always into diagonal matrices $\tilde\cD$, and thus in this case $S$ is always a $\tilde\cD$-valued semicircular.
\end{example}

For another example of the application of Theorem \ref{thm:7.5} see Exercise \ref{exercise:22}.

\begin{proof}[Proof of Theorem \ref{thm:7.5}]
We have the Cauchy transforms
$$G(b)=E_\cB[(b-S)^{-1}]\qquad\text{for $b\in H^+(\cB)$}$$
and
$$g(d)=E_\cD[(b-S)^{-1}]\qquad\text{for $d\in H^+(\cD)$.}$$
Note that $H^+(\cD)\subset H^+(\cB)$ and that
$$g(d)=E_\cD \underbrace{E_\cB[(d-S)^{-1}]}_{=G(d)}=E_\cD[G(d)].$$
The main claim is to see that
\begin{equation}\label{eq:GdD}
G(d)\in\cD\qquad\text{for all $d\in H^+(\cD)$;}
\end{equation}
then we have that $g(d)=G(d)$ for all $d\in H^+(\cD)$ and the equation
$$b G(b)=1+\eta(G(b))\cdot G(b)\qquad (b\in H^+(\cB))$$
gives for $b=d\in H^+(\cD)$:
$$dg(d)=1+\eta(g(d))\cdot g(d),$$
which shows that $g$ is the Cauchy transform of a $\cD$-valued semicircular element with covariance $\eta\vert_\cD$.

So it remains to prove \eqref{eq:GdD}.  We know, by Theorem \ref{thm:HRS}, that we get $G(d)\in H^-(\cB)$ as the limit of iterates $w_n=F_d^n(w_0)$ for arbitrary $w_0\in H^-(\cB)$, with $F_d(w):=(d-\eta(w))^{-1}$. Now note that since $\eta$ maps $\cD$ to $\cD$, the map $F_d$ also maps $\cD$ to $\cD$; hence if we choose $w_0\in H^-(\cD)\subset H^-(\cB)$  (as we are free to do), all iterates $w_n$, and thus also their limit $G(d)$, are in $\cD$.
\end{proof}

\section{Operator-valued semicirculars as limits of block random matrices} 

\begin{remark}
Note the relevance of this for random matrices. If $X_1^{(N)}, \dots, X_d^{(N)}$ are independent Gaussian $N\times N$ random matrices, then we know (see Chapter 6 of the \href{https://rolandspeicher.files.wordpress.com/2019/08/free-probability.pdf}{Free Probability Lecture Notes}) that for $N\to\infty$
$$(X_1^{(N)},\dots,X_d^{(N)}) \to (S_1,\dots,S_d)$$
in distribution. But this implies that for $N\to\infty$
$$b_1\otimes X_1^{(N)}+\cdots+ b_d\otimes X_d^{(N)}\to S=b_1\otimes S_1+\cdots + b_d\otimes S_d$$
in distribution with respect to $\tr_n\otimes\tr_N$ and $\tr_n\otimes \ff$, respectively. The matrices on the left side are $nN\times nN$ block random matrices, considered as scalar-valued random variables. Thus the scalar-valued distribution of $S$ gives us the asymptotic eigenvalue distribution of the block matrices. See Exercise \ref{exercise:20} for an example of this.
\end{remark}

\chapter{Polynomials in Free Semicirculars and Linearization}

Going over to matrices over a non-commutative algebra gives surprising flexibility in dealing with problems in the algebra. In particular, one can rewrite non-linear problems in the algebra into linear problems in the matrices. This linearization idea has tremenduous impact in our context; it allows to reduce the calculation of polynomials in free variables to the calculation of operator-valued free convolution. We follow here quite closely the presentation in
\cite{HMS,MSp}.

\section{The idea of linearization}
\begin{remark}
\begin{enumerate}
\item
In Proposition \ref{prop:7.2} we saw that we can deal with linear matrices
$$S=b_1\otimes S_1+\cdots + b_d\otimes S_d\qquad (b_1,\dots,b_d\in M_n(\CC))$$
in free semicirculars $S_1,\dots,S_d$. Note that we can also consider ``affine'' matrices by adding a constant $b_0\otimes 1$ $\hat=$ $b_0\in M_n(\CC)$, since this gives only a shift in the argument of the Cauchy transform:
$$G_{b_0+S}(b)=
\id\otimes\ff[(b-(b_0+S))^{-1}]=G_S(b-b_0).$$
Since we consider selfadjoint random variables we need $b_0=b_0^*$ and thus we have
$$\Im (b-b_0)=\Im b\in H^+(M_n(\CC)).$$
So we can calculate $G_S(b-b_0)$ (at least numerically) and from this also the scalar-valued Cauchy transform
$$g_{b_0+S}(z)=\tr [G_{b_0+S}(z\cdot 1)]=\tr [G_S(z\cdot 1-b_0)].$$
\item
In Corollary \ref{cor:5.15} we saw that also for arbitrary selfadjoint polynomials $p\in\CC\lb x_1,\dots, x_d\rb$ the distribution of this polynomial applied to our free semicirculars, $p(S_1,\dots,S_d)$, is uniquely determined; however, up to now it is not clear how to calculate this. We will now see that we can do this by relating this problem with a corresponding problem in affine matrices.
\end{enumerate}
\end{remark}

\begin{example}\label{ex:8.2}
Let us consider the example
$$p(x_1,x_2)=x_1x_2+x_2x_1+x_1^2,\qquad\text{i.e.,}\qquad
P:=p(S_1,S_2)=S_1S_2+S_2S_1+S_1^2.$$
Note that $P=P^*$.
The distribution of $P$ is given by its Cauchy transform
$G_P(z)=\ff[(z-P)^{-1}]$, for $z\in H^+(\CC)$. We lift the problem of calculating the inverse now from the ground level $\CC$ to matrices by finding there a factorization of $P$ into affine terms
$$-P=\begin{pmatrix}
S_1& \frac{S_1}2+S_2
\end{pmatrix}\cdot 
\begin{pmatrix}
0&-1\\ -1&0
\end{pmatrix}\cdot
\begin{pmatrix}
S_1\\ \frac{S_1}2+S_2
\end{pmatrix}.$$
Let us denote
$$U:=\begin{pmatrix}
S_1& \frac{S_1}2+S_2
\end{pmatrix},\quad
Q^{-1}:=\begin{pmatrix}
0&-1\\ -1&0
\end{pmatrix},\quad
V:=\begin{pmatrix}
S_1\\ \frac{S_1}2+S_2
\end{pmatrix},$$
then we have  $P=-UQ^{-1}V$, where $U,Q,V$ are affine in $S_1$ and $S_2$. This does not directly give a factorization for $P^{-1}$, since $U$ and $V$ are
not invertible, but we get a factorization of a lifted version of $z-P$ into invertible factors:
$$\begin{pmatrix}
z-P& 0\\ 0&-Q \end{pmatrix}=
\begin{pmatrix}
1& -UQ^{-1}\\ 0& 1\end{pmatrix}\cdot
\begin{pmatrix}
z& -U\\ -V& -Q \end{pmatrix}
\cdot
\begin{pmatrix}
1&0\\ -Q^{-1}V& 1
\end{pmatrix}.$$
Since the first and third term are always invertible
$$\begin{pmatrix}
1& A\\ 0&1\end{pmatrix}^{-1}=
\begin{pmatrix}
 1&-A\\ 0& 1\end{pmatrix},\qquad
\begin{pmatrix}
1& 0\\ B&1\end{pmatrix}^{-1}=
\begin{pmatrix}
 1&0\\ -B& 1\end{pmatrix},
$$
we have
$$\begin{pmatrix}
(z-P)^{-1}&0\\ 0& -Q^{-1}\end{pmatrix}=
\begin{pmatrix}
z-P&0\\0& -Q\end{pmatrix}^{-1}=
\begin{pmatrix}
1&0\\ Q^{-1}V& 1
\end{pmatrix}\cdot
\begin{pmatrix}
z& -U\\ -V& -Q
\end{pmatrix}^{-1}\cdot
\begin{pmatrix}
1& UQ^{-1}\\0&1\end{pmatrix}.$$
If we put
$$\hat P:=\begin{pmatrix}
0& U\\ V& Q \end{pmatrix},\qquad
\Lambda(z):=\begin{pmatrix}
z&0\\0&0
\end{pmatrix},$$
then we have
$$\begin{pmatrix}
(z-P)^{-1}&0\\ 0& -Q^{-1}\end{pmatrix}=
\begin{pmatrix}
[(\Lambda(z)-\hat P)^{-1}]_{1,1}& *\\ *&*
\end{pmatrix}$$
(where $[A]_{1,1}$ denotes the $(1,1)$-entry of the $3\times 3$-matrix $A$),
i.e.,
$$(z-P)^{-1}=[(\Lambda(z)-\hat P)^{-1}]_{1,1},$$
and thus
$$G_P(z)=\ff[(z-P)^{-1}]=\ff\bigl\{[(\Lambda(z)-\hat P)^{-1}]_{1,1}\bigr\}
=\bigl\{\underbrace{\id\otimes \ff[(\Lambda(z)-\hat P)^{-1}]}_{G_{\hat P}(\Lambda(z))}\bigr\}_{1,1}.
$$
Note that
$$\hat P=
\begin{pmatrix}
0&S_1& \frac {S_1}2+S_2\\
S_1&0& -1\\
\frac{S_1}2+S_2&-1&0
\end{pmatrix}
$$
is an affine matrix in free semicirculars, thus an $M_3(\CC)$-valued semicircular shifted by a constant, for which we can calculate its $M_3(\CC)$-valued Cauchy transform $G_{\hat P}(b)$. Note also that 
$$\Lambda(z)=\begin{pmatrix}
z&0&0\\ 0&0&0\\ 0&0&0 \end{pmatrix}$$
is not in $H^+(M_3(\CC))$, so our theory from Chapter \ref{chapter:6} for solving for $G_{\hat P}(\Lambda(z))$ does not apply directly. But since, by the above calculation, $\Lambda(z)-\hat P$ is invertible, the function $G_{\hat P}$ is holomorphic, hence continuous, in a neighborhood of $b=\Lambda(z)$ and thus we have
$$G_{\hat P}(\Lambda(z))=\lim_{\ee\searrow 0} G_{\hat P}(\Lambda_\ee(z)),\quad
\text{where}\quad
\Lambda_\ee(z):=\underbrace{\begin{pmatrix}
z&0&0\\ 0& i\ee& 0\\ 0&0& i\ee
\end{pmatrix}}_{\substack{\in H^+(M_3(\CC))\\ \text{for all $\ee>0$}}}
\overset{\ee\searrow 0}\longrightarrow \Lambda(z).$$
\end{example}

\section{Rigorous theory of linearization}

\begin{remark}
The main ingredient in the above calculation was that we can factorize our polynomial as $P=-UQ^{-1}V$, with affine $U,Q,V$. This works for all
polynomials and is actually independent from having semicircular elements as variables. Let us now do the general case on the level of formal variables,
$\CC\lb x_1,\dots,x_d\rb$.
\end{remark}

\begin{definition}
Let $p\in\CC\lb x_1,\dots,x_d\rb$ be given. A matrix
$$\hat p=\begin{pmatrix}
0& u\\ v& q
\end{pmatrix}\in M_n(\CC\lb x_1,\dots,x_d\rb),$$
where
\begin{itemize}
\item
$n\in\NN$,
\item
$q\in M_{n-1}(\CC\lb x_1,\dots,x_d\rb)$ is invertible as a matrix over polynomials,
\item
$u\in M_{1,n}(\CC\lb x_1,\dots,x_d\rb)$ is a row and $v\in M_{n,1}(\CC\lb x_1,\dots,x_d\rb)$ is a column,
\end{itemize}
is called a \emph{linearization} of $p$, if the following two conditions are satisfied: 
\begin{enumerate}
\item[(i)]
$\hat p$ is an affine matrix in $x_1,\dots,x_d$, i.e., there are $b_0,b_1,\dots,
b_d\in M_n(\CC)$ such that
$\hat p=b_0\otimes 1+b_1\otimes x_1+\cdots + b_d\otimes x_d$;
\item[(ii)]
$p=-uq^{-1}v$.
\end{enumerate}

\end{definition}

\begin{theorem}\label{thm:8.5}
For any polynomial $p\in \CC\lb x_1,\dots,x_d\rb$ there exists a linearization. If $p$ is selfadjoint, then there is also a selfadjoint linearization.
\end{theorem}

\begin{remark}
Such linearizations are not unique; it is interesting to find minimal ones, where the matrix size $n$ is as small as possible. Our algorithm in the following proof will not produce minimal realizations in general.
\end{remark}

\begin{proof}[Proof of Theorem \ref{thm:8.5}]
\begin{enumerate}
\item
For monomials we have linearizations:
\begin{enumerate}
\item
for degree 0, $p=\alpha$ ($\alpha\in\CC$)
$$\hat p=\begin{pmatrix}
0&\alpha\\
1&-1\end{pmatrix}\in M_2(\CC);\qquad \text{as\quad $\alpha=-\alpha\cdot (-1)^{-1}\cdot 1$;}$$
\item
for degree 1, $p=\alpha x_i$
$$\hat p=\begin{pmatrix}
0& \alpha x_i\\ 1&-1\end{pmatrix}\in M_2(\CC);\qquad\text{as\quad $\alpha x_i=
-\alpha x_i\cdot (-1)^{-1}\cdot 1$;}$$
\item
for degree $k\geq 2$, $p=\alpha x_{i_1}\cdots x_{i_k}$
$$\hat p=
\begin{pmatrix}
0& 0&\dots&0&0& \alpha x_{i_1}\\
0&0& \dots&0& x_{i_2}& -1\\
0&0&\dots &x_{i_3}&-1&0\\
\vdots&\vdots&\udots&\udots&\vdots&\vdots\\
0& x_{i_{k-1}}&\udots&0&0&0\\
x_{i_k}&-1&\dots&0&0&0
\end{pmatrix};$$
to be concrete, consider for example $k=3$:
$$\begin{pmatrix}
0& 0& \alpha x_{i_1}\\
0& x_{i_2}& -1\\
x_{i_3}& -1&0
\end{pmatrix}$$
note that matrices corresponding to $q$ are always invertible
$$\begin{pmatrix}
x_{i_2}& -1\\ -1&0 \end{pmatrix}^{-1}
=\begin{pmatrix} 0&-1\\ -1& - x_{i_2}
\end{pmatrix}$$
and we have
$$- \begin{pmatrix}
0&\alpha x_{i_1}\end{pmatrix}\cdot
\begin{pmatrix} 0&-1\\ -1& - x_{i_2}
\end{pmatrix}\cdot
\begin{pmatrix}
0\\ x_{i_3}\end{pmatrix}=
\alpha x_{i_1}x_{i_2} x_{i_3}.$$
\end{enumerate}
\item
If we have linearizations
$\begin{pmatrix} 
0& u_i\\ v_i& q_i
\end{pmatrix}$
for polynomials $p_i$ ($i=1,\dots, r$), then their sum $p_1+\dots + p_r$ has a
linearization
$$\begin{pmatrix}
0&u_1& u_2&\dots& u_r\\
v_1&q_1&0&\dots&0\\
v_2&0&q_2&\dots&0\\
\vdots&\vdots&\vdots&\ddots&\vdots\\
v_r&0&0&\dots&q_r
\end{pmatrix}.$$
Thus we can build linearizations for any polynomial out of linearizations for its monomials.
\item
If $p$ has a linearization 
$\begin{pmatrix} 
0& u\\ v& q
\end{pmatrix}$,
then $p^*$ has a linearization
$\begin{pmatrix} 
0& v^*\\ u^*& q^*
\end{pmatrix}$.
If $p=p^*$ we want to take a linearization of $p=\frac 12
(p+p^*)$. The construction in (2) however does not give a selfadjoint $\hat p$. 
Instead we take
$$\frac 12\begin{pmatrix}
0& u& v^*\\
u^*& 0& q^*\\
v&q&0
\end{pmatrix}.$$
\end{enumerate}
\end{proof}

\section{Linearization and Cauchy transforms}

\begin{remark}
The linearization and the calculations for expressing the Cauchy transform of
$P$ in terms of the Cauchy transform of $\hat P$ are independent of the concrete nature of our random variables, neither freeness nor being semicircular is important. Thus we have the following.
\end{remark}

\begin{theorem}\label{thm:8.8}
Let $(\cA,\ff)$ be a $C^*$-probability space and consider selfadjoint $X_1$,\dots ,$X_d$ $\in\cA$. For a selfadjoint polynomial $p\in\CC\lb x_1,\dots,x_d\rb$ 
let 
$$\hat p=b_0\otimes 1+b_1\otimes x_1+\dots+ b_d\otimes x_d,\qquad
\text{with $b_0,b_1,\dots,b_d\in M_n(\CC)$,}$$ 
be a selfadjoint linearization of $p$. Put
$P:=p(X_1,\dots,X_d)\in\cA$ and
$$\hat P=\hat p(X_1,\dots,X_d)=b_0\otimes 1+ b_1\otimes X_1+\dots b_d\otimes X_d\in M_n(\cA).$$ Then we have for $z\in H^+(\CC)$
$$G_P(z)=[G_{\hat P}(\Lambda(z))]_{1,1}=\lim_{\ee\searrow 0}[G_{\hat P}(\Lambda_\ee))]_{1,1}$$
with
$$\Lambda(z)=
\begin{pmatrix}
z&0&\dots&0\\0&0&\dots&0\\
\vdots&\vdots&\ddots&\vdots\\
0&0&\dots&0
\end{pmatrix}\qquad\text{and}\qquad
\begin{pmatrix}
z&0&\dots&0\\0&i\ee&\dots&0\\
\vdots&\vdots&\ddots&\vdots\\
0&0&\dots&i\ee
\end{pmatrix}\in H^+(M_n(\CC)).
$$
\end{theorem}

\begin{remark}
Thus, in order to deal with polynomials of variables, we need to understand
linear matrices in the variables. If the variables are free semicirculars we understand linear matrices in them, as they are operator-valued semicirculars. But how about general free variables: If $X_1,\dots,X_d$ are free, what can we say about $X=b_0\otimes 1+b_1\otimes X_1+\cdots +b_d\otimes X_d$. 
Note
\begin{itemize}
\item[(i)]
$b_0\otimes 1$ is just a shift of the argument in $G_X$, thus easy to deal with;
\item[(ii)]
the operator-valued Cauchy transform of $b_i\otimes X_i$ is determined (theoretically and numerically) in terms of the Cauchy transform of $X_i$;
\item[(iii)]
if $X_1,\dots,X_d$ are free in $(\cA,\ff)$, then $b_1\otimes X_1,\dots,
b_d\otimes X_d$ are, by  Proposition \ref{prop:5.7}, free in $(M_n(\cA),M_n(\CC),\id\otimes \ff)$; hence we have to understand how to deal with sums of free variables on an operator-valued level -- i.e., we need to have a closer look on how to describe operator-valued free convolution.

\end{itemize}

\end{remark}

\chapter{Combinatorial and Analytic Description of Operator-Valued Freeness: Free Cumulants and $R$-Transforms}

Up to now we looked on moments and Cauchy transforms. As in the scalar-valued case it is advantegeous to go over to cumulants and $R$-transforms. Much of the theory is the same, modulo ``respecting the nesting'', as in the scalar-valued case, see Chapters 3 and 4 of \href{https://rolandspeicher.files.wordpress.com/2019/08/free-probability.pdf}{Free Probability Lecture Notes}. We we are not going to give proofs of the operator-valued statements, but we urge the reader (for example, in Exercise \ref{exercise:24}) to check that the scalar-valued arguments are not affected by the requirement that we now have to respect the nesting.

\section{Operator-valued free cumulants}

\begin{definition}
Let $(\cA,\cB,E)$ be an operator-valued probability space. We denote by $E_n$, for $n\in\NN$, the $\cB$-balanced map
$$E_n:\cA^n\to\cB;\qquad (a_1,\dots,a_n)\mapsto E_n(a_1,a_2,\dots,a_n):=
E[a_1a_2\cdots a_n],$$
and by $E_\pi$, for all $n\in\NN$, $\pi\in NC(n)$, the corresponding multiplicative map $E_\pi:\cA^n\to \cB$, for $\pi\in NC(n)$; see Definition
\ref{def:5.10}. Then we define the corresponding \emph{(operator-valued) free cumulants} $\kk_n:\cA^n\to \cB$ by 
\begin{equation}\label{eq:def-ocumulants}
\kk_n(a_1,\dots,a_n):=
\sum_{\pi\in NC(n)} \mu(\pi,1) E_\pi(a_1,\dots,a_n),
\end{equation}
where $\mu$ is the M\"obius function of $NC(n)$.
\end{definition}

\begin{remark}
The $\kk_n$ are also $\cB$-balanced and with their multiplicative extension the Equation \eqref{eq:def-ocumulants} is equivalent to
$$E[a_1\cdots a_n]=E_n(a_1,\dots,a_n)=\sum_{\pi\in NC(n)} \kk_\pi(a_1,\dots,a_n).$$
\end{remark}

\begin{example}
\begin{enumerate}
\item
For $n=1$ we have
$$E[a_1]=\kk_{\,
\begin{tikzpicture}[baseline=0pt]
    \begin{scope}[yshift=0.2em,xshift=1.4em, scale=0.8]
 \draw (0em,-0em) --  (0em,-.5em) ;
    \end{scope}
  \end{tikzpicture} 
}(a_1)=\kk_1(a_1).$$
\item
For $n=2$ we have
\begin{align*}
E[a_1a_2]&=\kk_{
\begin{tikzpicture}[baseline=0pt]
    \begin{scope}[yshift=0.2em,xshift=0.2em, scale=0.8]
 \draw (0em,-0em) --  (0em,-.5em) -- (.5em,-.5em) -- (.5em,-0em);
    \end{scope}
  \end{tikzpicture} 
}(a_1,a_2)+\kk_{
\begin{tikzpicture}[baseline=0pt]
    \begin{scope}[yshift=0.2em,xshift=0.2em, scale=0.8]
 \draw (0em,-0em) --  (0em,-.5em) ;
\draw  (.5em,-.5em) -- (.5em,-0em);
    \end{scope}
  \end{tikzpicture} 
}(a_1,a_2)\\
&=\kk_2(a_1,a_2)+\kk_1(a_1)\kk_1(a_2)\\
&=\kk_2(a_1,a_2)+E[a_1]E[a_2],
\end{align*}
and thus
$$\kk_2(a_1,a_2)=\underbrace{E[a_1a_2]}_{E_{
\begin{tikzpicture}[baseline=0pt]
    \begin{scope}[yshift=0.2em,xshift=0.2em, scale=0.8]
 \draw (0em,-0em) --  (0em,-.5em) -- (.5em,-.5em) -- (.5em,-0em);
    \end{scope}
  \end{tikzpicture} 
}(a_1,a_2)}-\underbrace{E[a_1]E[a_2]}_{E_{
\begin{tikzpicture}[baseline=0pt]
    \begin{scope}[yshift=0.2em,xshift=0.2em, scale=0.8]
 \draw (0em,-0em) --  (0em,-.5em) ;
\draw  (.5em,-.5em) -- (.5em,-0em);
    \end{scope}
  \end{tikzpicture} 
}(a_1,a_2)}.$$
\item
For $n=3$ we have
\begin{multline*}
E[a_1a_2a_3]=\kk_{\,
\begin{tikzpicture}[baseline=0pt]
    \begin{scope}[yshift=0.2em,xshift=0.2em, scale=0.8]
 \draw (0em,-0em) --  (0em,-.5em) -- (.5em,-.5em) -- (.5em,-0em);
\draw (.5em,-.5em) -- (1em,.-.5em) -- (1em,0em);
    \end{scope}
  \end{tikzpicture} 
}(a_1,a_2,a_3) +
\kk_{\,
\begin{tikzpicture}[baseline=0pt]
    \begin{scope}[yshift=0.2em,xshift=0.2em, scale=0.8]
 \draw (0em,-0em) --  (0em,-.5em) -- (.5em,-.5em) -- (.5em,-0em);
\draw (1em,.-.5em) -- (1em,0em);
    \end{scope}
  \end{tikzpicture} 
}(a_1,a_2,a_3)\\
+\kk_{\,
\begin{tikzpicture}[baseline=0pt]
    \begin{scope}[yshift=0.2em,xshift=0.2em, scale=0.8]
 \draw (0em,-0em) --  (0em,-.5em);
\draw (.5em,-.5em) -- (.5em,-0em);
\draw (.5em,-.5em) -- (1em,.-.5em) -- (1em,0em);
    \end{scope}
  \end{tikzpicture} 
}(a_1,a_2,a_3)+\kk_{\,
\begin{tikzpicture}[baseline=0pt]
    \begin{scope}[yshift=0.2em,xshift=0.2em, scale=0.8]
 \draw (0em,-0em) --  (0em,-.5em) -- (.5em,-.5em);
\draw (.5em,-.3em) -- (.5em,-0em);
\draw (.5em,-.5em) -- (1em,.-.5em) -- (1em,0em);
    \end{scope}
  \end{tikzpicture} 
}(a_1,a_2,a_3)
+\kk_{\,
\begin{tikzpicture}[baseline=0pt]
    \begin{scope}[yshift=0.2em,xshift=0.2em, scale=0.8]
 \draw (0em,-0em) --  (0em,-.5em);
\draw (.5em,-.5em) -- (.5em,-0em);
\draw (1em,.-.5em) -- (1em,0em);
    \end{scope}
  \end{tikzpicture} 
}(a_1,a_2,a_3).
\end{multline*}
The interesting term is here
\begin{align*}
\kk_{\,
\begin{tikzpicture}[baseline=0pt]
    \begin{scope}[yshift=0.2em,xshift=0.2em, scale=0.8]
 \draw (0em,-0em) --  (0em,-.5em) -- (.5em,-.5em);
\draw (.5em,-.3em) -- (.5em,-0em);
\draw (.5em,-.5em) -- (1em,.-.5em) -- (1em,0em);
    \end{scope}
  \end{tikzpicture} 
}(a_1,a_2,a_3)&=\kk_2(a_1\kk_1(a_2),a_3)\\[0.5em]
&=E\bigl[a_1 E[a_2] a_3\bigr]-E\bigl[a_1 E[a_2]\bigr]\cdot E[a_3]\\[0.5em]
&=E\bigl[a_1 E[a_2] a_3\bigr]-E[a_1]\cdot E[a_2]\cdot E[a_3]
.\end{align*}
\end{enumerate}
This leads in the end to
\begin{multline*}
\kk_3(a_1,a_2,a_3)=E[a_1a_2a_3]-E[a_1]\cdot E[a_2a_3]
-E[a_1a_2]\cdot E[a_3]\\[0.5em]
-E\bigl[a_1 E[a_2] a_3\bigr]+2E[a_1]\cdot E[a_2]\cdot E[a_3].
\end{multline*}
\end{example}

As in the scalar-valued case (compare 3.23 and 3.24 of \href{https://rolandspeicher.files.wordpress.com/2019/08/free-probability.pdf}{Free Probability Lecture Notes}) one proves the
following characterization of freeness.

\begin{theorem}[freeness $\hat =$ vanishing of mixed cumulants]
Let $(\cA,\cB,E)$ be a $\cB$-valued probability space and $(\kk_n)_{n\in\NN}$ the corresponding free cumulants.
\begin{enumerate}
\item
Consider subalgebras $\cB\subset\cA_i\subset \cA$  for $i\in I$. Then the following are equivalent.
\begin{enumerate}
\item
The subalgebras $\cA_i$, $i\in I$, are free with respect to $E$.
\item
Mixed cumulants in the subalgebras vanish, i.e., 
$\kk_n(a_1,\dots,a_n)=0$ whenever:
$n\geq 2$; $a_j\in\cA_{i_j}$ for $j=1,\dots,n$; and there exist $1\leq k,l\leq n$ such that $i_k\not=i_l$.
\end{enumerate}
\item
Consider random variables $X_i\in\cA$  for $i\in I$. Then the following are
equivalent.
\begin{enumerate}
\item
The random variables $X_i$, $i\in I$, are free with respect to $E$.
\item
Mixed cumulants in the random variables vanish, i.e.,
$$\kk_n(X_{i_1}b_1,X_{i_2}b_2,\dots,X_{i_{n-1}}b_{n-1},X_{i_n})=0$$ 
whenever:
$n\geq 2$;
$i_1,\dots,i_n\in I$;
there exist $1\leq k,l\leq n$ such that $i_k\not= i_l$;
and $b_1,\dots,b_{n-1}\in\cB$.
\end{enumerate}
\end{enumerate}
\end{theorem}

\begin{example}\label{ex:9.6}
This yields then formulas for the calculation of mixed moments; those formulas show that mixed moments of free variables are exponentially bounded, if this is true for of of the variables -- thus providing the missing argument for our proof of Theorem \ref{thm:5.14}. 

As a concrete calculation, consider for $X$ and $Y$ free the following mixed moment:
$$
E[XYXY]=\sum_{\pi\in NC(4)} \kk_\pi(X,Y,X,Y).
$$
Because of the vanishing of mixed cumulants in $X$ and $Y$ only non-crossing $\pi$ with $\pi\leq 
\begin{tikzpicture}[baseline=0pt]
    \begin{scope}[yshift=.5em,xshift=0.2em, scale=1]
 \draw (0em,-0em) --  (0em,-.5em)--(1em,.-.5em) -- (1em,0em);
\draw (.5em,-0em) -- (.5em,-0.3em)-- (1.5em,-.3em)--(1.5em,0);
    \end{scope}
  \end{tikzpicture} 
$
will make a contribution; so we can  continue with
\begin{align*}
&E[XYXY]=\kk_{\,
\begin{tikzpicture}[baseline=0pt]
    \begin{scope}[yshift=.2em,xshift=0.2em, scale=.8]
 \draw (0em,-0em) --  (0em,-.5em) (1em,.-.3em) -- (1em,0em);
\draw (.5em,-0em) -- (.5em,-0.5em)-- (1.5em,-.5em)--(1.5em,0);
    \end{scope}
  \end{tikzpicture} 
}(X,Y,X,Y)
+\kk_{\,
\begin{tikzpicture}[baseline=0pt]
    \begin{scope}[yshift=.2em,xshift=0.2em, scale=.8]
 \draw (0em,-0em) --  (0em,-.5em)--(1em,.-.5em) -- (1em,0em);
\draw (.5em,-0em) -- (.5em,-0.3em)  (1.5em,-.5em)--(1.5em,0);
    \end{scope}
  \end{tikzpicture} 
}(X,Y,X,Y)
+\kk_{\,
\begin{tikzpicture}[baseline=0pt]
    \begin{scope}[yshift=.2em,xshift=0.2em, scale=.8]
 \draw (0em,-0em) --  (0em,-.5em)  (1em,.-.5em) -- (1em,0em);
\draw (.5em,-0em) -- (.5em,-0.5em) (1.5em,-.5em)--(1.5em,0);
    \end{scope}
  \end{tikzpicture} 
}(X,Y,X,Y)\\[0.5em]
&=\kk_1(X)\cdot \kk_2(Y \kk_1(X),Y)+
\kk_2(X\kk_1(Y),X)\cdot\kk_1(Y)+
\kk_1(X)\cdot \kk_1(Y)\cdot \kk_1(X)\cdot \kk_1(Y)\\[0.5em]
&=E[X]\cdot \Bigr(E\bigl[YE[X]Y\bigr]-E\bigl[YE[X]\bigr]\cdot E[Y]\Bigr)\\
&\qquad +
 \Bigr(E\bigl[XE[Y]X\bigr]-E\bigl[XE[Y]\bigr]\cdot E[X]\Bigr)\cdot E[Y]
+E[X]\cdot E[Y]\cdot E[X]\cdot E[Y]\\[0.5em]
&=
 E[X]\cdot E\bigl[YE[X]Y\bigr]+E\bigl[XE[Y]X\bigr]\cdot E[X]
-E[X]\cdot E[Y]\cdot E[X]\cdot E[Y].
\end{align*}
This recovers the formula from Example \ref{ex:5.3} (3).
\end{example}

\begin{prop}\label{prop:9.7}
Let $(\cA,\cB,E)$ be a $\cB$-valued probability space with corresponding cumulants $(\kk_n)_{n\in\NN}$. Consider, for $n\in\NN$, random variables $X_1,\dots,X_n\in \cA$ and $b_1,\dots,b_{n-1}\in\cB$. Then we have
\begin{multline*}
E[X_1b_1X_2b_2\cdots X_{n-1}b_{n-1}X_n]\\[0.5em]
=\sum_{s=1}^n
\sum_{1=j_1<j_2<\dots < j_s\leq n}\kk_s\bigl(X_1 E[b_1X_2\cdots X_{j_2-1} b_{j_2-1} ],X_{j_2} E[b_{j_2}X_{j_2+1}\cdots X_{j_3-1}b_{j_3-1}],\dots,X_{j_s}\bigr)\\ \times E[b_{j_s}\cdots b_{n-1}X_{n}].
\end{multline*}
\end{prop}

\section{Operator-valued $R$-transform}

\begin{theorem}
Let $(\cA,\cB,E)$ be a $\cB$-valued $C^*$-probability space and $X=X^*\in\cA$. Consider the following fully matricial functions on a uniform neighborhood of 0, given via the coefficients in the power series expansion about 0:
\begin{itemize}
\item
the Cauchy transform $G_X$ given via $H_X(z)=G_X(z^{-1})$ by
$$\partial^{n+1} H_X(0,\dots,0)\sharp (b_0,\dots,b_n)=E[b_0Xb_1\cdots b_{n-1} Xb_n]$$
\item
the $R$-transform $R_X$ given by
$$\partial^{n} R_X(0,\dots,0)\sharp (b_1,\dots,b_n)= \kk_{n+1}(Xb_1,Xb_2,\dots,Xb_n,X)
$$
\end{itemize}
Then we have that on suitable domains
\begin{equation}\label{eq:G-R}
zG(z)=1+R[G(z)]\cdot G(z),
\end{equation}
and $G$ and $R$ determine each other via \eqref{eq:G-R}.
\end{theorem}

\begin{remark}
\begin{enumerate}
\setcounter{enumi}{-1}
\item
Note that $R$ has a constant term, whereas $H$ starts with the linear term; on the base level we have the power series expansions, for $z=b\in\cB$
$$H_X(b)= b+bE[X]b+bE[XbX]b+bE[XbXbX]b+\cdots$$
and
$$R_X(b)=\kk_1(X)+\kk_2(Xb,X)+\kk_3(Xb,Xb,X)+\cdots$$
\item
Note that with $G=(G^{(n)})_{n\in\NN}$ and $R=(R^{(n)})_{n\in\NN}$,
\eqref{eq:G-R} means that there exists $R>0$ such that for each $n\in\NN$ we have
$$zG^{(n)}(z)=1+R^{(n)}[G^{(n)}(z)]\cdot G^{(n)}(z)\qquad
\text{for $z\in M_n(\cB)$ with $\Vert z\Vert>R$.}$$
\item
For our applications to polynomials in Theorem \ref{thm:8.8},
$$G_P(z)=\lim_{\ee\searrow 0}[G_{\hat P} (\Lambda_\ee(z))]_{1,1},$$
we actually only need the base level $n=1$ of $G_{\hat P}$.
\item
Since mixed cumulants in free variables vanish we have for free $X_1,X_2$ that
\begin{multline*}
\kk_{n+1}\bigl((X_1+X_2)b_1,(X_1+X_2)b_2,\dots,(X_1+X_2)b_n,(X_1+X_2)\bigr)\\[0.5em]
=
\kk_{n+1}(X_1b_1,X_1b_2,\dots,X_1b_n,X_1)+
\kk_{n+1}(X_2b_1,X_2b_2,\dots,X_2b_n,X_2),
\end{multline*}
and thus also 
$$R_{X_1+X_2}(z)=R_{X_1}(z)+R_{X_2}(z) \qquad\text{for $\Vert z\Vert$ sufficiently small}.$$
This allows in principle to express $G_{X_1+X_2}$ in terms of $G_{X_1}$ and $G_{X_2}$: for $i=1,2$ we calculate from $G_{X_i}$ its $R$-transform $R_{X_i}$ via \eqref{eq:G-R}, then we get easily the $R$-transform of the sum, $R_{X_1+X_2}=R_{X_1}+R_{X_2}$, and use again \eqref{eq:G-R} (now in the other direction) to get from this $G_{X_1+X_2}$. There is, however, a problem with this, namely \eqref{eq:G-R} can usually not be solved explicitly and there is also no good numerical algorithm for dealing with \eqref{eq:G-R}. Hence, as in the scalar-valued
case, we will rewrite the $R$-transform approach into the ``subordination'' language.

\end{enumerate}
\end{remark}

\chapter{Operator-Valued Free Convolution via Subordination Function and the Distribution of Polynomials in Free Variables}

The subordination description of operator-valued free convolution yields as in the scalar-valued case algorithms which can be analytically controlled.
Combining this with the linearization idea solves then the problem of calculating the distribution of polynomials in free variables, which in turn can be used to calculate the asymptotic eigenvalue distribution of polynomials in random matrices. We follow here the presentation from \cite{MSp}, by refering the proof of the main statement to the original paper \cite{BMS}. 

\section{Subordination for operator-valued free convolution}

\begin{remark}
\begin{enumerate}
\item
We want to describe $X_1+X_2$, for $X_1$ and $X_2$ free, in a subordinated form via
$$G_{X_1+X_2}(z)=G_{X_1}(\omega_1(z)),\qquad\text{and}\qquad
G_{X_1+X_2}(z)=G_{X_2}(\omega_2(z))$$
for some subordination functions $\omega_1,\omega_2$.
Let us check, on a formal level, the properties of those (compare also
5.1 of \href{https://rolandspeicher.files.wordpress.com/2019/08/free-probability.pdf}{Free Probability Lecture Notes}):
$$\omega_1(z)=G_{X_1}^{<-1>}(G_{X_1+X_2}(z)).$$
Note that $zG(z)=1+R[G(z)]\cdot G(z)$ means that (for $z=G^{<-1>}(b)$):
$$G^{<-1>}(b)\cdot b=1+R(b)\cdot b,\qquad\text{i.e.,}\qquad 
G^{<-1>}(b)=b^{-1}+R(b).$$
Put now $G_1=G_{X_1}$, $G_2=G_{X_2}$, $G=G_{X_1+X_2}$, and the same for $R$. Then we have
$$\omega_1(z)=G_1^{<-1>}(G(z))=G(z)^{-1}+R_1(G(z))\quad\text{and}\quad
\omega_2(z)=G(z)^{-1}+R_2(G(z))$$
and thus
\begin{align*}
\omega_1(z)+\omega_2(z)&=2 G(z)^{-1}+\underbrace{R_1(G(z))+R_2(G(z))}_{=R[G(z)]=z-G(z)^{-1}}\\
&=z+G(z)^{-1}\\
&=z+G_1(\omega_1(z))^{-1}\\
&=z+F_1(\omega_1(z)),
\end{align*}
and thus
$$\omega_2(z)=z+\underbrace{F_1(\omega_1(z))-\omega_1(z)}_{h_1(\omega_1(z))},$$
where we put
$$F_1(z):=G_1(z)^{-1},\qquad\text{and}\qquad
h_1(z):=F_1(z)-z=G_1(z)^{-1}-z.$$
So we have
$$\omega_2(z)=z+h_1(\omega_1(z))$$
and, by symmetry,
$$\omega_1(z)=z+h_2(\omega_2(z)).$$
Inserting the first equation into the second gives finally
$$\omega_1(z)=z+h_2(z+h_1(\omega_1(z))).$$
This is a fixed point equation for $\omega_1(z)$, which can be used for calculating $\omega_1(z)$ via iterations.
\item
The crucial point is that the fixed point equation can be used to define $\omega_1(z)$ (and, in the same way, $\omega_2(z)$) not just on some
suitably chosen domain, but always on all of $H^+(\cB)$. To show the convergence of the iterates on all of $H^+(\cB)$ one uses again the Earle--Hamilton Theorem.
\item
To make the formal calculations above rigorous is much harder than in the scalar-valued case (in particular, as the involved domains are harder to control), but it can be done. We only give the final result from \cite{BMS}. It would be nice to find a simpler, more streamlined proof of this theorem.
\end{enumerate}
\end{remark}

\begin{theorem}[Belinschi, Mai, Speicher 2017]
Let $(\cA,\cB,E)$ be an operator-valued $C^*$-probability space and consider
selfadjoint $X_1,X_2\in\cA$ which are free with respect to $E$. Then there exists a unique pair of Fr\'echet analytic maps
$$\omega_1,\omega_2: H^+(\cB)\to H^+(\cB)$$
such that
\begin{itemize}
\item[(i)] $\Im \omega_j(z)\geq \Im z$ for all $z\in H^+(\cB)$ and $j=1,2$;
\item[(ii)]
for all $z\in H^+(\cB)$
$$F_1(\omega_1(z))+z=F_2(\omega_2(z))+z=\omega_1(z)+\omega_2(z);$$
\item[(iii)]
for all $z\in H^+(\cB)$
$$G_1(\omega_1(z))=G_2(\omega_2(z))=G(z).$$
\end{itemize}
Moreover, if $z\in H^+(\cB)$, then $\omega_1(z)$ is the unique fixed point of the map $f_z:H^+(\cB)\to H^+(\cB)$, given by
$$f_z(w):=h_2(h_1(w)+z)+z;$$
and $\omega_1(z)=\lim_{n\to\infty} f_z^n(w)$ for any $w\in H^+(\cB)$.
The same statements hold for $\omega_2$, where $f_z$ is replaced by
$$w\mapsto h_1(h_2(w)+z)+z.$$
\end{theorem}

\section{Distribution of polynomials in free variables}

\begin{remark}
This can then be used, together with the linearization idea, to compute numerically distributions of polynomials in free variables. This has relevance for the 
asymptotic eigenvalue distribution of random matrices. Assume that
$X_1^{(N)},\dots,X_d^{(N)}$ are $N\times N$ random matrices which are
asymptotically free, i.e.,
$$(X_1^{(N)},\dots,X_d^{(N)}) \overset{N\to\infty}
\longrightarrow (X_1,\dots, X_d),$$
where $X_1,\dots,X_d$ are free. Then, for any polynomial $p\in\CC\lb x_1,\dots,x_d\rb$, we also have
$$p(X_1^{(N)},\dots,X_d^{(N)}) \overset{N\to\infty}
\longrightarrow p(X_1,\dots, X_d);$$
and the distribution of the limit can be calculated via linearization and operator-valued free convolution.

Note the following typical situations for asymptotically free random matrices:
\begin{itemize}
\item[(i)]
independent \GUE\ are asymptotically free;
\item[(ii)] 
 $\GUE$ are asymptotically free from deterministic (e.g., diagonal) matrices;
\item[(iii)]
``randomly rotated'' matrices are asymptotically free: for $D_1^{(N)},
D_2^{(N)}$ deterministic (e.g., diagonal) matrices and $U_N$ Haar unitary
$N\times N$ random matrices, we have that
$D_1^{(N)}$ and $U_N D_2^{(N)} U_N^*$ are asymptotically free; so, in particular, asymptotically the eigenvalue distribution of
$p(D_1^{(N)},U_N D_2^{(N)} U_N^*)$ is given by the distribution of $p(X_1,X_2)$ where $X_1$ and $X_2$ are free and 
$\mu_{D_1^{(N)}}\to \mu_{X_1}$ and 
$\mu_{D_2^{(N)}}\to \mu_{X_2}$.
\end{itemize}
\end{remark}

\begin{example}
Let us compare, for the polynomial $p(x,y)=xy+yx+x^2$, the distribution of asymptotically free random matrices with the limit distribution, which we calculate by our linearization and operator-valued convolution machinery. 

\begin{enumerate}
\item
Consider first, for $N=4000$, a \GUE(N) matrix
$A_N$ and
 a deterministic diagonal matrix $X_N$ with 2000 eigenvalues -2, 1000 eigenvalues -1 and 1000 eigenvalues 1.
We compare the histogram of the $N$ eigenvalues of
$p(X_N,A_N)$
with the distribtion (red curve) of $p(X,S)$, where $S$ and $X$ are free, $S$ is a semicircular element and $X$ has distribution
$\mu_X=\frac 14(2\delta_{-2}+\delta_{-1}+\delta_{+1})$.
$$\includegraphics[width=11cm]{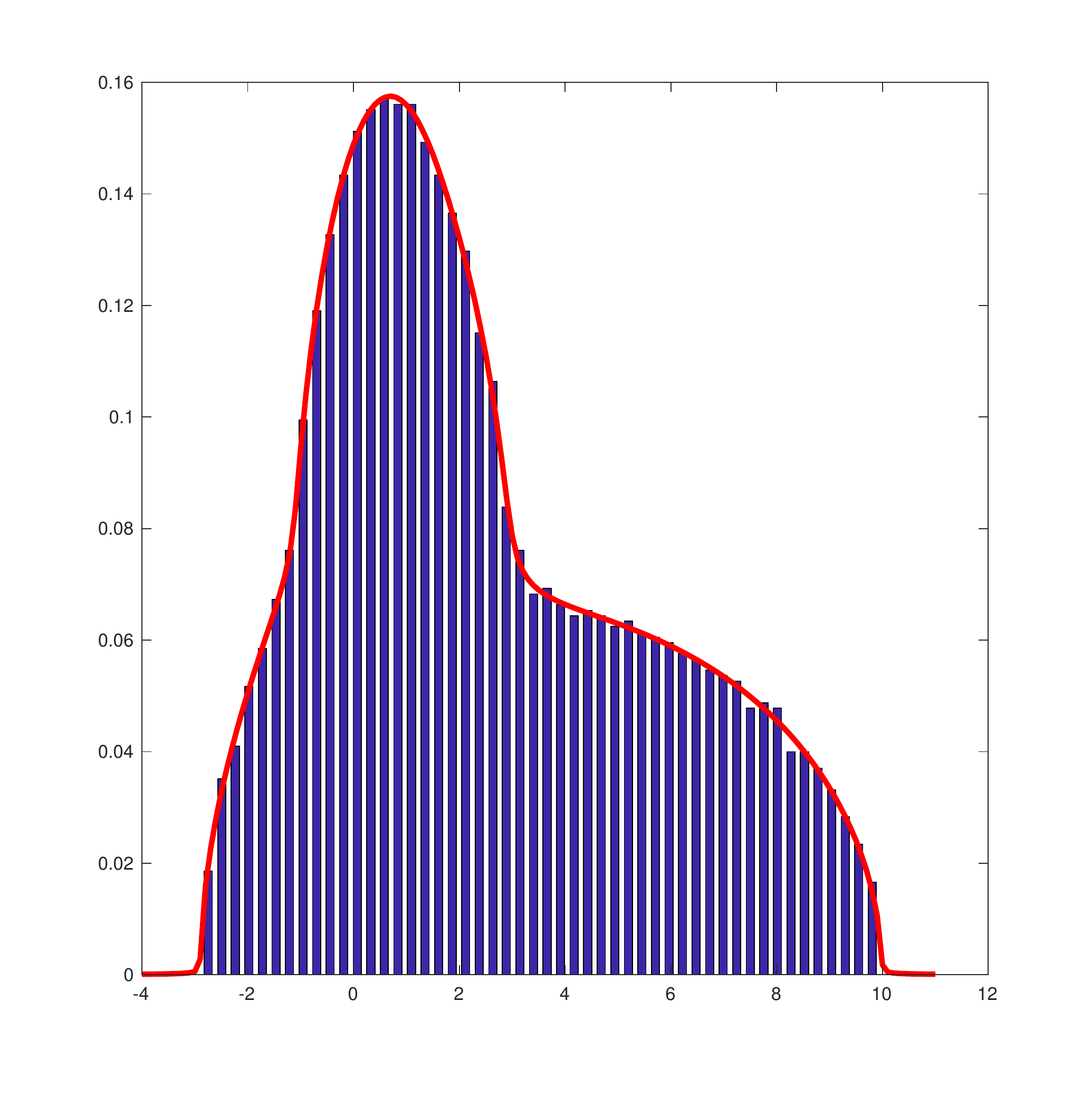}$$

\item
Consider now, again for $N=4000$, two deterministic diagonal matrices $X_N$ and
$Y_N$;  $Y_N$ has 2000 eigenvalues 1 and 2000 eigenvalues 3; and $X_N$ is the same as before, i.e., a diagonal matrix with 2000 eigenvalues -2, 1000 eigenvalues -1 and 1000 eigenvalues 1. In addition we take now a Haar unitary random matrix $U_N$ and compare the histogram of the $N$ eigenvalues of
$p(X_N,U_NY_N U_N^*)$ 
with the distribution (red curve) of $p(X,Y)$, where $X$ and $Y$ are free, with distribution
$\mu_X=\frac 14(2\delta_{-2}+\delta_{-1}+\delta_{+1})$ and
$\mu_Y=\frac 12(\delta_{1}+\delta_{3})$.

\hskip-1cm\includegraphics[width=16cm]{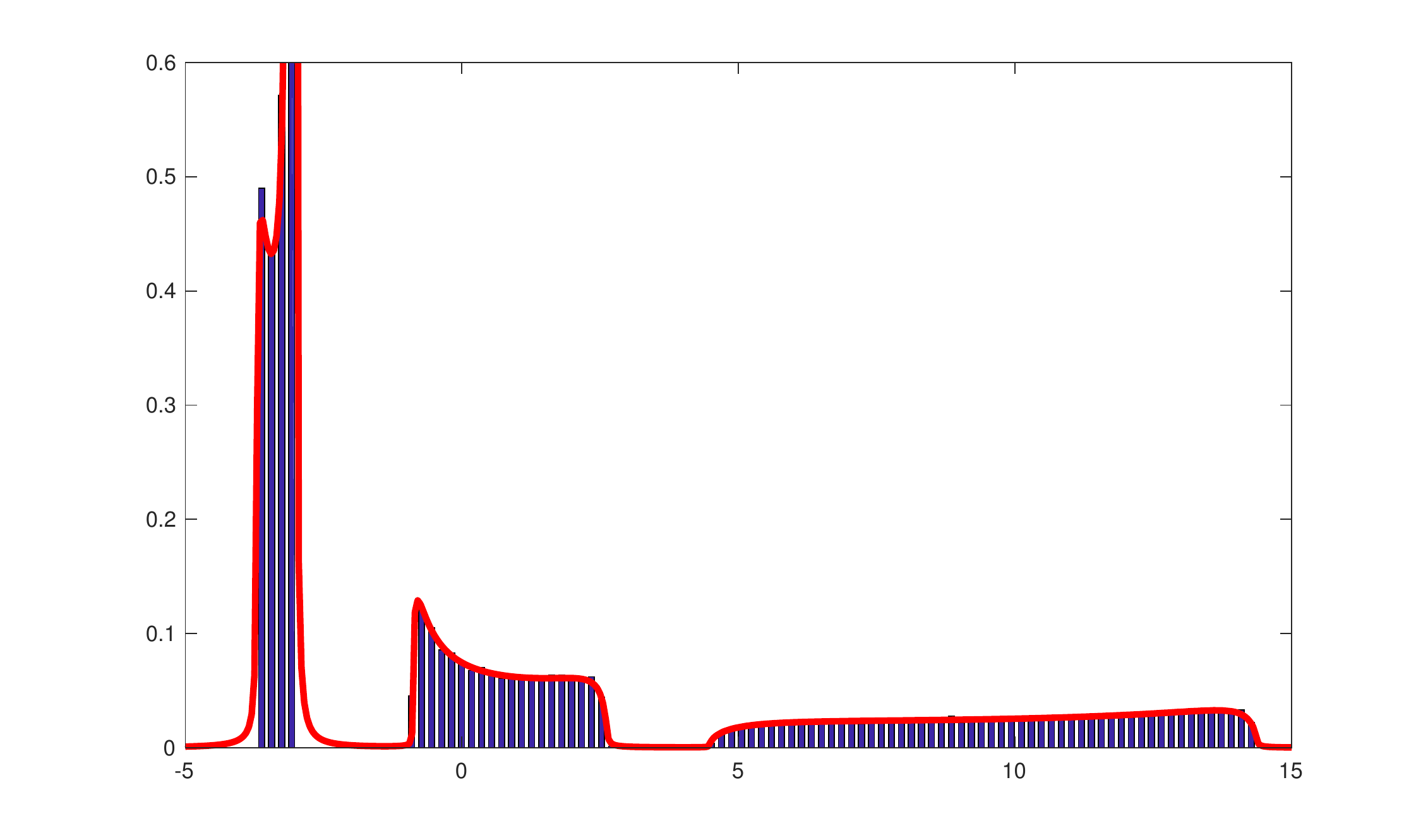}
\end{enumerate}

\end{example}

\chapter{Distribution of Rational Expressions in Free Random Variables}

The linearization idea is usually (i.e., in other contextes than free probability) used for dealing with rational functions, not just polynomials. Thus it looks feasible to try to extend our results to rational functions. We will follow quite closely \cite{HMS}, where one can also find more information on the history of the linearization idea and more details on non-commutative rational functions.

\section{Linearization for non-commutative rational functions}

\begin{remark}
Recall the idea of the linearization of a polynomial. For a polynomial
$P=p(X_1,\dots,X_d)\in\cA$ we need to find $U,Q,V$ with
\begin{itemize}
\item[$\circ$]
$U,Q,V$ are affine matrices in $X_1,\dots,X_d$;
\item[$\circ$]
$Q$ is invertible;
\item[$\circ$]
$P=-UQ^{-1}V$.
\end{itemize}
Then the linearization
$\hat P
=\begin{pmatrix}
0&U\\ V& Q\end{pmatrix}$
knows a lot about $P$, namely
\begin{equation}\label{eq:11.1}
G_P(z)=[G_{\hat P}(\Lambda(z))]_{1,1}.
\end{equation}
\emph{Question}: Can we linearize more general ``functions''?
\end{remark}

\begin{example}
Note that in $P=-UQ^{-1}V$ the inverse shows up, which suggests that we might also linearize inverses. Try the simplest case, $P=X^{-1}$. We can write this as
$$P=- (1)\cdot (-X)^{-1}\cdot (1),$$
i.e., $U=1$, $Q=-X$, $V=1$, all $1\times 1$ matrices, and thus
$$\hat P=\begin{pmatrix}
0&1\\1&-X \end{pmatrix}
\in M_2(\cA).$$
(Note that $\hat P$ is here also selfadjoint!) If we assume that $Q=-X$ is invertible, then this satisfies all properties of our linearization and \eqref{eq:11.1} allows to calculate the distribution of $X^{-1}$ via the linear matrix $\hat P$. (Of course, in this case of one variable we would calculate the distribution of $X^{-1}$ form the distribution of $X$ just via ordinary function calculus.) 

Note that in this case invertibility of $Q$ is not just an algebraic issue, which is true for all $p(X)$, but depends on the existence of $p(X)=X^{-1}$ for the concretely considered $X$. We have to be careful that the existence of $P$ implies the existence of all inverses which show up in our calculations. The basic ingredient for all this is the following well-known formula.
\end{example}

\begin{theorem}[Schur complement formula]\label{thm:Schur}
Let $\cA$ be a complex unital algebra. Let matrices
$$a\in M_k(\cA),\quad b\in M_{k,l}(\cA),\quad c\in M_{l,k}(\cA),\quad
d\in M_l(\cA)$$
be given and assume that $d$ is invertible in $M_l(\cA)$. Then the following are equivalent.
\begin{itemize}
\item[(i)]
$\begin{pmatrix} a&b\\c&d\end{pmatrix}\in M_{k+l}(\cA)$ is invertible.
\item[(ii)]
The \emph{Schur complement} $a-bd^{-1}c$ is invertible in $M_k(\cA)$.
\end{itemize}
If those are satisfied, then we have
$$\begin{pmatrix}
a&b\\c&d\end{pmatrix}^{-1}=
\begin{pmatrix}
(a-bd^{-1}c)^{-1}&*\\ *&*\end{pmatrix}.$$
\end{theorem}

\begin{proof}
We have
\begin{equation}\label{eq:11.2}
\begin{pmatrix}
a&b\\ c&d\end{pmatrix}=
\begin{pmatrix}
1& bd^{-1}\\ 0&1\end{pmatrix}
\begin{pmatrix}
a-bd^{-1}c& 0\\ 0&d \end{pmatrix}
\begin{pmatrix}
1&0\\ d^{-1} c& 1 \end{pmatrix}.
\end{equation}
Since the first and the third factor are always invertible, the invertibility of the left hand side is equivalent to the invertibility of 
$$\begin{pmatrix}
a-bd^{-1}c& 0\\ 0&d \end{pmatrix},$$
which in turn is equivalent to the invertibility of $a-bd^{-1}c$ (since $d$ is invertible by assumption).

The formula for the inverse follows by taking the inverse of \eqref{eq:11.2}.
\end{proof}

\begin{definition}
Let $r$ be a rational expression in the formal variables $x_1,\dots,x_d$. A
\emph{linear representation} $\rho=(u,q,v)$ of $r$ consists of 
\begin{itemize}
\item[$\circ$]
an affine matrix $q$ in the variables $x_1,\dots,x_d$, of size $n\times n$ for some $n\in\NN$,
\item[$\circ$]
an $1\times n$ matrix $u$ over $\CC$,
\item[$\circ$]
and an $n\times 1$ matrix $v$ over $\CC$
\end{itemize}
such that we have for any unital algebra $\cA$ and any $X_1,\dots,X_d\in\cA$:
whenever $r(X_1,\dots,X_d)$ makes sense in $\cA$ (i.e., all inverses appearing in $r$ must exist in $\cA$), then $q(X_1,\dots,X_d)$ is also invertible in $M_n(\cA)$ and we have then
$$r(X_1,\dots,X_d)=-u q(X_1,\dots,X_d)^{-1} v.$$
\end{definition}

\section{Distribution of non-commutative rational functions}

\begin{theorem}\label{thm:11.5}
Let $r$ be a selfadjoint rational expression and $\rho=(u,q,v)$ a selfadjoint linear representation of $r$ (i.e., $u=v^*$, $q=q^*$). Consider a 
$C^*$-probability space $(\cA,\ff)$ and selfadjoint random variables 
$X_1,\dots,X_d\in\cA$ such that $r(X_1,\dots,X_d)$ is defined in $\cA$ (necessarily as bounded operator). Then, with
$$\hat R:=\begin{pmatrix}
0& u\\ v& q(X_1,\dots,X_d) \end{pmatrix}\in M_{n+1}(\cA),$$
we have for all $z\in H^+(\CC)$
$$G_{r(X_1,\dots,X_d)}(z)=\lim_{\ee\searrow 0}[G_{\hat R}(\Lambda_\ee(z))]_{1,1}.$$
\end{theorem}

\begin{proof}
Compare also Example \ref{ex:8.2}. We have
$$\Lambda(z)-\hat R=\begin{pmatrix}
z& -u\\ -v& -q(X_1,\dots,X_d)
\end{pmatrix};$$
by definition of linear representation, $q(X_1,\dots,X_d)$ is invertible; so, by the Schur complement formula \ref{thm:Schur}, $\Lambda(z)-\hat R$ is invertible
if and only if 
$$z-u(-q(X_1,\dots,X_d))^{-1}v=z-r(X_1,\dots,X_d)$$
is invertible, and then
$$[(\Lambda(z)-\hat R)^{-1}]_{1,1}=(z-r(X_1,\dots,X_d))^{-1}.$$
Applying $\ff$ to this, and taking into account the continuity in $\ee$ as in Example \ref{ex:8.2}, gives the statement on Cauchy transforms. 
\end{proof}

\textbf{11.6. Algorithm for linear representations.}
For every rational expression one can build a linear representation according to the following algorithm.
\begin{enumerate}
\item
Scalars $\lambda\in\CC$ and variables $x_j$ have respective linear representations
$$\left(
\begin{pmatrix}
0& 1
\end{pmatrix},
\begin{pmatrix}
\lambda& -1\\ -1& 0 \end{pmatrix},
\begin{pmatrix}
0\\ 1
\end{pmatrix}
\right)\qquad\text{and}\qquad
\left(
\begin{pmatrix}
0& 1
\end{pmatrix},
\begin{pmatrix}
x_j& -1\\ -1& 0 \end{pmatrix},
\begin{pmatrix}
0\\ 1
\end{pmatrix}
\right).
$$
\item
If $(u_1,q_1,v_1)$ is a representation of $r_1$ and $(u_2,q_2,v_2)$ is
a representation of $r_2$, then representations for $r_1+r_2$ and for $r_1\cdot r_2$ are respectively given by
$$\left( 
\begin{pmatrix}
u_1& u_2 \end{pmatrix},
\begin{pmatrix}
q_1& 0\\ 0& q_2\end{pmatrix},
\begin{pmatrix}
v_1\\ v_2 \end{pmatrix}\right)\qquad\text{and}\qquad
\left(\begin{pmatrix}
0& u_1 \end{pmatrix},
\begin{pmatrix}
v_1u_2& q_1\\ q_2& 0\end{pmatrix},
\begin{pmatrix}
0\\ v_2 \end{pmatrix}\right).$$
\item
If $(u,q,v)$ is a representation of $r\not=0$, then
$$\left(\begin{pmatrix}
1& 0 \end{pmatrix},
\begin{pmatrix}
0& u\\ v& -q\end{pmatrix},
\begin{pmatrix}
1\\ 0 \end{pmatrix}\right)$$
is a representation of $r^{-1}$.
\end{enumerate}
\begin{proof}
Let us just check (3). We have to see: if $r^{-1}(X_1,\dots,X_d)$ makes sense
(i.e., $r(X_1,\dots,X_d)\not=0$ and invertible in $\cA$), then
$$\begin{pmatrix}
0& u\\ v& -q(X_1,\dots,X_d)\end{pmatrix}$$
is invertible. Since $r(X_1,\dots,X_d)$ makes sense, $q(X_1,\dots,X_d)$ is 
invertible (by the definition of a linear representation) and, by the Schur complement formula \ref{thm:Schur}, the matrix above is invertible if and only if
$-uq(X_1,\dots,X_d)v=r(X_1,\dots,X_d)$ is invertible; but this is the case by our assumption; and then we have, still by \ref{thm:Schur},
$$
\begin{pmatrix}
1&0
\end{pmatrix}
\begin{pmatrix}
0& u\\ v& -q(X_1,\dots,X_d)\end{pmatrix}^{-1}
\begin{pmatrix} 1\\0\end{pmatrix}
=\left[\begin{pmatrix}
0& u\\ v& -q(X_1,\dots,X_d)\end{pmatrix}^{-1}\right]_{1,1}=r(X_1,\dots,X_d)^{-1}.$$

\end{proof}
\setcounter{theorem}{6}
\begin{example}
Let us apply the above algorithm to $r(x,y)=[x^{-1}+y^{-1}]^{-1}$. First, for $x^{-1}$ and $y^{-1}$ we have the linearizations
$$\left(
\begin{pmatrix}
1&0&0 \end{pmatrix},
\begin{pmatrix}
0&0&1\\ 0& -x& 1\\
1&1&0
\end{pmatrix},
\begin{pmatrix}
1\\0\\0
\end{pmatrix}
\right)
\qquad\text{and}\qquad
\left(
\begin{pmatrix}
1&0&0 \end{pmatrix},
\begin{pmatrix}
0&0&1\\ 0& -y& 1\\
1&1&0
\end{pmatrix},
\begin{pmatrix}
1\\0\\0
\end{pmatrix}
\right),
$$
which gives for $x^{-1}+y^{-1}$ the linearization
$$\left(
\begin{pmatrix}
1&0&0&1&0&0
\end{pmatrix},
\begin{pmatrix}
0&0&1&0&0&0\\
0&-x&1&0&0&0\\
1&1&0&0&0&0\\
0&0&0&0&0&1\\
0&0&0&0&-y&1\\
0&0&0&1&1&0
\end{pmatrix},
\begin{pmatrix}
1\\0\\0\\1\\0\\0
\end{pmatrix}
\right).$$
Finally, the inverse $[x^{-1}+y^{-1}]^{-1}$ has then the linearization
$$\left(
\begin{pmatrix}
1&0&0&0&0&0&0
\end{pmatrix},
\begin{pmatrix}
0&1&0&0&1&0&0\\
1&0&0&-1&0&0&0\\
0&0&x&-1&0&0&0\\
0&-1&-1&0&0&0&0\\
1&0&0&0&0&0&-1\\
0&0&0&0&0&y&-1\\
0&0&0&0&-1&-1&0
\end{pmatrix},
\begin{pmatrix}
1\\0\\0\\0\\0\\0\\0
\end{pmatrix}
\right).
$$

\end{example}

\chapter{Unbounded Rational Expressions}

Evaluating rational expressions in operators will typically lead to unbounded operators. Here we will see that we can also say quite a bit about such a situation. Actually, understanding what is going on there is crucial for getting a grasp on one of the most basic regularity questions about non-commutative distributions: the absence of atoms in the distribution of polynomials or rational functions of our operators. The material here relies on the original work \cite{MSY}, where one can also find more details about the algebraic description of non-commutative rational functions as the ``free skew field'', and the ``fullness'' of matrices in this context.

\section{Going unbounded}

\begin{remark}
Note that we have to restrict to $X_1,\dots,X_d\in\cA$ for which $r(X_1,\dots,X_d)$ is defined in $\cA$, for a rational expression $r$. Up to now we
considered this in a $C^*$-algebra $\cA$, which means that $r(X_1,\dots,X_d)$ has to exist in $\cA$, i.e., as a bounded operator. 
Can we weaken this?
\end{remark}

\begin{example}
Let $X:\Omega\to\RR$ be a classical real-valued random variable, defined on
a probability space $(\Omega,\frak{A},P)$. When does $Y:=X^{-1}=1/X$ make sense as a random variable. Since our functions are defined only almost
everywhere, we need
\begin{equation}\label{eq:measurezero}
\mu_X(\{0\})=P(X=0)=0.
\end{equation}
If we consider $X$ as multiplication operator on $L^2(\Omega,P)$, then
\eqref{eq:measurezero} says that the kernel
$$\ker(X):=\{f\in L^2(\Omega)\mid Xf=0\}$$
is trivial, i.e., $\ker(X)=\{0\}$. Under this condition, $X^{-1}$ exists, but it might
be an unbounded operator, namely if $0$ is in the spectrum $\sigma(X)$ of $X$.
\begin{enumerate}
\item
For example, let $X=S$ be a semicircular variable with distribution
$$\mu_S=
\begin{tikzpicture}[baseline=0pt]
\begin{scope}[yshift=-0em, scale=.7]
	\tikzset{dot/.style={circle,fill=#1,inner sep=0,minimum size=4pt}}

	\node[dot=black] at (-2, 0)   (1) {};
	\node[dot=black] at (2, 0)   (2) {};
	\node[dot=black] at (0, 2)   (3) {};
		
	\draw (-2,-0.5) -- node {$-2$} (-2,-0.5);
	\draw (2,-0.5) -- node {$2$} (2,-0.5);
	\draw (0.5,2.6) -- node {$\frac{1}{\pi}$} (0.5,2.6);

	\draw [->] (-2.5,0) -- (2.5,0);	
	\draw [->] (0,-0.5) -- (0,2.5);	
	\draw (-2,0) -- (2,0) arc(0:180:2) --cycle;
\end{scope}
	\end{tikzpicture}
$$
We can realize $X$ as multiplication operator on the interval $[-2,2]$; i.e.,
$$(Xf)(t)=tf(t)\qquad\text{for $f\in L^2([-2,2],\mu_S)$.}$$
Then $0\in\sigma(S)$ and $S^{-1}$ does not exist as bounded operator, but
makes sense as unbounded operator:
$(X^{-1}f)(t)=t^{-1}f(t)$ for $f$ such that $t\mapsto t^{-1}f(t)$ is in $ L^2([-2,2],\mu_S)$. Note that we only need injectivity of $X$ -- i.e., $\ker(X)=\{0\}$ -- to ensure ``surjectivity'' --  i.e., that the image of $X$ is dense, so that we can invert it there. This is like for matrices, but of course is not true for general infinite-dimensional operators.
\item
Without injectivity we have no chance of making sense of $X^{-1}$, even as unbounded operator. E.g., if $\mu_X=\frac 12(\delta_0+\delta_1)$, there is
no $X^{-1}$.
\end{enumerate}
\end{example}

\section{Affiliated unbounded operators}

\begin{definition}
Let $M\subset B(\HH)$ be a von Neumann algebra. A densely defined and closed unbounded operator $X$ on $\HH$ is \emph{affiliated} with $M$, if for every unitary $U\in M'$ ($M'$ is the commutant) we have $UX=XU$. [Equivalently, in the polar
decomposition $X=U\vert X\vert$ we have $U\in M$ and $\vert X\vert$ is affiliated with $M$, i.e., all spectral projections of $\vert X\vert$ are in $M$.]
We write $\tilde M$ for the set of operators affiliated to $M$.
\end{definition}

\begin{example}
\begin{enumerate}
\item
If $M=B(\HH)$, then $\tilde M$ consists of all unbounded densely defined and closed operators on $\HH$; for $\dim \HH=\infty$ this is a nasty object without much structure.
\item
If $M=L^\infty(\mu)$, then $\tilde M$ is the $*$-algebra of all $\mu$-measurable functions.
\item
If $M$ is a finite von Neumann algebra (i.e., it has a faithful normal trace $\tau$), then the situation is as nice as in the classical commutative case or in the case of matrices; namely, then $\tilde M$ is a $*$-algebra and for $X\in\tilde M$ the  inverse $X^{-1}\in\tilde M$ exists if and only if $X$ is injective, i.e., $\ker(X)=\{0\}$. [Those are results of Murray and von Neumann.]
\end{enumerate}
\end{example}

\begin{remark}
\begin{enumerate}
\item
Note that the case of a finite von Neumann algebra is relevant for us; our $C^*$-probability spaces $(\cA,\ff)$ are usually $W^*$-probability spaces $(M,\tau)$ where $M$ is a von Neumann algebra and $\tau$ is a trace. In particular, free semicirculars $S_1,\dots,S_d$ are living in a finite von Neumann algebra. More general, limits of random matrices do so, since our $\ff$ as the limit of traces
on matrices is necessarily also a trace.
\item
If we are in a finite von Neumann algebra setting $(M,\tau)$, then we can 
replace $\cA$ in Theorem \ref{thm:11.5} by $\tilde M$ and thus also treat $r(X_1,\dots,X_d)$ which are defined as unbounded affiliated operators. Via our linearization $r$ $\hat=$ $(u,q,v)$ this requirement on the existence of $r(X_1,\dots,X_d)$ as unbounded operator is the same as the existence of the inverse of $q(X_1,\dots,X_d)$ as unbounded operator; and then we still have
$r(X_1,\dots,X_d)=-uq(X_1,\dots,X_d)^{-1} v$. 
\item
This raises the question whether there are operators $X_1,\dots,X_d$ for which
\begin{enumerate}
\item all rational expressions $r(X_1,\dots,X_d)$ are defined as unbounded operators or
\item
all inverses of $q(X_1,\dots,X_d)$ exist as unbounded operators.
\end{enumerate}
Note that we have to specify more precisely which $r$ and $q$ we mean.
\begin{enumerate}
\item
We have to make sure that we never invert 0; thus $0^{-1}$ is not allowed as $r$; but there can also be more subtle versions of this, like
$$(yxx^{-1}-y)^{-1}\qquad\text{or}\qquad
\bigl\{x-[x^{-1}+(y^{-1}-x)^{-1}]^{-1}-xyx\bigr\}^{-1}.$$
\item
The $q$ arising in our linearization algorithm are \emph{full} in the following sense: $q$ has no proper rectangular factorization in matrices over $\CC\lb x_1,\dots,x_d\rb$, i.e. if we can factorize $q\in M_n(\CC\lb x_1,\dots,x_d\rb)$ as
$q=q_1q_2$ with $q_1\in M_{n,r}(\CC\lb x_1,\dots,x_d\rb)$ and $q_2\in M_{r,n}(\CC\lb x_1,\dots,x_d\rb)$, then we necessarily have $r\geq n$.

Note: if $q$ is not full, then
$q(X_1,\dots,X_d)=q_1(X_1,\dots,X_d)\cdot q_2(X_1,\dots,X_d)$; since $q_2(X_1,\dots,X_d)$ as an $r\times n$ matrix with $r<n$ has no dense image, it must also have a kernel, but then $q(X_1,\dots,X_d)$ has also a kernel. So we need clearly fullness as a requirement for the considered $q$.
\end{enumerate}
\end{enumerate}
\end{remark}

\section{Realization of non-commutative rational functions as unbounded operators}

\begin{theorem}[Mai, Speicher, Yin 2019]
Let $(M,\tau)$ be a tracial $W^*$-probability space and consider $X_1,\dots,X_d\in M$. Then the following are equivalent.
\begin{itemize}
\item[(i)]
For all meaningful rational expressions $r\not=0$, the operator $r(X_1,\dots,X_d)$ exists as unbounded operator in $\tilde M$ and is invertible in $\tilde M$.
\item[(ii)]
For all full affine $q\in M_n(\CC\lb x_1,\dots,x_d\rb)$ the operator
$q(X_1,\dots,X_d)\in M_n(M)$ is invertible in $M_n(\tilde M)$.
\item[(iii)]
$\Delta(X_1,\dots,X_d)=d$, which means the following: if we have finite
rank operators $T_1,\dots,T_d$ on $L^2(M,\tau)$ such that
$\sum_{k=1}^d [T_k,X_k]=0$, then necessarily $T_1=\dots =T_d=0$.
\end{itemize}
\end{theorem}

\begin{remark}
\begin{enumerate}
\item
The equivalence between (i) and (ii) is more or less the linearization idea; the relation between (ii) and (iii) relies on the following. Consider linear and selfadjoint
$$\hat R=b^{(0)}\otimes 1+b^{(1)}\otimes X_1+\cdots+ b^{(d)}\otimes X_d$$
with $b^{(0)}, b^{(1)},\dots,b^{(d)}\in M_n(\CC)$ selfadjoint, and assume we have an element $f=(f_1,\dots,f_n)$, with $f_i\in L^2(M)$ for $i=1,\dots ,n$, in the kernel of $\hat R$, i.e., $\hat R f=0$; then put
$$T_k:=\sum_{i,j=1}^n b_{ij}^{(k)} \lb \cdot, f_i\rb f_j \qquad (k=0,1,\dots,d).$$
Those $T_0,T_1,\dots,T_d$ are finite rank operators and $\hat R f=0$ is then
$$T_0+\sum_{k=1}^k X_kT_k=0.$$
Since the $T_i$ are selfadjoint, we get by taking the adjoint
$$T_0+\sum_{k=1}^d T_kX_k=0.$$
By taking the difference between those two equations we have then
$$\sum_{k=1}^d \underbrace{(T_kX_k-X_kT_k)}_{[T_k,X_k]}=0.$$
The theorem holds also for non-selfadjoint $X_i$, the arguments are getting then more involved.
\item
It is not obvious how to check whether $\Delta(X_1,\dots,X_d)=d$ is satisfied or not. However, there are a couple of free probability tools to decide on this, like ``maximality of free entropy dimension'' or ``existence of a dual system''. So we know, for example, that $\Delta(S_1,\dots,S_d)=d$ for free semicirculars $S_1,\dots,S_d$.
\end{enumerate}
\end{remark}

The above gives us directly some strong implications about the absence of atoms in the distribution of polynomials, or even rational functions, of operators
which satisfy $\Delta(X_1,\dots,X_d)=d$. Let us formulate this just for the most prominent case of free semicirculars.

\begin{cor}
Let $(M,\tau)$ be a finite $W^*$-probability space and $S_1,\dots,S_d\in M$ free
semicirculars.
\begin{enumerate}
\item
For any meaningful rational expression $r$ the operator $r(S_1,\dots,S_d)\in
\tilde M$ exists as unbounded operator. If $r=r^*$ and not constant, then
$\mu_{r(S_1,\dots,S_d)}$ has no atoms.
\item
For any full $q\in M_n(\CC\lb x_1,\dots,x_d\rb)$ the operator $q(S_1,\dots,S_d)$ is invertible in $M_n(\tilde M)$. If $q=q^*$, then $\mu_{q(S_1,\dots,S_d)}$, with respect to $(M_n(M),\tr_n\otimes\tau)$, has no atom at 0.
\end{enumerate}
\end{cor}

\chapter{Exercises}

\section{Assignment 1}\label{assignment:1}

In Examples \ref{ex:1.7} and \ref{ex:1.8} we saw two realizations of the most important non-com-mutative distribution, namely $n$ free semicircular elements.
In this assignment you are asked to familiarize yourself with the meaning of this.
For the notion of freeness you might watch Lecture 1 and 2 from the class ``
\href{https://www.math.uni-sb.de/ag/speicher/web_video/freeprobws1819/free_prob_ws1819.html}{Free Probability Theory}'' from last term or read the corresponding Chapter 1 of the \href{https://rolandspeicher.files.wordpress.com/2019/08/free-probability.pdf}{class notes}. For random matrices you might watch Lecture 17 and 18 or read Chapter 6.

\begin{exercise}\label{exercise:1}
Let $S_1,\dots,S_n$ be the operators on the full Fock space from Example \ref{ex:1.7}.
\begin{enumerate}
\item[(i)]
Saying that each $S\in\{S_i \mid 1\leq i\leq n\}$ is a semicircular variable means that its odd moments are zero and the even moments are given by the Catalan numbers, i.e.
$$\ff(S^{2k+1})=0\quad\text{and}\quad\ff(S^{2k})=\frac1{k+1} \binom{2k}k.$$
Check the latter for small $k$, i.e. show that
$$\ff(S^2)=1,\qquad \ff(S^4)=2,\qquad \ff(S^6)=5,\qquad \ff(S^8)=14.$$
\item[(ii)]
Saying that the $S_1,\dots,S_n$ are free means that special mixed moments vanish. Show this for the following special cases.
$$\ff(S_1S_2S_1S_2)=0,\qquad \ff((S_1^4-2)(S_2^6-5)(S_1^2-1))=0.$$
\end{enumerate}
\end{exercise}

\begin{exercise}\label{exercise:2}
Let $X_i^{(N)}$ be the independent Gaussian random matrices from Example \ref{ex:1.8}.
Familiarize yourself with computer programs (e.g., matlab) to produce random matrices and calculate and plot histograms of their eigenvalues.
\begin{enumerate}
\item[(i)]
Saying that, for each $i$, $X_i^{(N)}$ is asymptotically a semicircular variable means that for large $N$ the eigenvalue distribution of the $N$ eigenvalues of such a matrix is close to the semicircle distribution.
Check this by producing a histogram for a $1000\times 1000$ Gaussian random matrix.
\item[(ii)] Saying that $X_1^{(N)},\dots,X_n^{(N)}$ are asymptotically free means that special mixed moments (with respect to the normalized trace $\tr$) are, for large $N$, close to zero. Check this numerically for the following special cases:
$$\tr(ABAB),\qquad \tr((A^4-2)(B^6-5)(A^2-1)),$$
where $A$ and $B$ are two independent $1000\times 1000$ Gaussian random matrices.
\end{enumerate}
\end{exercise}

\section{Assignment 2}

\begin{exercise}
Let $(\cC,\ff)$ be a non-commutative probability space. Put
$$\cA:=M_n(\cC),\qquad \cB:=M_n(\CC),\qquad E:=\id\otimes \ff:\cA\to\cB.$$

\begin{enumerate}

\item[(i)] Show that $(\cA,\cB,E)$ is an operator-valued probability space.

\item[(ii)] Assume that $(\cC,\ff)$ is a $C^*$-probability space. Show that $(\cA,\cB,E)$ is then an operator-valued $C^*$-probability space.

\item[(iii)] Show that in the $C^*$-case we also have: if $\ff$ is faithful, then $E$ is also faithful. [Faithful means: $E(A^*A)=0$ implies that $A=0$.]

\item[(iv)] Assume that $\ff$ is a trace, i.e., $\ff(AB)=\ff(BA)$ for all $A,B\in \cC$. Does then also $E$ have the tracial property? Give a proof or counter example!

\end{enumerate}

\end{exercise}

\begin{exercise}\label{exercise:4}
Let $\cB$ be a unital algebra. Consider a collection of functions $F=(F_m)_{m\in \NN}$
\begin{align*}
F_m:M_m(\cB)\to M_m(\cB),\quad z\mapsto F_m(z).
\end{align*}
\begin{enumerate}

\item[(i)] We say that $F$ respects direct sums if
\begin{align*}
F_{m_1+m_2}\begin{pmatrix}
z_1&0\\0&z_2
\end{pmatrix}
=\begin{pmatrix}
F_{m_1}(z_1)&0\\0&F_{m_2}(z_2)
\end{pmatrix}
\end{align*}
for all $m_1,m_2\in \NN$, $z_1\in M_{m_1}(\cB)$, $z_2\in M_{m_2}(\cB)$.

\item[(ii)] We say that $F$ respects similarities if
\begin{align*}
F_m(SzS^{-1})=SF_m(z)S^{-1}
\end{align*}
for all $m\in \NN$, $z\in M_m(\cB)$ and all invertible $S\in M_m(\CC)$.

\item[(iii)] We say that $F$ respects intertwininigs if for  all $n,m\in \NN$, $z_1\in M_n(\cB)$, $z_2\in M_m(\cB)$, $T\in M_{n,m}(\CC)$ (the latter are the $n\times m$ matrices with complex entries) we have the following:
\begin{align*}
z_1T=Tz_2 \implies F_n(z_1)T=TF_m(z_2).
\end{align*}

\end{enumerate}
Prove that [(i) and (ii)] is equivalent to (iii).
\end{exercise}

\section{Assignment 3}

\begin{exercise}\label{exercise:5}
Prove the second item from the proof of Lemma \ref{lem:partial-linear}: Let $f$ be a non-commutative function, then we have for $z_1\in M_n(\cB)$, $z_2\in M_m(\cB)$ that 
$$\partial f(z_1,z_2)\sharp (w_1+w_2)=\partial f(z_1,z_2)\sharp w_1 +
\partial f(z_1,z_2)\sharp w_2$$
for all $w_1,w_2\in M_{n,m}(\cB)$.

\end{exercise}

\begin{exercise}\label{exercise:6}
Let $r\in \NN$ and $b_0,b_1,\dots,b_{r+1}\in\cB$ be given and consider the monomial $f$
$$f(z)=b_0zb_1zb_2z\cdots b_rzb_{r+1}.$$
\begin{enumerate}

\item[(i)] Show that  $f=(f_m)_{m\in \NN}$ is a non-commutative function. (For this, also give first the precise definition of all $f_m:M_m(\cB)\to M_m(\cB)$.)

\item[(ii)] 
Calculate the first and second order derivatives of $f$, i.e.,
$$\partial f(z_1,z_2)\sharp w,\qquad\text{and}\qquad
\partial^{2} f(z_1,z_2,z_3)\sharp (w_1,w_2).$$

\end{enumerate}
\end{exercise}

\begin{exercise}\label{exercise:7}
For a non-commutative function $f$ we define the mappings 
$$\partial^{k-1}(z_1,\dots,z_k)\sharp (w_1,\dots, w_{k-1})$$ by
\begin{multline*}
f
\begin{pmatrix}
z_1& w_1 &0&\dots& 0\\
0&z_2& w_2 &\dots&0\\
\vdots&\vdots&\ddots&\ddots&\vdots\\
0& 0&\hdots & z_{k-1}& w_{k-1}\\
0& 0&\hdots&0&z_k
\end{pmatrix}
\qquad\\
\\=
\begin{pmatrix}
f(z_1) &\partial f(z_1,z_2)\sharp w_1&\partial^{2}(z_1,z_2,z_3)\sharp(w_1,w_2)
&\hdots&\partial^{k-1} f(z_1,\dots,z_k)\sharp (w_1,\dots,w_{k-1})\\
0& f(z_2)& \partial f(z_2,z_3)\sharp w_2&\hdots& \partial^{k-2} f(z_2,\dots,z_k)\sharp (w_2,\dots,w_{k-1}) \\
\vdots&\vdots &\vdots&\vdots&\vdots\\
0&0&0&\vdots& \partial f(z_{k-1},z_k)\sharp w_{k-1}\\
0&0&0&\hdots &f(z_k)
\end{pmatrix}
\end{multline*}
Show that for each $N\in\NN$ we have the expansion
$$f(z+tw)=
\sum_{k=0}^N t^k\partial^{k}(z,\dots,z,z)\sharp(w,\dots,w)
+t^{N+1}\partial^{N+1}f(z,\dots,z,z+tw)\sharp (w,\dots,w)$$
for $m\in\NN$, $z,w\in M_m(\cB)$ and $t\in\CC$. \newline
You can assume for this that 
$\partial^{k-1}(z_1,\dots,z_k)\sharp (w_1,\dots, w_{k-1})$ is linear in the arguments $w_i$.

Hint: It might be helpful, to consider the matrix
$$y:=\begin{pmatrix}
z&tw&0&\hdots&0&0\\
0&z&tw&\hdots&0&0\\
\vdots&\vdots&\vdots&\vdots&\vdots&\vdots \\
0&0&0&\hdots & z& tw\\
0&0&0&\hdots &0&z+tw
\end{pmatrix}
$$
and observe that 
$$y \cdot
\begin{pmatrix}
1\\
\vdots\\
1
\end{pmatrix}=
\begin{pmatrix}
1\\
\vdots\\
1
\end{pmatrix}\cdot 
\begin{pmatrix}
z+tw
\end{pmatrix}
$$
\end{exercise}

\begin{exercise}

Consider the $C^*$-algebra $M_n(\CC)$ of $n\times n$ matrices over $\CC$. We define its upper half-plane by
$$\mathbb H^+(M_n(\CC)):=\{b\in M_n(\CC)\mid \exists \varepsilon>0: \Im (b)\geq \varepsilon 1\},$$ 
where $\Im(b):=(b-b^*)/(2i)$.
\begin{itemize}
\item[(i)] In the case $n=2$, show that in fact
$$\mathbb H^+(M_2(\CC)):=\left\{
\begin{pmatrix} b_{11} &b_{12}\\ b_{21}&b_{22} \end{pmatrix}\right\vert \left.
\Im(b_{11})>0, \Im (b_{11})\Im(b_{22})> \frac 14 \vert b_{12}-\overline{b_{21}}
\vert^2\right\}.$$

\item[(ii)] For general $n\in\NN$, prove: if a matrix $b\in M_n(\CC)$ belongs to
$\mathbb H^+(M_n(\CC))$ then all eigenvalues of $b$ lie in the complex upper half-plane $\mathbb H^+(\CC)$. Is the converse also true?
\end{itemize}
\end{exercise}

\section{Assignment 4}

Let $\cA$ and $\cB$ be unital $C^*$-algebras.
A linear map $\Phi:\cA\to\cB$ is called \emph{completely positive} if all matrix amplifications $\Phi\otimes \id: M_n(\cA)\to M_n(\cB)$ are positive.

\begin{exercise}\label{exercise:9}
Show that the following are equivalent:
\begin{itemize}
\item[(i)]
$ \Phi:\cA\to\cB$ is completely positive.
\item[(ii)]
For each $n\in \N$ and all $a_1,\dots,a_n\in\cA$ the matrix
$(\Phi(a_ia_j^*))_{i,j=1}^n \in M_n(\cB)$ is positive.
\end{itemize}
\end{exercise}

\begin{exercise}
Show that the transpose map on $2\times 2$ matrices,
$$\Phi: M_2(\CC)\to M_2(\CC), \qquad\begin{pmatrix} a_{11}&a_{12}\\a_{21}&a_{22}
 \end{pmatrix}\mapsto \begin{pmatrix} a_{11}&a_{21}\\a_{12}&a_{22}
 \end{pmatrix},
$$
is positive, but not completely positive.
\end{exercise}

\begin{exercise}
Show that a positive conditional expectation $E:\cA\to\cB$ is completely positive. What does this tell us about the complete positivity of states $\ff:\cA\to\CC$?

Hint: For this you can use the following characterization: A matrix $(b_{ij})_{i,j=1}^n \in M_n(\cB)$ is positive if and only if we have
$$\sum_{i,j=1}^n b_ib_{ij}b_j^*\geq 0\qquad\text{for all $b_1,\dots,b_n\in\cB$.}$$
\end{exercise}

\begin{exercise}
\begin{itemize}
\item[(i)] Let $(\cA,\cB,E)$ be a $\cB$-valued $C^*$-probability space. Consider a ``constant'' selfadjoint random variable $b=b^*\in\cB\subset \cA$. Calculate the fully matricial Cauchy transform of $b$.
\item[(ii)] Consider a $C^*$-probability space $(\cA,\ff)$ as a special case of an operator-valued $C^*$-probability space, where $\cB=\CC$. Consider a selfadjoint $X=X^*\in \cA$. Its distribution $\mu_X$ is then a probability measure on $\R$. Express the fully matricial $\CC$-valued Cauchy transform $G_X$ in terms of $\mu_X$.

\item[(iii)] Assume that $X_1$ and $X_2$ are classical (thus commuting) bounded selfadjoint random variables. Hence they have a classical distribution, which is a probability measure on $\R^2$ with compact support. Consider now the $2\times 2$ matrix
$$X=\begin{pmatrix} X_1& 0\\ 0& X_2
\end{pmatrix}.$$
The $M_2(\CC)$-valued Cauchy transform of $X$, as a fully matricial function, should now be determined in terms of this classical data. Make this concrete!

\end{itemize}
\end{exercise}

\section{Assignment 5}

\begin{exercise}\label{exercise:13}
Show the following easy direction of Theorem \ref{thm:4.9}: Let $(\cA,\cB,E)$ be an operator-valued $C^*$-probability space and $X=X^*\in \cA$. Show that
$\mu_X\in \Sigma_\cB^0$.

\end{exercise}

\begin{exercise}
Let $\Phi:\cA\to\cB$ be a completely positive map between two unital $C^*$-algebras with $\Phi(1)=1$. Show that $\Phi$ satisfies the following kind of Cauchy-Schwarz inequality: for $a\in\cA$ we have $\Phi(a)^*\Phi(a)\leq \Phi(a^*a)$.

Hint: Consider the positive matrix
$$\begin{pmatrix}
a^*a& a^*\\
a& 1
\end{pmatrix}
$$

\end{exercise}

\begin{exercise}
Let $X$ and $Y$ be free in an operator-valued probability space $(\cA,\cB,E)$. Calculate the mixed moment
$E[Xb_1Yb_2Xb_3Y]$, for $b_1,b_2,b_3\in\cB$,
in terms of moments of $X$ and of $Y$.

\end{exercise}

\section{Assignment 6}

\begin{exercise}\label{exercise:16}
Let $\eta:\cB\to\cB$ be a completely positive map on the unital $C^*$-algebra $\cB$. We want to construct an operator $X$ which has $\eta$ as its second moment; this will be a kind of operator-valued Bernoulli element. For this we consider the degenerate Fock space
$$\F:=\cB\oplus \cB x\cB \subset \cB\la x\ra,$$
equipped with the $\cB$-valued inner product
$$\la\cdot,\cdot\ra:\F\times\F\to \cB,$$
given by linear extension of
$$\la b_0+b_1xb_2,\tilde b_0+\tilde b_1x\tilde b_2\ra:= b_0^*\tilde b_0+
b_2^*\eta(b_1^*\tilde b_1)\tilde b_2.$$
On $\F$ we define the creation operator $l^*$ by
$$l^*b=xb\qquad l^*b_1xb_2=0,$$
and the annihilation operator $l$ by
$$l b=0,\qquad l b_1xb_2=\eta(b_1)b_2.$$
Let $\cA$ be the $*$-algebra which is generated by $l$ and by elements $b\in\cB$ acting as multiplication operators on $\F$. We also put
$$E:\cA\to\cB, \qquad A\mapsto E[A]:= \la 1,A1\ra.$$
\begin{itemize}
\item[(i)] 
Show that the inner product is positive and that $l$ and $l^*$ are adjoints of each other.
\item[(ii)] 
Show that $E$ is positive.
\item[(iii)] 
Show that the second moment of the selfadjoint operator $X=l+l^*$ is given by $\eta$.
\item[(iv)] 
What is the formula for a general moment of $X$?
\end{itemize}

\end{exercise}

\begin{exercise}\label{exercise:17}
Let $S\in\cA$ be a $\cB$-valued semicircular element with covariance $\eta:\cB\to\cB$. Fix $n\in\N$ and $b\in M_n(\cB)$. Consider now
$$\hat S:=b (1\otimes S) b^*=b 
\begin{pmatrix}
S&\dots&0\\
\vdots&\ddots&\vdots\\
0&\dots&S
\end{pmatrix} b^*
\in M_n(\cA).$$
Show that $\hat S$ is an $M_n(\cB)$-valued semicircular element and calculate its covariance $$\hat \eta:M_n(\cB)\to M_n(\cB).$$
Compare also Remark \ref{rem:6.3}.

\end{exercise}

\begin{exercise}
Assume that we have $X^{(1)}_i$ ($i\in\N$) which are f.i.d., with first moment zero and second moment given by a covariance $\eta_1:\cB\to\cB$; and that we have $X^{(2)}_i$ ($i\in\N$) which are f.i.d with first moment zero and second moment given by a covariance $\eta_2:\cB\to\cB$. According to the operator-valued version of the free central limit theorem we know then that the normalized sum of the $X^{(1)}_i$ converges to an operator-valued semicircular element $S_1$ with covariance $\eta_1$ and that the normalized sum of the $X^{(2)}_i$ converges to an operator-valued semicircular element $S_2$ with covariance $\eta_2$. 

Assume now that the $X^{(1)}_i$ and $X^{(2)}_i$ are realized in the same $C^*$-probability space and are also free for each $i$. Then the joint distribution of
$(X^{(1)}_i, X^{(2)}_i)$ converges to the joint distribution of the pair $(S_1,S_2)$. Convince yourself that our argument (from the \href{https://rolandspeicher.files.wordpress.com/2019/08/free-probability.pdf}{Free Probability Lecture Notes}, Assignment 3, Exercise 4) for the scalar-valued case that freeness goes over to the limit remains valid in the operator-valued case. Thus we get in the limit two semicircular elements which are free.

By repeating the calculation in our proof of the central limit theorem, Theorem \ref{thm:6.2}, for this multivariate setting derive the formula for mixed moments of two free semicircular elements $S_1$ and $S_2$, with covariance mappings $\eta_1$ and $\eta_2$, respectively.

\end{exercise}

\section{Assignment 7}

\begin{exercise}\label{exercise:19}
Let $\eta:\cB\to\cB$ be a completely positive map on the $C^*$-algebra $\cB$. We want to construct a semicircular operator $X$ which has $\nu_\eta$ as its distribution. This operator will be constructed on an operator-valued version of the full Fock space. The latter is nothing but our polynomials
$\cB \la x\ra$, equipped with the $\cB$-valued inner product
$$\la b_0xb_1x\cdots b_nxb_{n+1}, \tilde b_0x\tilde b_1x\cdots\tilde b_mx\tilde b_{m+1}\ra:=\delta_{nm} b_{n+1}^* \eta\left( b_n^*\cdots \eta\bigl(b_1^*\eta(b_0^*\tilde b_0) \tilde b_1\bigr)\cdots \tilde b_n\right) \tilde b_{n+1}.
$$
On this full fock space $\F$ we define again a creation operator $l^*$, now given by
$$l^*b_0xb_1\cdots x b_{n+1}:= xb_0xb_1\cdots x b_{n+1},$$
and an annihilation operator $l^*$, given by $l b=0$ and 
$$lb_0xb_1 x\cdots x b_{n+1} := (\eta(b_0)b_1)x\cdots x b_{n+1}.$$
Elements from $\cB$ act on $\F$ by left multiplication. For $\cA$ we take now the $*$-algebra which is generated by $l$ and by all multiplication operators from $\cB$. Furthermore, we put
$$E:\cA\to\cB, \qquad A\mapsto E[A]:=\la 1, A1\ra.$$
\begin{itemize}
\item[(i)] Show that the inner product on $\F$ is positive and that $l$ and $l^*$ are adjoints of each other.
\item[(ii)] Show that $E$ is positive.
\item[(iii)] Calculate explicitly the second and the fourth moments of $X:=l+l^*$.
\item[(iv)] Prove that $X=l+l^*$ has semicircular distribution $\nu_\eta$.
\end{itemize}
\end{exercise}

\begin{exercise}\label{exercise:20}
Let $S_1$ and $S_2$ be two free (scalar-valued) standard semicircular elements and consider 
$$S:=\begin{pmatrix}
0&S_1\\
S_1&S_2
\end{pmatrix}.$$
We have seen in item \ref{rem:6.3.3} of Remark \ref{rem:6.3} that $S$ is then an $M_2(\C)$-valued semicircular element whose covariance function $\eta:M_2(\C)\to M_2(\C)$ is given by
$$\eta \begin{pmatrix}
b_{11}& b_{12}\\
b_{21}& b_{22}
\end{pmatrix}=
\begin{pmatrix}
b_{22}& b_{21}\\
b_{12}& b_{11}+b_{22}
\end{pmatrix}.
$$
Refresh your memory on the relation between free semicircular elements and independent $\GUE$ random matrices (for example, from Chapter 6 of the \href{https://rolandspeicher.files.wordpress.com/2019/08/free-probability.pdf}{Free Probability Lecture Notes}). From this it follows that $S$ is the limit of a random matrix
$$X_N=\begin{pmatrix}
0 & A_N\\
A_N & B_N
\end{pmatrix},$$
where $A_N$ and $B_N$ are independent \GUE\ random matrices.
(If $A_N$ and $B_N$ are $N\times N$ matrices, then $X_N$ is of course a $2N\times 2N$ matrix.)
Since
$$g(z)=\tr E[(z-S)^{-1}]=\tr G(z)$$
is the scalar-valued Cauchy transform of $S$ with respect to 
$\tr\circ E$ ($\tr$ is here the normalized trace over $2\times 2$ matrices), we can calculate the Cauchy transform $g(z)$ of the limiting eigenvalue distribution of $X_N$ by first calculating the $M_2(\CC)$-valued Cauchy transform $G(z)$ of $S$ and then taking the trace of this. For invoking the Cauchy-Stieltjes inversion formula, we should calculate this for $z$ close to the real axis.
\begin{itemize}
\item[(i)] We know that the operator-valued Cauchy transform (on the ground level) $G(b)$ satisfies  the matrix equation
$$b G(b) =1 +\eta (G(b)) G(b).$$
This is true for all $b\in M_2(\CC)$, but we are here only interested in arguments of the form $b=z 1$, where $z\in H^+(\CC)$.
Try to solve this equation (exactly or numerically) for $z\in H^+(\CC)$ close to the real axis, so that you can produce from this a density for the scalar-valued distribution of $S$.
\item[(ii)] Realize for large $N$ the random matrix $X_N$ and calculate histograms for its eigenvalue distribution. Compare this with the result from (i).
\end{itemize}

\end{exercise}

\section{Assignment 8}

\begin{exercise}\label{exercise:21}
Prove Proposition \ref{prop:7.2}: Let $(\cA,\ff)$ be a $C^*$-probability space and $S_1,\dots,S_d\in\cA$ free standard semicirculars (i.e., $\ff(S_i^2)=1$). For $n\geq1$ and selfadjoint matrices $b_1,\dots,b_d\in M_n(\CC)$ we consider
$$S:=b_1\otimes S_1+\cdots + b_d\otimes S_d \in M_n(\cA).$$
Then $S$ is in the matrix-valued $C^*$-probability space $(M_n(\cA),M_n(\CC),\id\otimes\ff)$ a matrix-valued semicircular element with covariance $\eta:M_n(\CC)\to M_n(\CC)$ given by
$$\eta(b)=\sum_{j=1}^d b_jbb_j.$$
\end{exercise}

\begin{exercise}\label{exercise:22}
Let $S_{ij}$ for $i\geq j$ be free standard semicircular elements, and put 
$S_{ij}=S_{ji}$. Furthermore, let $\alpha_{ij}\in\R$ with $\alpha_{ij}=\alpha_{ji}$ be given. Then we consider
$$S:=(\alpha_{ij} S_{ij})_{i,j=1}^n.$$
From the previous exercise we know that this is an $M_n(\CC)$-valued semicircular element. Give, by relying on Theorem \ref{thm:7.5}, a criterium to decide whether $S$ is also a scalar-valued semicircular element. Use this to decide whether the following are scalar-valued semicircular elements (for $S_1,\dots, S_6$ free standard semicirculars):
$$S=\begin{pmatrix}
3 S_1 &0 &4 S_2\\
0& 5 S_3 &0\\
4 S_2&0&3 S_4
\end{pmatrix}
\qquad\text{or}\qquad
\tilde S=
\begin{pmatrix}
3 S_1 &6 S_5 &4 S_2\\
6 S_5& 5 S_3 & 6S_6\\
4 S_2&6 S_6&3 S_4
\end{pmatrix}.
$$

\end{exercise}

\begin{exercise}
Check your conclusion from the last exercise numerically by producing histograms, for $N=1000$ or higher, of the eigenvalues of the matrices
$$X^{(3N)}=\begin{pmatrix}
3 X^{(N)}_1 &0 &4 X^{(N)}_2\\
0& 5 X^{(N)}_3 &0\\
4 X^{(N)}_2&0&3 X^{(N)}_4
\end{pmatrix}
\qquad\text{or}\qquad
\tilde X^{(3N)}=
\begin{pmatrix}
3 X^{(N)}_1 &6 X^{(N)}_5 &4 X^{(N)}_2\\
6 X^{(N)}_5& 5 X^{(N)}_3 & 6X^{(N)}_6\\
4 X^{(N)}_2&6 X^{(N)}_6&3 X^{(N)}_4
\end{pmatrix},
$$
where $X^{(N)}_1,\dots,X^{(N)}_6$ are independent \GUE(N) random matrices.
\end{exercise}

\section{Assignment 9}

\begin{exercise}\label{exercise:24}
Prove the recursion between moments and free cumulants from Proposition \ref{prop:9.7}, by checking that the arguments from the scalar-valued case work also in the operator-valued situation.
\end{exercise}

\begin{exercise}
Write down explicitly the linearization for a monomial of degree $k=5$, as given in the proof of Theorem 8.5 and check that this satisfies indeed all the requirements for a linearization.
\end{exercise}

\begin{exercise}
Find a linearization $\hat p$ of the polynomial
$$p(x,y)=xy^2+y^2x-y.$$
\end{exercise}

\textbf{Bonus Questions:}

\begin{exercise}\label{exercise:27}
Calculate, via linearization and numerical calculation of the corresponding operator-valued semicircular or of the corresponding operator-valued free convolution, the distribution of $p(X,Y)=XY^2+Y^2X-Y$, where 
\begin{itemize}
\item 
$X$ and $Y$ are free standard semicircular elements
\item
$X$ and $Y$ are free random variables, with 
$$\mu_X=\frac 12 (\delta_{0}+\delta_1),\qquad
\mu_Y=\frac 12 (\delta_{-1}+\delta_1).$$
\end{itemize}
\end{exercise}

\begin{exercise}
Realize $X$ and $Y$, as given in Exercise \ref{exercise:27}, (asymptotically) via large $N\times N$ random matrices $X_N$ and $Y_N$, and produce histograms of the eigenvalue distribution of $p(X_N,Y_N)$. Compare the results with the calculations from Exercise \ref{exercise:27}.
\end{exercise}

\chapter*{Some Off-the-Record Remarks} 
\addcontentsline{toc}{chapter}{Some Off-the-Record Remarks}

The preceding presentation has hopefully convinced the reader that we have developed powerful analytic tools to deal with non-commutative distributions and that we have reached a deep understanding of many facets of this non-commutative world.



Let us reconsider what we have achieved so far. We have different ways to describe non-commutative distributions, namely by
\begin{itemize}
\item
presenting concrete operators on Hilbert spaces
\item
by describing the joint moments or the joint cumulants of the operators
\item
by giving the Cauchy transform of the distributions
\item
or by describing the classical distribution of all polynomials (or maybe even
all rational functions) in the operators
\end{itemize}
We consider a situation nice and well-understood when we have something to say about all of those ways and usually progress comes from being able to switch between the different points of view. In particular, we should be able to get our hands on the Cauchy transform and distributions of polynomials.

In the case of free variables, so in particular for free semicirculars, we are in such a nice situation.

Also if we move away from free variables many of our tools still apply and lead to quite non-trivial statements.
In particular, the statements about the abscence of atoms in polynomials for operators which have maximal $\Delta$ are of this type and in the continuation of such investigations we have many more qualitative results on regularity properties of polynomials in such variables; like, for example, in the recent work \cite{BM} on H\"older continuity of the distribution function of such polynomials.

In this chapter we want to point out that there are of course also situations where the situation is not so satisfactory, and that we still hope for many more
exciting discoveries in the non-commutative territory.

\section*{Is there anything special about distributions of generators of non-embeddable von Neumann algebras}
\addcontentsline{toc}{section}{Is there anything special about distributions of generators of non-embeddable von Neumann algebras}
By the refutation \cite {JNVWY} of Connes embedding problem we know now that there are tracial von Neumann algebras which cannot be embedded into the ultrapower of the hyperfinite factor - or to put it more in our language: there are operators in a tracial $W^*$-probability space, whose joint moments cannot be approximated well by moments of matrices. Up to now nobody was able to construct explicit examples of such objects. Can our theory of non-commutative distributions say anything about the distribution of such operators? Can we address them by any of the above mentioned ways to deal with non-commutative distributions? Let us have a look.
\begin{itemize}
\item
We do not know any concrete operators - that's of course what we would like to find!
\item
Neither do we know any candidates for joint moments or joint cumulants. Since positivity is always an issue here, it is not only the problem of coming up with moments which are unreachable for matrices, but one also needs arguments guaranteeing that those are really moments of selfadjoint operators.
\item
Again, we would have to come up with a Cauchy transform which is not reachable by Cauchy transforms of matrices (and which is indeed a Cauchy transform, so satisfies Theorem \ref{thm:Williams}). It's hard to imagine how to get one without writing it down concretely, or maybe at least writing down an equation for it. Actually, a ``random'' Cauchy transform might do the job, but it is not clear how to make this rigorous.
\item
This is even more unclear; without having knowledge about the Cauchy transform it seems quite unlikely to get a grasp on other functions of the variables.

\end{itemize}

So, for the moment, there is nothing we have to offer from our non-commutative disribution perspective and we can only hope for some more
insights.

\section*{The $q$-Gaussian operators}
\addcontentsline{toc}{section}{The $q$-Gaussian operators}

Since we had no place to start in the preceding case it is not surprising that we could not say anything. So one might still have the hope that given some concrete operators, of which we have at least some knowldege, we should have good chances of saying something about its Cauchy transform and then mabye also the distribution of polynomials in them. Here comes an example which shows that even then the situation is not so promising. This is a deformation of the situation of free semicirculars, but as they have no freeness in them we have problems with getting a grasp on their Cauchy transform.

The $q$-Gaussian distribution, also known as $q$-semicircular distribution, was 
introduced in \cite{Mare-1991,BKS} in the context of non commutative probability. 
Let us review some basic definitions. In the following $q\in[-1,1]$ is fixed.
Consider a Hilbert space $\HH$. 
The following is a $q$-deformation of the contructions from Example \ref{ex:1.7}.
On the algebraic full Fock space
$\bigoplus_{n\geq 0}\HH^{\otimes n}$
-- where $\HH^0=\mathbb{C}\Omega$ with a norm one vector $\Omega$, called ``vacuum'' 
-- we define a $q$-deformed inner product as follows: 
 \begin{equation*}
   \langle h_1\otimes\cdots\otimes h_n,g_1\otimes\cdots\otimes g_m\rangle_q 
       = \delta_{nm}
           \sum_{\sigma\in S_n}
               \prod^n_{r=1}
                   \langle h_r,g_{\sigma(r)}\rangle q^{i(\sigma)},
 \end{equation*}
where 
$$i(\sigma)=\#\{(k,l)\mid 1\leq k<l\leq n; \sigma(k)>\sigma(l)\}$$
is the number of inversions of a permutation $\sigma\in S_n$.
In \cite{Mare-1991} it was shown that this inner product is positive definite, and has a kernel only for $q=1$ and $q=-1$.

 The $q$-Fock space is then defined as the completion of the algebraic full Fock space 
 with respect to this inner product
\begin{equation*}
 \mathcal{F}_q(\mathcal{H})
    =\overline{\bigoplus_{n\geq 0}\mathcal{H}^{\otimes n}}^{\langle\cdot,\cdot\rangle_q}.
\end{equation*}
In the cases $q=1$ and $q=-1$ we have to first divide out the kernel, thus leading to the symmetric and anti-symmetric Fock space, respectively.

Now for $h\in\mathcal{H}$ we define the $q$-creation operator $a^*(h)$, given by
 \begin{align*}
 a^*(h)\Omega&=h,\\
  a^*(h)h_1\otimes\cdots\otimes h_n&=h\otimes h_1\otimes\cdots\otimes h_n.
 \end{align*}
 Its adjoint (with respect to the $q$-inner product), the $q$-annihilation operator 
 $a(h)$, is given by
  \begin{align*}
     a(h)\Omega&=0,\\
     a(h)h_1\otimes\cdots\otimes h_n&=
        \sum_{r=1}^n 
           q^{r-1} 
             \langle h,h_r\rangle 
                h_1\otimes \cdots \otimes h_{r-1}\otimes h_{r+1}\otimes\cdots \otimes h_n.
 \end{align*}

[Never mind that we have switched here the convention whether the creation or the annihilation operator gets the $*$. There are two conflicting traditions, one
from physics, where creation goes with the $*$, and one from operator theory where, in the case $q=0$, the left shift $l$, and not its adjoint $l^*$, is the basic isometry. Since we are now more on the physics side, our inner product has also become linear in its second argument.]

Those operators satisfy the $q$-commutation relations
$$a(f)a^*(g)-q a^*(g)a(f)=\langle f,g\rangle \cdot 1\qquad (f,g\in \mathcal{H}).$$
For $q=1$, $q=0$, and $q=-1$ this reduces to the CCR-relations, the Cuntz relations, 
and the CAR-relations, respectively. With the exception of the case $q=1$, the operators 
$a^*(f)$ are bounded.

Let $\xi_1,\dots,\xi_n$ be an orthonormal system of vectors in $\HH$, then we consider the selfadjoint operators 
$X_i:=a(\xi_i)+a^*(\xi_i)$ ($i=1,\dots,n$).
For $\ff$ we take again the vacuum expectation state
$\ff(A):=\lb \Omega,A\Omega\rb$.
We are now interested in the non-commutative distribution $\mu_{X_1,\dots,X_n}$ of the operators $X_1,\dots,X_n$ in the $C^*$-probability space $(B(\cF_q(\HH)),\ff)$. We call this the \emph{(multivariate) $q$-Gaussian distribution}. For $q=0$ it reduces to the non-commutative distribution of $n$ free semicirculars. The $q$-deformation has still some of the features of the $q=0$ case. First of all, by definition we have nice and concrete operators with this non-commutative distribution. Also the formula for mixed moments in free semicirculars survives the deformation and one has the following $q$-deformed Wick formula: for any 
$\ee:\{1,\dots,k\}\rightarrow\{1,\dots,n\}$ we have
        \begin{equation*}
         \ff(X_{\ee(1)}\cdots X_{\ee(k)})
             =\sum_{\stackrel{\pi\in\cP_2(k)}{\pi\leq \ker\ee}}
                q^{cr(\pi)},
        \end{equation*}
where $cr(\pi)$ denotes the number of crossings of the pair-partition $\pi$, i.e., the number of pairs of blocks which have a crossing.

So, this looks quite good: we have a nice realization of the $q$-Gaussian distribution by very concrete operators and we have nice combinatorial formulas for all joint moments. But does this mean that we understand this non-commutative distribution well? Unfortunately, not really. In particular, we do not get a hold on its operator-valued Cauchy transform.

Following our general strategy of going over from tuples of non-commuting operators to one operator-valued operator we put our operators $X_1,\dots,X_n$ on the diagonal of an $n\times n$ matrix
\begin{equation}
X=\begin{pmatrix}\label{eq:matrix-X}
X_1& 0&\dots&0\\
0&X_2&\dots&0\\
\vdots&\vdots&\ddots&\vdots\\
0&0&\hdots&X_n
\end{pmatrix}.
\end{equation}
Understanding the distribution of $(X_1,\dots,X_n)$ is now the same as understanding the $\cB$-valued distribution of $X$, where we have
put $\cB:=M_n(\CC)$; the matrix $X$ is what we would call an operator-valued $q$-semicircular element. In order to deal with this we should understand the $\cB$-valued Cauchy transform $G_X=(G_X^{(k)})_{k\in\NN}$. For a full understanding we need its structure as a fully matricial function with all its matrix amplifications, but for many applications even the knowledge just on the base level would be very helpful.
But here we are stuck. We do not have any nice concrete analytic description of this Cauchy transform. 

From the situation for $q=0$, the case of free semicirculars, one might have
got the impression that the one-dimensional and the multi-variate case are not so different after all. In that case the quadratic equation for the Cauchy transform of the scalar-valued semicircular distribution was replaced by a corresponding operator-valued quadratic equation. The latter was of course harder than its scalar-valued counterpart, but we could still deal with it. This might give the impression that also in the case of general $q$ we should be able to extend  results from the $n=1$ case to  the general operator-valued situation. This is, unfortunately, not the case. We understand the $n=1$ case for all $q$ quite well, but all the nice structure there does not extend into the operator-valued regime.

For $n=1$, the $q$-Gaussian distribution is a probability measure on the interval 
$[-2/\sqrt{1-q}, 2/\sqrt{1-q}]$, with analytic formulas for its density, see Theorem 
1.10 in \cite{BKS}. For its Cauchy transform $G$ we do not have an algebraic equation, but we know 
a good continued fraction expansion of the form
$$G(z)=\cfrac 1{z-\cfrac 1{z- \cfrac {1+q}{ z-\cfrac {1+q+q^2}{z-\dots}}}}.$$
The naive guess that one might also have a corresponding operator-valued version of such a continued fraction expansion is unfortunately not true. Whereas in the scalar case any probability measure has a continued fraction expansion for its Cauchy transform, this does not hold any more in the operator-valued setting (see \cite{AW}), and it is easy to check that the matrix $X$ in \eqref{eq:matrix-X} for the $q$-Gaussian distribution is one of the basic examples where this fails. 

So in a sense, at the moment our machinery for operator-valued Cauchy-transforms has unfortunately nothing to offer for dealing with $q$-Gaussian distributions.
Of course, Cauchy transforms are not everything and we have also other approaches and tools to understand non-commutative distributions. In particular, there has been quite some progress \cite{GSh,Jek2} in our understanding of the $q$-Gaussian distributions, by describing them as free Gibbs states and using non-commutative versions of transport to relate different $q$'s. Combined with \cite{BM} this gives then also regularity properties of polynomials in $q$-Gaussian operators. What is missing, compared to the free case, is a way to calculate the distribution of polynomials in $q$-Gaussians. 

But this would be the content of another lecture series ...



\end{document}